\documentclass[12pt,reqno]{amsart}  

\usepackage{amsmath,amsthm,amssymb,amscd,mathdots,mathtools,enumerate,float,shuffle,stmaryrd,ytableau,graphicx,dsfont}
\usepackage{fancybox}

\usepackage{arydshln}
\usepackage{bbm}
\usepackage{scalerel}
\usepackage{aliascnt}
\usepackage{hyperref}

\definecolor{mylinkcolor}{RGB}{166, 77, 121} 
\definecolor{mycitecolor}{RGB}{17, 85, 204}  
\hypersetup{
    colorlinks=true,
    linkcolor=mylinkcolor,
    citecolor=mycitecolor,
    urlcolor=mycitecolor,
    bookmarksnumbered=true,
    bookmarksopen=true,
}
\usepackage{graphicx}
\usepackage{xparse}
\usepackage{tikz,tikz-cd}
\allowdisplaybreaks
\usetikzlibrary{decorations.pathmorphing,shapes,calc}
\usetikzlibrary{decorations.pathreplacing}

\usepackage{xcolor}
\usepackage[margin=1in]{geometry}

\newtheorem{theorem}{Theorem}[section]
\newtheorem*{theorem*}{Theorem}

\newaliascnt{proposition}{theorem}
\newtheorem{proposition}[proposition]{Proposition}
\aliascntresetthe{proposition}

\newaliascnt{lemma}{theorem}
\newtheorem{lemma}[lemma]{Lemma}
\aliascntresetthe{lemma}

\newaliascnt{corollary}{theorem}
\newtheorem{corollary}[corollary]{Corollary}
\aliascntresetthe{corollary}

\newaliascnt{conjecture}{theorem}

\aliascntresetthe{conjecture}

\newaliascnt{question}{theorem}

\aliascntresetthe{question}

\numberwithin{equation}{section}

\theoremstyle{definition}
\newaliascnt{definition}{theorem}
\newtheorem{definition}[definition]{Definition}
\aliascntresetthe{definition}

\newaliascnt{example}{theorem}
\newtheorem{example}[example]{Example}
\aliascntresetthe{example}

\newaliascnt{remark}{theorem}
\newtheorem{remark}[remark]{Remark}
\aliascntresetthe{remark}

\newaliascnt{lemdef}{theorem}

\aliascntresetthe{lemdef}

\newaliascnt{propdef}{theorem}
\newtheorem{propdef}[propdef]{Definition/Proposition}
\aliascntresetthe{propdef}

\newtheorem{thmA}{Main Theorem}

\mathtoolsset{showonlyrefs} 

\newcommand{\h}{\mathfrak H}
\newcommand{\N}{\mathbb{N}}
\newcommand{\Q}{\mathbb{Q}}
\newcommand{\Z}{\mathbb{Z}}
\newcommand{\R}{\mathbb{R}}

\newcommand{\C}{\mathbb{C}}

\newcommand{\OH}{\mathcal{O}(\mathbb{H})}
\newcommand{\Hd}{\mathcal{H}_{\diamond}}

\newcommand{\A}{A}
\newcommand{\K}{\mathsf{k}}
\newcommand{\KA}{\K\langle \A\rangle}

\newcommand{\G}{\mathbb{G}}

\newcommand{\calP}{\mathcal{P}}
\newcommand{\X}{\mathcal{X}}
\newcommand{\YT}{\operatorname{YT}}
\newcommand{\SSYT}{\operatorname{SSYT}}
\newcommand{\Sh}{\operatorname{sh}}
\newcommand{\Y}{\K\boxed{\hspace{-1pt} A}}
\newcommand{\Yd}{\Y_{\diamond}}

\newcommand{\con}{\operatorname{SC}}

\newcommand{\Ydiag}{\K\boxed{A}^{\operatorname{diag}}}

\definecolor{mycolor}{RGB}{194, 8, 88}

\definecolor{mycolor}{RGB}{194, 8, 88}
\newcommand{\todo}[1]{\message{LaTeX Warning: You did not finish your work :-( on input line \the\inputlineno} {\color{mycolor} {\big[\,}{\bf Todo:} #1\,\big]}}

\makeatletter
\newcommand{\cop}{\DOTSB\cop@\slimits@}
\newcommand{\cop@}{\mathop{\bigstar}}
\newcommand{\cop@@}[2]{%
  \vphantom{\sum}%
  \ifx#1\displaystyle\big#2\else#2\fi
}
\makeatother

\title[The Young Tableaux Hopf algebra and multiple Schur series]{The Young Tableaux Hopf algebra \\and multiple Schur series}

\author{Jinbo Yu}
\address{Graduate School of Mathematics,  Nagoya University, Nagoya, Japan.}
\email{jinbo.yu.e6@math.nagoya-u.ac.jp}


\usepackage[nameinlink]{cleveref}
\crefname{theorem}{Theorem}{Theorems}
\crefname{proposition}{Proposition}{Propositions}
\crefname{lemma}{Lemma}{Lemmas}
\crefname{corollary}{Corollary}{Corollaries}
\crefname{conjecture}{Conjecture}{Conjectures}
\crefname{question}{Question}{Questions}
\crefname{definition}{Definition}{Definitions}
\crefname{example}{Example}{Examples}
\crefname{remark}{Remark}{Remarks}
\crefname{lemdef}{Definition/Lemma}{Definitions/Lemmas}
\crefname{propdef}{Definition/Proposition}{Definitions/Propositions}
\crefname{thmA}{Main Theorem}{Main Theorems}
\crefname{section}{Section}{Sections}
\Crefname{section}{Section}{Sections}
\crefname{subsection}{Section}{Sections}
\Crefname{subsection}{Section}{Sections}

\begin{document}
\date{\today}

\begin{abstract}
In this paper, we introduce multiple Schur series, which are defined by Schur-type sums over semi-standard Young tableaux and generalize both Schur multiple zeta values and multiple Eisenstein series. To study their algebraic structure, we construct a connected, commutative, graded Hopf algebra of Young tableaux and identify its linearized quotient with the quasi-shuffle algebra. Within this Hopf algebra and its quotient, we establish several relations, including a hook formula and the Jacobi--Trudi formula. Furthermore, we relate this Hopf algebra to the ring of symmetric functions, which yields polynomial reduction formulas for tableaux with constant entries. As applications, we recover Schur multiple zeta values, introduce Schur multiple Eisenstein series together with a $q$-analogue of Schur multiple zeta values, and discuss their (quasi)modularity.
\end{abstract}

\maketitle

\section{Introduction}

The purpose of this paper is to study the algebraic structure of Young tableaux and several of its realizations. To study these realizations, we introduce \emph{multiple Schur series}. Multiple Schur series generalize multiple zeta values and multiple Eisenstein series by replacing the index with a Young tableau and replacing the linear order on the summation variables by the semi-standard condition. A prominent example of this framework is given by the Schur multiple zeta values, defined and studied by M.~Nakasuji, O.~Phuksuwan, and Y.~Yamasaki \cite{NPY}. A common algebraic setup for multiple zeta values and multiple Eisenstein series \cite{GKZ,B} is the quasi-shuffle algebra introduced by M.~E.~Hoffman and K.~Ihara \cite{H,HI}.  However, for Schur multiple zeta values, no such framework has been proposed before.

In this paper, we construct a Hopf algebra structure on Young tableaux (\cref{mainthmA:1}) that incorporates and extends the properties of the quasi-shuffle algebra. We establish new relations in this Young tableaux Hopf algebra and extend several results of \cite{NPY}, including the Jacobi--Trudi formula. The key point is that the maps involved are Hopf algebra homomorphisms (\cref{Thm:Main-L-diamond-Hook,Thm:Main-L-diamond}). Next, we introduce multiple Schur series as a generalization of Schur multiple zeta values and as a realization of the Young tableaux Hopf algebra. Starting with a family of algebra homomorphisms indexed by a totally ordered set, we use convolution to construct an algebra homomorphism from the Young tableaux algebra to any commutative $\K$-algebra (\cref{Thm:Main-MSS}). The same construction gives a connection with the ring of symmetric functions. It also yields an explicit injective embedding of this ring into the Young tableaux algebra. As a main application, we introduce the Schur multiple Eisenstein series. They are Schur analogues of the multiple Eisenstein series of H.~Gangl, M.~Kaneko, and D.~Zagier \cite{GKZ,B}. We also prove their modularity properties. A detailed study is carried out in \cite{BY,Yu}.

\subsection{Schur multiple zeta values}\label{sec:SNZV}

For $r\geq1$ and $(h_1,\dots,h_r)\in \N^r$ with $h_r\geq2$, the multiple zeta values and multiple zeta-star values are defined by
\begin{equation}\label{eq:defMZV&MZSV}
    \zeta(h_1,\dots,h_r)=\sum_{0<n_1<\cdots<n_r}\frac{1}{n_1^{h_1}\cdots n_r^{h_r}},
    \qquad
    \zeta^{\star}(h_1,\dots,h_r)=\sum_{0<n_1\leq\cdots\leq n_r}\frac{1}{n_1^{h_1}\cdots n_r^{h_r}}.
\end{equation}
Both series converge absolutely. The two families determine each other by inclusion-exclusion. For example,
\begin{equation}\label{eq:zeta&zetas}
    \zeta^{\star}(h_1,h_2)=\zeta(h_1,h_2)+\zeta(h_1+h_2),
    \qquad
    \zeta(h_1,h_2)=\zeta^{\star}(h_1,h_2)-\zeta^{\star}(h_1+h_2).
\end{equation}
In general, each multiple zeta-star value is a $\Z$-linear combination of multiple zeta values, and conversely. The study of algebraic relations reduces to either family, and current research mainly centers around multiple zeta values.  K.~Ihara, M.~Kaneko, and D.~Zagier~\cite{IKZ} conjectured that the \emph{extended double shuffle relations} generate all algebraic relations among multiple zeta values. These relations originate from two distinct product structures defined on the same underlying space.

M.~E.~Hoffman~\cite{H} formalized this structure via \emph{quasi-shuffle algebras}. Fix a commutative associative product $\diamond$ on the $\K$-vector space $\K\A$ generated by an alphabet $\A$. The quasi-shuffle product $\ast_\diamond$ in the free algebra $\KA$ is determined by $\mathbf{1}\ast_\diamond w=w$ and
\begin{align*}
    aw \ast_\diamond bv
    = a(w \ast_\diamond bv) + b(aw \ast_\diamond v) + (a\diamond b)(w \ast_\diamond v).
\end{align*}
Different choices of $\diamond$ recover the harmonic and shuffle relations.

Take $\A=\{z_h\mid h\geq1\}$ and  set $z_{h_1}\diamond z_{h_2}=z_{h_1+h_2}$. This gives the harmonic product. The assignment $z_{h_1}\cdots z_{h_r}\mapsto\zeta^{\ast}(h_1,\dots,h_r)$ is an algebra homomorphism and encodes the harmonic relations. Here $\zeta^{\ast}$ is the harmonic regularization of $\zeta$ for $h_r=1$. M.~E.~Hoffman proved that $(\KA,\ast_\diamond,\Delta_{\mathrm{dec}})$ is a Hopf algebra, where $\Delta_{\mathrm{dec}}$ denotes the deconcatenation coproduct. M.~E.~Hoffman and K.~Ihara revisited this structure in~\cite{HI}. They obtained an explicit exponential identity that expresses each power of a single letter as a polynomial in its $\diamond$-powers. \cref{sec:qsa} reviews these constructions in detail.

M.~Nakasuji, O.~Phuksuwan, and Y.~Yamasaki~\cite{NPY} introduced \emph{Schur multiple zeta values}. The values generalize both families given in \eqref{eq:defMZV&MZSV}. The index $(h_1,\dots,h_r)$ is replaced by a skew Young tableau. The linear order on the summation variables is replaced by the semi-standard condition. For a skew Young tableau $\mathbf{h}=(h_{i,j})$ of shape $\lambda/\mu$, define
\begin{align*}
    \zeta(\mathbf{h})
    = \sum_{(m_{i,j})\,\in\,\SSYT(\lambda/\mu,\,\N)}
      \prod_{(i,j)\in D(\lambda/\mu)} \frac{1}{m_{i,j}^{h_{i,j}}},
\end{align*}
where the sum ranges over all semi-standard Young tableaux of shape $\lambda/\mu$ with entries in $\N$. The semi-standard condition imposes $m_{i,j}\leq m_{i,j+1}$ along rows and $m_{i,j}<m_{i+1,j}$ along columns. The series converges absolutely when $h_{i,j}\geq2$ at every corner of $\lambda/\mu$ and $h_{i,j}\geq1$ at all other boxes. A single-column tableau recovers $\zeta(h_1,\dots,h_r)$. A single-row tableau recovers $\zeta^\star(h_1,\dots,h_r)$. Schur multiple zeta values thus unify these two within one combinatorial object.

In \cite{NPY}, the \emph{Jacobi--Trudi formula} for Schur multiple zeta values was proved. A Young tableau $\mathbf{h}=(h_{i,j})$ is called \emph{diagonal constant} if $h_{i,j}=a_{j-i}$ depends only on the diagonal $j-i$. Here $a_n$ $(n\in\Z)$ denotes the common entry on the $n$-th diagonal. For a diagonal constant Young tableau $\mathbf{h}$ of shape $\lambda/\mu$, the formula reads
\begin{align*}
    \zeta(\mathbf{h})
    = \det\bigl[\zeta(a_{-\mu'_j+j-1},\dots,a_{-\mu'_j+j-(\lambda'_i-\mu'_j-i+j)})\bigr]_{1\leq i,j\leq \lambda_1},
\end{align*}
where $\lambda'=(\lambda_1',\dots,\lambda_{\lambda_1}')$ and $\mu'=(\mu_1',\dots,\mu_{\mu_1}')$ are the conjugates of the partitions $\lambda$ and $\mu$.
It expresses $\zeta(\mathbf{h})$ as a determinant of ordinary multiple zeta values and mirrors the classical Jacobi--Trudi identity for Schur polynomials. H.~Bachmann and S.~Charlton~\cite{BC} extended this circle of ideas to generalized determinant evaluations.

One goal of our algebraic framework is to provide a Hopf algebraic interpretation of the Jacobi--Trudi formulas. 

\subsection{Young tableaux Hopf algebra}\label{sec:introYTHA}

The central object of this paper is a Hopf algebra structure on Young tableaux. It plays the role for Schur multiple zeta values that the quasi-shuffle Hopf algebra plays for multiple zeta values.

Let $\A$ be a countable alphabet and $\K$ a commutative ring. By $\Y$ we denote the commutative $\K$-algebra generated by connected skew Young tableaux with entries in $\A$, modulo the relation that identifies a disconnected tableau with the product of its connected components. We will make this definition precise in \eqref{eq:defKA}. Further, we allow the Young tableaux to have entries in $\K\A$ and then extend them $\K$-linearly in each entry. For example, for $a,b \in \A$, $\alpha,\beta \in \K$, and $X=\alpha a+ \beta b$, we write
\begin{align}
    \begin{ytableau}
        X &c\\
        d
    \end{ytableau} =   \alpha\,\,   \begin{ytableau}
        a&c\\
        d
    \end{ytableau}+   \beta \,\,  \begin{ytableau}
        b&c\\
        d
    \end{ytableau}.
\end{align}

The product $\ast$ places two tableaux as disconnected components of a single shape. We equip $\Y$ with the \emph{cutting coproduct} $\Delta_{cut}$ (cf.~\cref{sec:YTalgsetup}). It cuts a tableau $w$ of shape $\lambda/\mu$ into an upper-left part and a lower-right part along each intermediate shape. For example,
\begin{align*}\ytableausetup{centertableaux,boxsize=1.3em}
    \Delta_{cut}\left(\begin{ytableau}
        a&b\\
        c
    \end{ytableau}\right)
    = 1\otimes\begin{ytableau}
        a&b \\ c
    \end{ytableau}
    +\begin{ytableau}
        a
    \end{ytableau} \otimes \begin{ytableau}
        \none & b\\ c
    \end{ytableau}
    + \begin{ytableau}
        a&b 
    \end{ytableau}\otimes \begin{ytableau}
        c
    \end{ytableau}
    + \begin{ytableau}
        a \\ c
    \end{ytableau}\otimes \begin{ytableau}
        b
    \end{ytableau}
    +\begin{ytableau}
        a&b\\ c
    \end{ytableau}\otimes1.
\end{align*}

To connect $\Y$ with the quasi-shuffle algebra $\KA$, we recall how a Schur multiple zeta value expands into multiple zeta values. The semi-standard condition only constrains the relative order of the summation variables. We compare all orderings compatible with this condition. This expands a Schur multiple zeta value as a sum of multiple zeta values. Consider the tableau
\begin{align}\label{eq:tableau21}
    \mathbf{h}=\ytableausetup{centertableaux,boxsize=1.3em}
    \begin{ytableau}
      a & b \\ c
    \end{ytableau}
\end{align}
of shape $(2,1)$, whose entries $a,b,c$ stand for exponents. The semi-standard condition requires $m_a\leq m_b$ along the row and $m_a<m_c$ along the column. Enumerating the orderings of $m_a,m_b,m_c$ compatible with these constraints gives
\begin{align}\label{eq:SMZVtoMZV}
\zeta\!\left(\begin{ytableau} a & b \\ c \end{ytableau}\right)
    = \zeta(a+b,\,c)
    + \zeta(a,\,b+c)
    + \zeta(a,\,b,\,c)
    + \zeta(a,\,c,\,b).
\end{align}

We formalize this procedure through the \emph{linearization map} $L_{\diamond}$. A \emph{semi-standard decomposition} of $\mathbf{h}$ is a way of cutting its boxes into ordered groups $D_1,\dots,D_r$ (cf.~\cref{sec:Lmap}, \cref{def:SSD}). Filling each box of $D_a$ with the number $a$ produces a semi-standard tableau: the entries weakly increase along rows and strictly increase along columns. Each decomposition records one ordering of the boxes compatible with the semi-standard condition. Writing $\operatorname{SSD}(\Sh(\mathbf{h}))$ for the set of all such decompositions and $|\mathbf{h}_{D_a}|_\diamond$ for the $\diamond$-product of the entries in $D_a$, we define
\begin{align*}
    L_\diamond:\Y\longrightarrow\KA,
    \qquad
    L_\diamond(\mathbf{h})
    = \sum_{(D_1,\dots,D_r)\in \operatorname{SSD}(\Sh(\mathbf{h}))}|\mathbf{h}_{D_1}|_\diamond\cdots|\mathbf{h}_{D_r}|_\diamond.
\end{align*}
For the tableau $\mathbf{h}$ given in \eqref{eq:tableau21}, the four semi-standard decompositions are represented by the labeled tableaux below. The label $a$ marks the boxes of $D_a$:
\begin{align*}
    \begin{ytableau} 1 & 1 \\ 2 \end{ytableau},\quad
    \begin{ytableau} 1 & 2 \\ 2 \end{ytableau},\quad
    \begin{ytableau} 1 & 2 \\ 3 \end{ytableau},\quad
    \begin{ytableau} 1 & 3 \\ 2 \end{ytableau},
\end{align*}
and therefore
\begin{align*}
    L_\diamond(\mathbf{h})
    = (a\diamond b)\,c
      + a\,(b\diamond c)
      + a\,b\,c
      + a\,c\,b.
\end{align*}
In the classical setting $\A=\N$ and $a\diamond b=a+b$, this recovers the expansion of $\zeta(\mathbf{h})$ in \eqref{eq:SMZVtoMZV}. The map $L_\diamond$ thus promotes the case-by-case linearization of Schur multiple zeta values to an algebra homomorphism valid for every product $\diamond$. It is surjective, and the quotient $$\Yd:=\Y/\ker(L_\diamond)$$ carries the induced structure.

The first main theorem establishes the Hopf algebra structure and identifies $\Yd$ with the quasi-shuffle Hopf algebra.
\begin{thmA}\label{mainthmA:1}
The following holds.
\begin{enumerate}
    \item[\textnormal{(1)}]
    $(\Y,\ast,\Delta_{cut},S)$ is a connected commutative graded Hopf algebra.
    \item[\textnormal{(2)}]
    $L_\diamond$ induces an isomorphism of Hopf algebras
    \begin{align*}
        (\Yd,\ast,\Delta_{cut})\cong(\KA,\ast_\diamond,\Delta_{\mathrm{dec}}).
    \end{align*}
\end{enumerate}
\end{thmA}
$\Y$ is graded by the number of boxes. The antipode $S$ can be computed by a recursive formula (cf.~\eqref{eq:antipode}). For example,
\begin{equation*}
    S\!\left(\begin{ytableau} a & b \\ c \end{ytableau}\right)=
    -\,\begin{ytableau} a & b \\ c \end{ytableau}
    +\begin{ytableau} a & b \end{ytableau}
    \ast\!\begin{ytableau}  c \end{ytableau}
    +\begin{ytableau} a  \\ c \end{ytableau}
    \ast\!\begin{ytableau}  b \end{ytableau}
    -\begin{ytableau} a  \end{ytableau}
    \ast\!\begin{ytableau} b  \end{ytableau}
    \ast\!\begin{ytableau} c  \end{ytableau}.
\end{equation*}
The isomorphism in part~(2) transports every relation in the quasi-shuffle Hopf algebra to a relation among Young tableaux, and conversely.

\subsection{The Hopf subalgebra of diagonal constant Young tableaux}
Let $\Ydiag$ be the subalgebra of $\Y$ generated by the connected \emph{diagonal constant} Young tableaux, in which all boxes on a common diagonal have the same entry. Equivalently, $\Ydiag$ is spanned by the diagonal constant Young tableaux. It is closed under the coproduct and is therefore a Hopf subalgebra of $\Y$. 

The second and third main theorems concern the subalgebra $\Ydiag$. 
The first identity expands a single $\diamond$-power $h^{\diamond l}$, which corresponds to a depth-one value, as an alternating sum of hook-shaped tableaux. This allows us to write a Riemann zeta value as a combination of Schur multiple zeta values with constant entries. We denote by $\{h\}^{\lambda/\mu}$ the Young tableau of shape $\lambda/\mu$ with all entries equal to $h \in A$ and we write $h^{\diamond n} = \underbrace{h \diamond \dots \diamond h}_n$.

\begin{thmA}\label{Thm:Main-L-diamond-Hook}
    For $h\in\A$ and $l\in\Z_{\geq 1}$, we have
    \begin{align*}
        L_\diamond\!\left(\,\{h^{\diamond l}\}^{(1)}\,\right)
        =\sum_{\substack{a+b=l,\\a\geq1,\,b\geq0}}(-1)^b
         L_\diamond\!\left(\{h\}^{(a,\{1\}^b)}\right).
\end{align*}
\end{thmA}

In \cref{sec:JTFY} we define the \emph{Jacobi--Trudi function} $J:\Ydiag\to\Ydiag$. Its determinant replaces each entry of the classical Jacobi--Trudi matrix with the corresponding column tableau. We show that $J$ acts trivially modulo $\ker L_\diamond$. This allows us to write a diagonal Schur multiple zeta value as a determinant of multiple zeta values. 

\begin{thmA}\label{Thm:Main-L-diamond}
The Jacobi--Trudi function $J$ satisfies the following.
\begin{enumerate}[(1)]
    \item
    $J:\Ydiag\to\Ydiag$ is a Hopf algebra homomorphism
    \item For any $\mathbf{h}\in\Ydiag$ we have
    \begin{align*}
        J(\mathbf{h})-\mathbf{h}\in\ker L_\diamond.
    \end{align*}
    Equivalently, $L_\diamond(\mathbf{h})$ equals the Jacobi--Trudi determinant evaluated in $(\KA,\ast_\diamond)$.
\end{enumerate}
\end{thmA}

Part~(2) states that $J$ acts trivially modulo $\ker L_\diamond$, so the Jacobi--Trudi identity holds in $\Yd$. Applying any algebra homomorphism out of $\Yd$ then yields the corresponding determinant identity. For the homomorphism $\zeta$, this recovers the Jacobi--Trudi formula of~\cite{NPY}. The proof relies only on the Hopf algebra structure of $\Yd$. The same argument applies to the H-type formula of~\cite{NPY}, to the ribbon variants
of~\cite{BC}, and to the Giambelli expression of~\cite{MN}. Since $\Yd$ is isomorphic to the quasi-shuffle Hopf algebra by \cref{mainthmA:1}, the identity in $\Yd$ is a relation in $(\KA,\ast_\diamond)$. This shows that each of them follows from the harmonic relations alone.

\subsection{Multiple Schur series}\label{sec:introMSS}
The Young tableaux Hopf algebra provides a common source for many concrete arithmetic functions, each obtained as an algebra homomorphism from $\Y$. Let $\calP$ be a commutative $\K$-algebra and $\X=(X,\prec)$ a finite totally ordered set. Given a family $\{f_m\}_{m\in\X}$ of $\K$-linear maps $f_m:\K\A\to\calP$, the associated \emph{multiple Schur series} is the Schur-type sum
\begin{align*}
    F_\X(\mathbf{h})
    = \sum_{(m_{i,j})\in\SSYT(\Sh(\mathbf{h}),\X)}
      \prod_{(i,j)\in D(\Sh(\mathbf{h}))} f_{m_{i,j}}(h_{i,j}),
\end{align*}
indexed by semi-standard Young tableaux with entries in $\X$. The summation order matches that of the Schur multiple zeta values, which is the source of the name. For any family $\{f_m\}$, this sum makes $F_\X$ an algebra homomorphism on $\Y$. If each $f_m$ is a $\diamond$-algebra homomorphism, then $F_\X$ descends to the quotient $\Yd$. This is the case that produces our concrete examples. The third main theorem records this structural property.
\begin{thmA}\label{Thm:Main-MSS}
If, for every $m\in \X$, the map $f_m:(\K\A,\diamond)\to\calP$ is a $\K$-algebra homomorphism, then the multiple Schur series $F_\X$ is a $\K$-algebra homomorphism from $(\Yd,\ast)$ to $\calP$.
\end{thmA}
The proof rewrites $F_\X$ as a convolution of the single-variable maps $\widetilde{f}_m$ in the Hopf algebra $\Yd$. The homomorphism property follows from the compatibility of the convolution product with $\ast$. We carry this out in \cref{sec:MSSconvolution}.
A single choice of the family $\{f_m\}$ specializes $F_\X$ to a concrete object, and its homomorphism property encodes the product relations of that object. We treat the Schur multiple zeta values, the Schur multiple Eisenstein series, and the $q$-analogues of the Schur multiple zeta values. In \cref{sec:SMLV} we give two representations of the Schur multiple Dirichlet $L$-values. In the first representation, the maps $\{f_m\}$ are not homomorphisms. Hence the series does not descend to $\Yd$ and is genuinely defined by the semi-standard sum. However, we can always find a ``nice'' algebraic setup to ensure the maps $\{f'_m\}$ are homomorphisms. In the second representation, we demonstrate how this can be achieved.
For Schur multiple Eisenstein series, decomposing $F_\X$ along consecutive segments of $\X$ yields a Fourier expansion in terms of Schur multiple zeta values and the Schur sum of monotangent functions. These functions are studied in~\cite{Yu}.

The constant-entry case links the construction to $\Lambda_{\K}$, the ring of symmetric functions over $\K$. Assume that $(\K\A,\diamond)$ is freely generated and choose $f_m(a)=x_m$ for every generator $a$. Passing to the limit gives 
$$\lim_{M\to\infty}F_{\{1,\dots,M\}}(\{a\}^{\lambda/\mu})=s_{\lambda/\mu}.$$ 
This is the skew Schur function. This identity defines an injective map
\begin{align*}
    \iota_a:\Lambda_{\K} &\longrightarrow \Yd,\\
    s_{\lambda/\mu} &\longmapsto \{a\}^{\lambda/\mu},
\end{align*}
which embeds the ring of symmetric functions into the Young tableaux algebra. Through this embedding, we combine \cref{Thm:Main-L-diamond} with the Hoffman--Ihara exponential identity~\cite{HI}. The result expresses every Young tableau with identical entries $a$ as an explicit polynomial in the depth-one tableaux $\{a^{\diamond h}\}^{(1)}$, for ${h\geq1}$. Under $\iota_a$, this is the classical expansion of a Schur function in the power sums $p_h$, where  $\{a^{\diamond h}\}^{(1)}$ corresponds to $p_h$.

This reduction has a direct consequence for Schur multiple Eisenstein series. For a constant entry $h\geq2$, this writes $\G(\{h\}^{\lambda/\mu};\tau)$ as a polynomial in the depth-one series $\G(hl;\tau)$. Hence $\G(\{h\}^{\lambda/\mu};\tau)$ lies in $\Q[\G(hl;\tau)\mid l\geq1]$. When $h=2$, the series $\{\G(\{2\}^{\lambda/\mu};\tau)\}$ span the ring of quasimodular forms. For even $h\geq4$, each $\G(\{h\}^{\lambda/\mu};\tau)$ is a modular form.

\section*{Acknowledgement}
The author would like to thank Henrik Bachmann for his supervision and kind advice. The author was financially supported by JST SPRING, Grant Number JPMJSP2125, and would like to take this opportunity to thank the ``THERS Make New Standards Program for the Next Generation Researchers''.

\section{Algebra setup}
\subsection{Quasi-shuffle algebras}\label{sec:qsa}
Let $\K$ be a commutative ring, and let $\A$ be a countable set of symbols. Set $\A^{\ast}=\{a_1\cdots a_r|r\geq0, a_1,\dots, a_r\in \A\}$. We call its elements \textit{words} in the alphabet $\A$. Let $\K\A$ be the $\K$-vector space with basis $\A$, and $\KA$ the free algebra over $\K$ generated by the alphabet $\A$. Denote by $\mathbf{1}$ the empty word. Let $\diamond$ be a commutative and associative $\K$-bilinear product on $\K\A$. The corresponding \emph{{quasi-shuffle product}} $\ast_\diamond$ on $\KA$ is defined by $\mathbf{1}\ast_\diamond w=w\ast_\diamond\mathbf{1}=w$ and 
\begin{align}
aw\ast_\diamond bv=a(w\ast_\diamond bv)+b(aw\ast_\diamond v)+(a\diamond b)(w\ast_\diamond v)
\label{eq:quasishuffle} 
\end{align}
for arbitrary words $w,v\in\A^{\ast}$ and symbols $a,b\in \A$. 

\begin{lemma}\label{lem:full-merge}
Let $a_1,\dots,a_t\in\A$ be letters. The word-length $1$ part of $a_1\ast_{\diamond}a_2\ast_{\diamond}\cdots\ast_{\diamond}a_t$ equals the element $a_1\diamond a_2\diamond\cdots\diamond a_t\in\K\A$.
\end{lemma}

\begin{proof}
For letters $a,b\in\A$ and words $u,v\in\KA$, the quasi-shuffle product is defined by
\begin{align}
    \mathbf{1}\ast_{\diamond}u=u\ast_{\diamond}\mathbf{1}=u,
    \quad
    (au)\ast_{\diamond}(bv)
    =
    a(u\ast_{\diamond}(bv))
    +b((au)\ast_{\diamond}v)
    +(a\diamond b)(u\ast_{\diamond}v).
    \label{eq:qsh-left-rec}
\end{align}
We use induction on $t$. The case $t=1$ is clear: the word $a_1$ has length $1$ with coefficient $1$.

Suppose the result holds for $t-1$ letters. Let $P:=a_1\ast_{\diamond}\cdots\ast_{\diamond}a_{t-1}$. By the induction hypothesis, the word-length $1$ part of $P$ is $a_1\diamond\cdots\diamond a_{t-1}$ with coefficient $1$. Write $P=\sum_{r\geq 1}P_r$, where $P_r$ is the word-length $r$ part of $P$. We compute $P\ast_{\diamond}a_t=\sum_{r\geq 1}P_r\ast_{\diamond}a_t$.

For $r\geq 2$: take a word $a_{i_1}\cdots a_{i_r}$ in $P_r$ and apply \eqref{eq:qsh-left-rec} with $u=a_{i_2}\cdots a_{i_r}$, $a=a_{i_1}$, $b=a_t$, $v=\mathbf{1}$. Letters of \(a_{i_1}\cdots a_{i_r}\) cannot merge with each other under quasi-shuffling with a single letter. They can only merge with \(a_t\); hence every output has length \(\ge r \ge 2\).

For $r=1$: $P_1=c$, where $c:=a_1\diamond\cdots\diamond a_{t-1}\in\K\A$. In general $c$ is a $\K$-linear combination of letters, not a single letter; since \eqref{eq:qsh-left-rec} is bilinear, it still applies with $a=c$, $u=\mathbf{1}$, $b=a_t$, $v=\mathbf{1}$:
\begin{align*}
    c\ast_{\diamond}a_t = c(\mathbf{1}\ast_{\diamond}a_t)+a_t(c\ast_{\diamond}\mathbf{1})+(c\diamond a_t)(\mathbf{1}\ast_{\diamond}\mathbf{1}).
\end{align*}
The first term is $ca_t$ and the second is $a_tc$, both of length $2$. The third term is $c\diamond a_t=a_1\diamond\cdots\diamond a_t$, of length $1$. So the word-length $1$ part of $P\ast_{\diamond}a_t$ equals $a_1\diamond\cdots\diamond a_t$.
\end{proof}

For any word $w\in \A^{\ast}$, define the linear map $\Delta_{dec}:\KA\to\KA\otimes\KA$ which splits $w$ into two subwords, i.e. 
\[\Delta_{dec}(w)=\sum_{uv=w}u\otimes v.\] 
This map is the \emph{deconcatenation coproduct}. M.~E.~Hoffman \cite{H} showed that $(\KA,\ast_{\diamond})$ is a commutative $\K$-algebra; see also \cite{HI}. Both references work over fields of characteristic $0$, but the same proof works over any commutative ring $\K$. With $\Delta_{dec}$, it is a bialgebra. In fact, $(\KA,\ast_{\diamond},\Delta_{dec})$ is a Hopf algebra.
\begin{example}
    \begin{enumerate}[(i)]
        \item An important example is the so-called \emph{{harmonic product}} $\ast_{\diamond}=\ast$. It is defined on the commutative algebra $\h^1_{\ast}:=(\Q\langle \A\rangle,\ast)$ generated by the alphabet $\A=\{z_{h}|h\geq1\}$. The corresponding $\diamond$ is defined by 
        \[z_{h_1}\diamond z_{h_2}=
            z_{h_1+h_2}.
        \]
        For an integer $M>0$, this algebra setup defines an algebra homomorphism given by the finite harmonic sum
        \begin{align*}
            \zeta_{M}:\h^1_{\ast}&\to\Q\\
        z_{h_1}\cdots z_{h_r}&\mapsto\zeta_{M}(h_1,\dots,h_r).
        \end{align*}
       The elements $\zeta_M(h_1,\dots,h_r)$ are the \emph{{truncated multiple zeta values}} given by
       \begin{equation*}
           \zeta_M(h_1,\dots,h_r)=\sum_{0<m_1<\cdots<m_r\leq M}\frac{1}{m_1^{h_1}\cdots m_r^{h_r}}.
       \end{equation*}
       \item Another important example uses the same $\diamond$ on the subalgebra $\h^2_{\ast}:=(\Q\langle \mathcal{B}\rangle,\ast)\subset\h^1_{\ast}$ generated by the sub-alphabet $\mathcal{B}=\{z_h|h\geq2\}$. This algebra setup gives an algebra morphism
       \begin{align*}
           \mathbb{G}:\h^2_{\ast}&\to\OH\\
           z_{h_1}\cdots z_{h_r}&\mapsto\mathbb{G}(h_1,\dots,h_r)
       \end{align*}
       whose images are the so-called \emph{{multiple Eisenstein series}}. They are holomorphic functions defined on the upper half-plane $\mathbb{H}$. The constant term of their Fourier expansion is the multiple zeta value with the same index. The other coefficients are MZV-linear combinations.
      \item One more example is the \emph{{$q$-analogue of multiple zeta values}} defined by Bradley \cite{Bra} and Zhao \cite{Zhao}. Taking $\K=\Q[1-q]$, define the $\diamond$ product on $\K\A$ by
   \[
       z_{h_1}\diamond z_{h_2}=z_{h_1+h_2}+(1-q)z_{h_1+h_2-1}.
   \]
   Let $\h^0_{\diamond}\subset\K\langle\A\rangle$ be the $\K$-span of $\mathbf 1$ and of the words $z_{h_1}\cdots z_{h_r}$ with $h_r\geq2$. This span is a $\ast_\diamond$-subalgebra: the last letter of every term in a quasi-shuffle of two such words is one of the two last letters or their $\diamond$-product, and in the latter case both indices $h+h'$ and $h+h'-1$ are at least $3$.
   This specific algebra setup gives an algebra homomorphism
   \begin{align*}
       \zeta_q^{BZ}:(\h^0_{\diamond},\ast_{\diamond})&\to\Q[[q]]\\
       z_{h_1}\cdots z_{h_r}&\mapsto\zeta_q^{BZ}(h_1,\dots,h_r),
   \end{align*}
   
   where the $q$-series $\zeta_q^{BZ}(h_1,\dots,h_r)$ is defined, using the $q$-integers $[m]_q = \frac{1-q^m}{1-q}$, by
   \begin{align*}
       \zeta_q^{BZ}(h_1,\dots,h_r)=\sum_{0<m_1<m_2<\cdots<m_r}\frac{q^{(h_1-1)m_1}\cdots q^{(h_r-1)m_r}}{[m_1]_q^{h_1}\cdots[m_r]_q^{h_r}}.
   \end{align*}
   One reason for its importance is that it recovers the multiple zeta values in the limit as $q \to 1$, i.e.
   \begin{align*}
       \lim_{q\to1}\zeta_q^{BZ}(h_1,\dots,h_r)=\zeta(h_1,\dots,h_r).
   \end{align*}
       \end{enumerate}
       \label{ex:KA}
\end{example}

\subsection{Young tableaux algebra}\label{sec:YTalgsetup}
A \emph{partition} of a natural number $n$ is a tuple $ \lambda=(\lambda_1,\dots,\lambda_l)$ of positive integers satisfying
\begin{align*}
    \lambda_1\geq\cdots\geq\lambda_l\geq1,
    \qquad
    |\lambda|:=\lambda_1+\cdots+\lambda_l=n.
\end{align*}
We call $l$ the \emph{depth} of $\lambda$ and write $\operatorname{dep}(\lambda)=l$.  We also allow the empty partition $\varnothing$, with $|\varnothing|=0$.  When comparing two partitions, we extend them by adding zero parts at the end if necessary.

Let $\lambda$ and $\mu$ be partitions.  We write $\mu\subseteq\lambda$ if
\begin{align*}
    0\leq \mu_i\leq \lambda_i
\end{align*}
for all $i$, after adding zero parts to $\lambda$ and $\mu$ if necessary. For $\mu\subseteq\lambda$, the {\emph{(skew) Young diagram}} of shape
$\lambda/\mu$ is
\begin{align*}
    D(\lambda/\mu)
    :=
    \left\{
        (i,j)\in\Z^2
        \,\middle|\,
        1\leq i\leq \operatorname{dep}(\lambda),
        \ \mu_i<j\leq\lambda_i
    \right\}.
\end{align*}
Here we use the convention that $\mu_i=0$ if $i>\operatorname{dep}(\mu)$.

If $\mu=\varnothing$, or equivalently if all parts of $\mu$ are zero, then we write
\begin{align*}
    D(\lambda):=D(\lambda/\varnothing)
\end{align*}
and call it a \emph{non-skew Young diagram}. In general, when $\mu$ is not necessarily empty, $D(\lambda/\mu)$ is called a \emph{skew Young diagram}.  We refer to $\lambda/\mu$ as a \emph{skew shape}, and to $\lambda$ as a \emph{non-skew shape}.  In this paper, both skew
and non-skew shapes are allowed.

A shape $\lambda/\mu$ is called a {\emph{horizontal strip}} if each column contains at most one box. Equivalently, there are no two distinct boxes $(i,j)$ and $(i',j)$ in $D(\lambda/\mu)$ with the same second coordinate. The simplest examples are the one-row non-skew shapes $\lambda=(n)$, whose Young diagrams consist of a single row.  More generally, a horizontal strip may be skewed and need not be connected.

Let $\A$ be a set and let $\lambda/\mu$ be a skew shape. A Young tableau $\mathbf{h}$ of shape $\lambda/\mu$ with entries in $\A$ is a collection $(h_{i,j})_{(i,j)\in D(\lambda/\mu)}$. For such a tableau, we write $\Sh(\mathbf{h})=\lambda/\mu$.  Later, we write $\eta\subseteq\Sh(\mathbf{h})$ to mean that there exists a partition $\eta'$ with $\mu\subseteq\eta'\subseteq\lambda$ such that $\eta=\eta'/\mu$.
We use the notation $\YT(\lambda/\mu, \A)$ for the set of all \emph{Young tableaux} of shape $\lambda/\mu$ with entries in $\A$:
    \begin{equation*}
        \YT(\lambda/\mu, \A) =\{ (m_{i,j})_{(i,j) \in D(\lambda/\mu)} \mid m_{i,j} \in \A \}.
    \end{equation*} 
 We denote by $\YT(\A)$ the set of all skew Young tableaux with entries in $\A$.

Let $\mathbf{h}$ be a Young tableau with $\Sh(\mathbf{h})=\lambda/\mu$, where $\lambda=(\lambda_1,\dots,\lambda_r)$, and $\mu=(\mu_1,\dots,\mu_r)$. We call $\mathbf{h}$ a  \emph{{connected Young tableau}} if $\mu_i<\lambda_{i+1}$ for every $1\leq i\leq r-1$. If $\mu_j\geq\lambda_{j+1}$ for some $1\leq j\leq r-1$, then we can see that there is no contiguous path of adjacent boxes connecting the boxes in row \( j \) to those in row \( j+1 \) of the Young tableau, so the tableau is not connected. 

\begin{definition} 
Let $\mathbf{h}$ be a Young tableau with $\Sh(\mathbf{h})=\lambda/\mu$. Write $\lambda=(\lambda_1,\dots,\lambda_r)$, and $\mu=(\mu_1,\dots,\mu_r)$, with $0\leq\mu_i\leq\lambda_i$ for $1\leq i\leq r$. We define $\con(\mathbf{h})=\{\mathbf{h}^{(1)},\dots,\mathbf{h}^{(l)}\}$ as the \emph{{multiset of shifted connected components}} of the Young tableau $\mathbf{h}$ as follows.
\begin{enumerate}[(i)]
    \item Let $J=\{1\leq j\leq r|\mu_j\geq\lambda_{j+1}\}$ be the set of all rows not connected to the next row. Set $l=|J|$. Then we can write the Young tableau $\mathbf{h}$ as the tuple $(\mathbf{h}^{'(1)},\dots,\mathbf{h}^{'(l)})$ by cutting the Young tableau $\mathbf{h}$ at each row $j\in J$. Here we set $\lambda_{r+1}=0$ so that $r\in J$ is included, which corresponds to cutting the Young tableau after the last row. 
    \item Assume $0=j_0<j_1<\dots<j_l=r$. For each \(1 \leq i \leq l\), the skew Young tableau \(\mathbf{h}^{(i)}\) is constructed by taking the connected subtableau \(\mathbf{h}'^{(i)} \subseteq \mathbf{h}\), shifting it to the left by \(\min\{\mu_{j_{i-1}+1}, \dots, \mu_{j_i}\}\) boxes, and then shifting it upward by \(j_{i-1}\) boxes, resulting in a skew shape. Here we allow the skew Young tableau $\mathbf{h}^{(i)}$ to be the empty Young tableau, which we denote by $\mathbf{1}$.
\end{enumerate}

\end{definition}

\begin{example} The following examples illustrate the construction of $\con(\mathbf{h})$.
    \begin{enumerate}[(i)]
        \item For $\lambda=(5,3,1)$ and $\mu=(4,3)$, the Young tableau $\mathbf{h}$ has shape $\Sh(\mathbf{h})=\lambda/\mu$
        \begin{center}
\begin{tikzpicture}[scale=0.6] 
    \def\lambda{{5,3,1}}
    \def\mu{{4,3}}

    \definecolor{muColor}{rgb}{0.85,0.85,0.85}

    \node at (-1.2,-1.5) {\(\mathbf{h} =\)};

    \foreach \i in {0,...,2} {
        \pgfmathsetmacro{\rowlen}{\lambda[\i]}
        \foreach \j in {0,...,9} {
            \ifnum\j<\rowlen
                \draw (\j,-\i) rectangle ++(1,-1);
            \fi
        }
    }

    \foreach \i in {0,...,1} {
        \pgfmathsetmacro{\rowlen}{\mu[\i]}
        \foreach \j in {0,...,9} {
            \ifnum\j<\rowlen
                \fill[muColor] (\j,-\i) rectangle ++(1,-1);
            \fi
        }
    }

    \node at (4.5,-0.5) {\(a\)};
    \node at (0.5,-2.5) {\(b\)};

\end{tikzpicture},
\end{center}

       Here $J=\{1,2,3\}$, and hence $\con(\mathbf{h})=\left\{\raisebox{0pt}{ \scalebox{1}{\begin{ytableau}
       a
       \end{ytableau}}},\mathbf{1},\raisebox{0pt}{ \scalebox{1}{\begin{ytableau}
       b
       \end{ytableau}}}\right\}.$
       \item The next example shows the difference between Young tableaux of shapes $(2,2)/(1,1)$ and $(1,1)$. Let
       \begin{center}
\begin{tikzpicture}[scale=0.7]

\node at (-1.2,-1) {\(\mathbf{h}_1 =\)};

\def\lambdaA{{2,2}}
\def\muA{{1,1}}
\definecolor{muColor}{rgb}{0.85,0.85,0.85}

\foreach \i in {0,1} {
    \pgfmathsetmacro{\rowlen}{\lambdaA[\i]}
    \foreach \j in {0,1} {
        \ifnum\j<\rowlen
            \draw (\j,-\i) rectangle ++(1,-1);
        \fi
    }
}

\foreach \i in {0,1} {
    \pgfmathsetmacro{\rowlen}{\muA[\i]}
    \foreach \j in {0,1} {
        \ifnum\j<\rowlen
            \fill[muColor] (\j,-\i) rectangle ++(1,-1);
        \fi
    }
}

\node at (1.5,-0.5) {\(a\)};  
\node at (1.5,-1.5) {\(b\)}; 

\begin{scope}[xshift=5cm]

\node at (-1.2,-1) {\(\mathbf{h}_2 =\)};

\foreach \i in {0,1} {
    \draw (0,-\i) rectangle ++(1,-1);
}

\node at (0.5,-0.5) {\(a\)};  
\node at (0.5,-1.5) {\(b\)};  

\end{scope}
\end{tikzpicture}
\end{center}
As Young tableaux, $\mathbf{h}_1\neq\mathbf{h}_2$, but they have the same multiset
of shifted connected components, i.e. $\con(\mathbf{h}_1)=\con(\mathbf{h}_2)=\left\{\raisebox{0pt}{ \scalebox{0.6}{\begin{ytableau}
       a\\
       b
       \end{ytableau}}}\right\}$, and $J_1=J_2=\{2\}$.
    \end{enumerate}
\end{example}
 
Let $\Y$ be the quotient of the free $\K$-algebra generated by $\YT(\A)$ modulo the relations $\mathbf{h} = \prod_{\mathbf{i} \in \con(\mathbf{h})} \mathbf{i}$ for all Young tableaux $\mathbf{h}$, i.e.
\begin{equation}\label{eq:defKA}
\Y \coloneqq \frac{\K[\YT(\A) ]}{\left\langle \mathbf{h} = \prod_{\mathbf{i} \in \con(\mathbf{h})} \mathbf{i} \right\rangle}.
\end{equation}

Notice that $\Y$ is freely generated by shifted connected Young tableaux. Later, we use $\ast$ to denote the commutative product in $\Y$.   

We now define the cutting coproduct $\Delta_{cut}:\Y\to\Y\otimes\Y$ by cutting a Young tableau into different parts:
\begin{definition}[Cutting coproduct $\Delta_{cut}$ ]
    Let $w$ be a Young tableau with shape $\Sh(w)= \lambda / \mu$. The coproduct $\Delta_{cut}$ is defined by
    \[
    \Delta_{cut}(w)=\sum_{\mu \subseteq \eta \subseteq \lambda}w_{\eta / \mu}\otimes w_{ \lambda /\eta},
    \]
    where, for a skew shape $\alpha$ with $D(\alpha)\subseteq D(\Sh(w))$, $w_{\alpha}$ denotes the corresponding sub-Young tableau.
    The counit $\epsilon:\Y\to\K$ is defined by $\epsilon(\mathbf{1})=1_{\K}$ and $\epsilon(w)=0_{\K}$ otherwise.
\end{definition}
\begin{theorem}
    $\Delta_{cut}$ is well-defined on $\Y$. Moreover, $(\Y,\ast,\Delta_{cut})$ is a bialgebra.
\end{theorem}
\begin{proof}
    Let us start with coassociativity. It suffices to show that $(\Delta_{cut}\otimes \mathrm{id})\circ\Delta_{cut}(w)=(\mathrm{id}\otimes\Delta_{cut})\circ\Delta_{cut}(w)$ for any $w\in\Y$. Let $\Sh(w)=\lambda/\mu$. Then 
    \begin{align*}
        (\Delta_{cut}\otimes \mathrm{id})\circ\Delta_{cut}(w)&=(\Delta_{cut}\otimes \mathrm{id})\left(\sum_{\mu\subseteq\eta\subseteq\lambda}w_{\eta/\mu}\otimes w_{\lambda/\eta}\right)=\sum_{\mu\subseteq\eta\subseteq\lambda}\left(\sum_{\mu\subseteq\nu\subseteq\eta}w_{\nu/\mu}\otimes w_{\eta/\nu}\right)\otimes w_{\lambda/\eta}\\
        &=\sum_{\mu\subseteq\nu\subseteq\eta\subseteq\lambda}w_{\nu/\mu}\otimes w_{\eta/\nu}\otimes w_{\lambda/\eta}=\sum_{\mu\subseteq\nu\subseteq\lambda}w_{\nu/\mu}\otimes \left(\sum_{\nu\subseteq\eta\subseteq\lambda}w_{\eta/\nu}\otimes w_{\lambda/\eta}\right)\\
        &=(\mathrm{id}\otimes\Delta_{cut})\circ\Delta_{cut}(w).
    \end{align*}
   It remains to show that $\Delta_{cut}$ and $\epsilon$ are algebra homomorphisms. Let $w,v\in\Y$ with $\Sh(w)=\lambda/\mu$ and $\Sh(v)=\nu/\eta$. It is enough to verify $\Delta_{cut}(w\ast v)=\Delta_{cut}(w)\ast\Delta_{cut}(v)$ and $\epsilon(w\ast v)=\epsilon(w)\cdot\epsilon(v)$. For the counit $\epsilon$, both sides map to $0$ unless both Young tableaux are empty, in which case they evaluate to $1$. For the coproduct, we first assume that $w,v$ are connected, so $|\con(w)|=|\con(v)|=1$. Set $x=w\ast v$. In the product $w\ast v$, the two tableaux occupy disjoint rows and disjoint columns. Hence $\con(x)=\{w,v\}$, and every sub-Young diagram $\alpha$ of $\Sh(x)$ splits uniquely as $\alpha=\beta\sqcup\gamma$ with $\mu\subseteq\beta\subseteq\lambda$ and $\eta\subseteq\gamma\subseteq\nu$. Therefore,
    \begin{align}
        \Delta_{cut}(x)&=\sum_{\alpha\subseteq\Sh(x)}x_{\alpha}\otimes x_{\Sh(x)/\alpha}=\sum_{\substack{\mu\subseteq\beta\subseteq\lambda\\\eta\subseteq\gamma\subseteq\nu}}(w_{\beta/\mu}\ast v_{\gamma/\eta})\otimes(w_{\lambda/\beta}\ast v_{\nu/\gamma})\\
        &=\left(\sum_{\mu\subseteq\beta\subseteq\lambda}w_{\beta/\mu}\otimes w_{\lambda/\beta}\right)\ast\left(\sum_{\eta\subseteq\gamma\subseteq\nu}v_{\gamma/\eta }\otimes v_{\nu/\gamma}\right)=\Delta_{cut}(w)\ast\Delta_{cut}(v).
        \label{eq:deltasingle}
    \end{align}
    For the second equality, consider a cut of $x$ into an upper part $x_1$ and a lower part $x_2$. Since $w$ and $v$ lie in disjoint rows and columns, this cut restricts to a cut of $w$ into $w_1,w_2$ and a cut of $v$ into $v_1,v_2$. The parts $w_i$ and $v_i$ may be empty or disconnected, so $\con(x_i)=\con(w_i)\sqcup\con(v_i)$. By the defining relation of $\Y$,
    \begin{align*}
        x_i=\prod_{c\in\con(x_i)}c=\Bigl(\prod_{c\in\con(w_i)}c\Bigr)\ast\Bigl(\prod_{c\in\con(v_i)}c\Bigr)=w_i\ast v_i,\qquad i=1,2.
    \end{align*}

    The computation of \eqref{eq:deltasingle} used only that $w$ and $v$ occupy disjoint rows and columns, not that they are connected. Hence \eqref{eq:deltasingle} holds for any $w,v\in\Y$. Now let $w\in\Y$ with $\con(w)=\{w^{(1)},\dots,w^{(l)}\}$. We show
    \begin{align}
        \Delta_{cut}(w)=\prod_{i=1}^l\Delta_{cut}\left(w^{(i)}\right)
    \end{align}
    by induction on $l$. The case $l\le1$ is clear. For $l\ge2$, write $w=w^{(1)}\ast v$ with $v=w^{(2)}\ast\cdots\ast w^{(l)}$. The factor $v$ need not be connected, but \eqref{eq:deltasingle} still applies to $(w^{(1)},v)$. Using it and the induction hypothesis for $v$,
    \begin{align*}
        \Delta_{cut}(w)=\Delta_{cut}\left(w^{(1)}\right)\ast\Delta_{cut}(v)=\prod_{i=1}^l\Delta_{cut}\left(w^{(i)}\right).
    \end{align*}
    Moreover, let $c$ be a connected Young tableau and let $\tilde c$ be a shifted copy of $c$. Shifting induces a bijection between the chains $\mu\subseteq\eta\subseteq\lambda$ for the two shapes. The corresponding cut pieces are shifted copies of each other, hence equal in $\Y$ by \eqref{eq:defKA}. Hence $\Delta_{cut}(c)=\Delta_{cut}(\tilde c)$. Therefore each factor $\Delta_{cut}\bigl(w^{(i)}\bigr)$ may be computed on the shifted representative in $\con(w)$.
    In particular, $\Delta_{cut}(w)$ depends only on the connected components of $w$, so $\Delta_{cut}$ is well-defined on $\Y$. Finally, for any $w,v\in\Y$, \eqref{eq:deltasingle} gives $\Delta_{cut}(w\ast v)=\Delta_{cut}(w)\ast\Delta_{cut}(v)$. Therefore $\Delta_{cut}$ is an algebra homomorphism.

    Finally, we check the counit. For a Young tableau $w$ of shape $\lambda/\mu$,
\begin{align*}
    (\epsilon\otimes\mathrm{id})\Delta_{cut}(w)=\sum_{\mu\subseteq\eta\subseteq\lambda}\epsilon(w_{\eta/\mu})\,w_{\lambda/\eta}=w,
\end{align*}
since only the term $\eta=\mu$ survives; similarly $(\mathrm{id}\otimes\epsilon)\Delta_{cut}(w)=w$. By linearity, the counit holds on $\Y$.
\end{proof}
Since both $\ast$ and $\Delta_{cut}$ respect the grading by the number of boxes, and since the degree-$0$ component is $\Y_0=\K\mathbf 1$, the bialgebra $(\Y,\ast,\Delta_{cut})$ is connected and graded; every connected graded bialgebra admits a unique antipode. We conclude the following.
\begin{theorem}
    $(\Y,\ast,\Delta_{cut},S)$ is a connected commutative graded Hopf algebra.
\end{theorem}
The antipode $S$ is uniquely given by the Takeuchi formula. For $\mathbf{h}\in\Y$ with shape $\lambda/\mu$, the antipode $S$ is 
\begin{equation}\label{eq:antipode}
    S(\mathbf{h})=\sum_{0\le r}(-1)^r\sum_{\mu=I_0\subsetneq I_1\subsetneq\cdots\subsetneq I_r=\lambda}\mathbf{h}_{I_1/I_0}\ast\cdots\ast\mathbf{h}_{I_r/I_{r-1}}.
\end{equation}

Let $\Ydiag$ be the subalgebra of $\Y$ generated by shifted connected Young tableaux whose entries are constant on each diagonal.
\begin{corollary}
    $\Ydiag$ is a commutative graded Hopf subalgebra of $\Y$.\label{cor:CRYdiagHopf}
\end{corollary}
\begin{proof}
$\Ydiag$ is closed under $\ast$ by definition. For a connected diagonal constant tableau $w$, every cut piece $w_{\eta/\mu}$ and $w_{\lambda/\eta}$ is again diagonal constant, so $\Delta_{cut}(w)\in\Ydiag\otimes\Ydiag$; since $\Delta_{cut}$ is an algebra homomorphism, $\Ydiag$ is closed under $\Delta_{cut}$. The grading is inherited, and the antipode preserves $\Ydiag$ by \eqref{eq:antipode}.
\end{proof}
\subsection{Linearization}\label{sec:Lmap}

A fundamental property of Schur multiple zeta values is that they can be expressed as linear combinations of multiple zeta values by ordering the summation indices. Multiple zeta values correspond to Schur multiple zeta values of single-column Young tableaux. Our algebra is designed to include this linearization property.
To formalize this expansion, we introduce the notion of semi-standard decompositions.

\begin{definition}[Semi-standard decomposition]\label{def:SSD}
Let $\lambda/\mu$ be a skew shape. If $\lambda=\mu$, we define $\operatorname{SSD}(\lambda/\mu)$ to consist of the empty chain. Assume $\lambda\neq\mu$. A \emph{semi-standard decomposition} of the shape $\lambda/\mu$ is a chain of partitions
\begin{equation*}
    \mu=\eta^{(0)}\subsetneq\eta^{(1)}\subsetneq\cdots
    \subsetneq\eta^{(r)}=\lambda
\end{equation*}
such that for $1\leq t\leq r$ each skew shape $\eta^{(t)}/\eta^{(t-1)}$ is a horizontal strip. For such a chain, set $D_t := D\bigl(\eta^{(t)}/\eta^{(t-1)}\bigr)$ for $1\leq t\leq r$. We use the tuple $(D_1,\dots,D_r)$ to denote the semi-standard decomposition. Thus
\begin{equation*}
    D(\lambda/\mu) =  D_1\sqcup\cdots\sqcup D_r.
\end{equation*}
We denote by $\operatorname{SSD}(\lambda/\mu)$ the set of all such semi-standard decompositions of the shape $\lambda/\mu$.
\end{definition}
\begin{example}
    \begin{align*}
    \operatorname{SSD}((2,1))=\operatorname{SSD}\left( \ydiagram{2,1}\right)=\left\{ { \begin{ytableau}
1 & 1 \\
2 &  \none
\end{ytableau}},\ { \begin{ytableau}
1 & 2 \\
2 &  \none
\end{ytableau}},\ { \begin{ytableau}
1 & 3 \\
2 &  \none
\end{ytableau}},\ { \begin{ytableau}
1 & 2 \\
3 &  \none
\end{ytableau}} \right\}.
\end{align*}
In the first case, ${ \ytableausetup{centertableaux, boxsize=1.3em} \begin{ytableau}
1 & 1 \\
2 &  \none
\end{ytableau} }$, $D_1=\{(1,1),(1,2)\}=\ydiagram{2}$, $D_2=\{(2,1)\}=\ydiagram{1}$.
\end{example}
The reason we call these chains semi-standard decompositions is that, if we fill every box in $D_i$ with $i$ for $1\leq i\leq r$, then the Young tableau corresponding to $(D_1,\dots,D_r)$ is a semi-standard Young tableau.
\begin{lemma}\label{lem:SSD}
Let $(D_1,\dots,D_r)\in\operatorname{SSD}(\lambda/\mu)$. Fill every box in $D_t$ with the entry $t$. Then the resulting filling is a semi-standard Young tableau of shape $\lambda/\mu$. The converse also holds.
\end{lemma}
\begin{proof}
By definition, $ D_1\sqcup\cdots\sqcup D_t = D\bigl(\eta^{(t)}/\mu\bigr)$ for each $t$. Hence, the entries weakly increase along rows. It remains to check the column condition. If two boxes in the same column had the same entry $t$, then both would belong to $D_t = D\bigl(\eta^{(t)}/\eta^{(t-1)}\bigr)$, contradicting the assumption that $\eta^{(t)}/\eta^{(t-1)}$ is a horizontal strip. Thus, the entries strictly increase down columns. Therefore, the filling is semi-standard.

Conversely, let $T$ be a semi-standard Young tableau of shape $\lambda/\mu$ with entry set $\{1,\dots,r\}$. Let $D_t$ be the boxes of $T$ with entry $t$, and set $\eta^{(t)}=\mu\cup(D_1\sqcup\cdots\sqcup D_t)$. We only need each $\eta^{(t)}$ to be a partition. Let $(i,j)\in\eta^{(t)}$. If $(i,j)\in\mu$, its left and top neighbors lie in $\mu\subseteq\eta^{(t)}$. Otherwise $(i,j)$ is a box of $T$ with entry at most $t$; by the row condition its left neighbor has entry at most $t$, and by the column condition its top neighbor has entry less than $t$, unless these neighbors lie in $\mu$. In every case the left and top neighbors lie in $\eta^{(t)}$, so $\eta^{(t)}$ is a partition. Hence $\mu=\eta^{(0)}\subsetneq\cdots\subsetneq\eta^{(r)}=\lambda$, the inclusions being strict since each entry occurs. Each $\eta^{(t)}/\eta^{(t-1)}=D_t$ is a horizontal strip, since no column contains two equal entries. Thus $(D_1,\dots,D_r)\in\operatorname{SSD}(\lambda/\mu)$, recovering $D_t$ as the boxes with entry $t$.
\end{proof}

We construct a linear map $L_{\diamond}$ from Young tableaux with entries in $\A$ to the quasi-shuffle algebra $\KA$. We call it a \emph{{linearization of Young tableaux}}.

\begin{definition}[Linearization]
    Let $\mathbf{h}$ be a Young tableau. Define
    \begin{align}
    L_{\diamond}:\Y&\longrightarrow \KA\\
    \mathbf{h}&\longmapsto\sum_{(D_1,\dots,D_r)\in \operatorname{SSD}(\Sh(\mathbf{h}))}|\mathbf{h} _{D_1}|_{\diamond}\cdots|\mathbf{h}_{D_r}|_{\diamond}
    \label{def:Ld}
\end{align}
Here $|\mathbf{h}_{D_t}|_{\diamond}:= \mathop{\diamond}_{(i,j)\in D_t} h_{i,j} \in \K\A$. The product $|\mathbf{h}_{D_1}|_{\diamond} \cdots |\mathbf{h}_{D_r}|_{\diamond}$ is the concatenation product in $\KA$. For the empty tableau $\mathbf{1}$, we set $L_{\diamond}(\mathbf{1})=\mathbf{1}$.
\end{definition}

\begin{proposition}
    $L_{\diamond}:(\Y,\ast)\to(\KA,\ast_{\diamond})$ is a surjective $\K$-algebra homomorphism.
    \label{prop:L-alghom}
\end{proposition}

\begin{proof}
We first prove multiplicativity.  It is enough to prove that
\begin{align*}
    L_{\diamond}(\mathbf{h}\ast\mathbf{g}) = L_{\diamond}(\mathbf{h})\ast_{\diamond}L_{\diamond}(\mathbf{g})
\end{align*}
for Young tableaux $\mathbf{h},\mathbf{g}\in\Y$.  If one of the two tableaux is empty, the identity is immediate.  Hence, we assume that both are non-empty.

In the product $\mathbf{h}\ast\mathbf{g}$, the two tableaux are regarded as two disconnected components.  Therefore, there are no row or column relations between a box of $\mathbf{h}$ and a box of $\mathbf{g}$. Let
\begin{align*}
    D=(D_1,\dots,D_r)\in \operatorname{SSD}(\Sh(\mathbf{h})),
    \qquad
    E=(E_1,\dots,E_s)\in \operatorname{SSD}(\Sh(\mathbf{g})),
\end{align*}
and put
\begin{align*}
    u_i:=|\mathbf{h}_{D_i}|_{\diamond},
    \qquad
    v_j:=|\mathbf{g}_{E_j}|_{\diamond}.
\end{align*}
We describe how the semi-standard decompositions of $\mathbf{h}\ast\mathbf{g}$ are obtained from those of $\mathbf{h}$ and $\mathbf{g}$. Fix $D$ and $E$.  An interleaving of the two ordered decompositions is encoded by an integer $p\ge\max\{r,s\}$ together with two strictly increasing maps
\begin{align*}
    \alpha:\{1,\dots,r\}\hookrightarrow\{1,\dots,p\},
    \qquad
    \beta:\{1,\dots,s\}\hookrightarrow\{1,\dots,p\},
\end{align*}
such that
\begin{align*}
    \operatorname{im}(\alpha)\cup\operatorname{im}(\beta) = \{1,\dots,p\}.
\end{align*}
Here $\alpha(i)$ is the new label assigned to the block $D_i$, and $\beta(j)$ is the new label assigned to the block $E_j$.  For $1\leq t\leq p$, define
\begin{align*}
    F_t = \left(\bigsqcup_{\alpha(i)=t}D_i\right) \sqcup \left(\bigsqcup_{\beta(j)=t}E_j\right).
\end{align*}
Since the images of $\alpha$ and $\beta$ cover $\{1,\dots,p\}$, each $F_t$ is non-empty.  Thus
\begin{align*}
    F=(F_1,\dots,F_p)
\end{align*}
is a decomposition of the boxes of $\mathbf{h}\ast\mathbf{g}$.

We claim that $F$ is semi-standard. Indeed, assign the label $t$ to the boxes of $F_t$.  On $\mathbf{h}$, this filling restricts to the semi-standard filling determined by $D$. On $\mathbf{g}$, it restricts to the filling determined by $E$. Since there are no row or column relations between the two connected components, the resulting filling of $\mathbf{h}\ast\mathbf{g}$ is semi-standard.  Hence
\begin{align*}
    F=(F_1,\dots,F_p)\in \operatorname{SSD}(\mathbf{h}\ast\mathbf{g}).
\end{align*}

Conversely, let
\begin{align*}
    F=(F_1,\dots,F_p)\in \operatorname{SSD}(\mathbf{h}\ast\mathbf{g}).
\end{align*}
Intersecting each block $F_t$ with the boxes of $\mathbf{h}$ and deleting the empty intersections gives a semi-standard decomposition
\begin{align*}
    D=(D_1,\dots,D_r)\in \operatorname{SSD}(\Sh(\mathbf{h})).
\end{align*}
Similarly, intersecting each block $F_t$ with the boxes of $\mathbf{g}$ and deleting the empty intersections gives
\begin{align*}
    E=(E_1,\dots,E_s)\in \operatorname{SSD}(\Sh(\mathbf{g})).
\end{align*}
The original entries $1,\dots,p$ then record two strictly increasing maps
\begin{align*}
    \alpha:\{1,\dots,r\}\hookrightarrow\{1,\dots,p\},
    \qquad
    \beta:\{1,\dots,s\}\hookrightarrow\{1,\dots,p\}.
\end{align*}
Since every $F_t$ is non-empty, the images of $\alpha$ and $\beta$ cover $\{1,\dots,p\}$.  This construction is inverse to the construction above.

It remains to compare the corresponding words in $\KA$.  For each label $1\leq t\leq p$, exactly one of the following three cases occurs:
\begin{align}\label{eq:threecases}
    t\in\operatorname{im}(\alpha)\setminus\operatorname{im}(\beta), \qquad t\in\operatorname{im}(\beta)\setminus\operatorname{im}(\alpha), \qquad t\in\operatorname{im}(\alpha)\cap\operatorname{im}(\beta).
\end{align}
In the first case, say $t=\alpha(i)$, the block $F_t$ comes only from $\mathbf{h}$, and its contribution is
\begin{align*}
    |(\mathbf{h}\ast\mathbf{g})_{F_t}|_{\diamond} = |\mathbf{h}_{D_i}|_{\diamond}  =  u_i.
\end{align*}
In the second case, say $t=\beta(j)$, the block $F_t$ comes only from $\mathbf{g}$, and its contribution is
\begin{align*}
    |(\mathbf{h}\ast\mathbf{g})_{F_t}|_{\diamond}  = |\mathbf{g}_{E_j}|_{\diamond}  =  v_j.
\end{align*}
In the third case, say $t=\alpha(i)=\beta(j)$, the two blocks receive the same label and are merged.  The contribution is therefore
\begin{align*}
    |(\mathbf{h}\ast\mathbf{g})_{F_t}|_{\diamond}  =  |\mathbf{h}_{D_i}|_{\diamond} \diamond |\mathbf{g}_{E_j}|_{\diamond}   = u_i\diamond v_j.
\end{align*}

Fix $D$ and $E$. Summing over all choices of $p,\alpha,\beta$ gives the quasi-shuffle product $(u_1\cdots u_r)\ast_{\diamond}(v_1\cdots v_s)$. Indeed, the three cases in \eqref{eq:threecases} are the three alternatives in a quasi-shuffle \eqref{eq:quasishuffle}: take the next block from the first word, take the next block from the second word, or merge one block from each word by $\diamond$.

Therefore,
\begin{align*}
    L_{\diamond}(\mathbf{h}\ast\mathbf{g}) &= \sum_{F\in \operatorname{SSD}(\mathbf{h}\ast\mathbf{g})} |(\mathbf{h}\ast\mathbf{g})_{F_1}|_{\diamond}\cdots  |(\mathbf{h}\ast\mathbf{g})_{F_p}|_{\diamond}  \\
    &= \sum_{\substack{D\in \operatorname{SSD}(\Sh(\mathbf{h}))\\ E\in \operatorname{SSD}(\Sh(\mathbf{g}))}} \left(      |\mathbf{h}_{D_1}|_{\diamond}\cdots|\mathbf{h}_{D_r}|_{\diamond} \right) \ast_{\diamond} \left(   |\mathbf{g}_{E_1}|_{\diamond}\cdots|\mathbf{g}_{E_s}|_{\diamond} \right) \\
    &= L_{\diamond}(\mathbf{h})\ast_{\diamond}L_{\diamond}(\mathbf{g}).
\end{align*}
Hence $L_{\diamond}$ is multiplicative.

It remains to prove surjectivity.  Let $w=a_1\cdots a_r\in\KA$ be an arbitrary word.  Let $\mathbf{h}$ be the one-column Young tableau whose entries from top to bottom are $a_1,\dots,a_r$:
\begin{align*}
    \mathbf{h}=
    \begin{ytableau}
        a_1\\
        a_2\\
        \vdots\\
        a_r
    \end{ytableau}.
\end{align*}
A one-column tableau has a unique semi-standard decomposition, namely $D_i=\{(i,1)\}$ for $1\leq i\leq r$.
Therefore $L_{\diamond}(\mathbf{h})=a_1\cdots a_r=w$. This completes the proof.
\end{proof}

\medskip
Following the above proposition, we define the quotient
\begin{align}\label{eq:defKAdia}
\Yd=\Yd(\K,\A):=\left. \Y \middle/ \ker(L_{\diamond}) \right..
\end{align}
Then \cref{prop:L-alghom} implies that
\begin{align}
    (\Yd,\ast)\cong(\KA,\ast_{\diamond}).
    \label{eq:isoYd}
\end{align}
\begin{proposition}\label{prop:Ldrt}
     Let $\mathbf{h}$ and $\mathbf{g}$ be non-empty Young tableaux. Define $\mathbf{h}|\mathbf{g}$ by attaching the left side of the bottom-left box of $\mathbf{g}$ to the right side of the top-right box of $\mathbf{h}$. Define $\frac{\mathbf{h}}{\mathbf{g}}$ by attaching the bottom of the bottom-left box of $\mathbf{g}$ to the top of the top-right box of $\mathbf{h}$. Then we have
    \begin{align}
        L_{\diamond}(\mathbf{h}\ast\mathbf{g})=L_{\diamond}(\mathbf{h}|\mathbf{g})+L_{\diamond}\left(\frac{\mathbf{h}}{\mathbf{g}}\right).
        \label{eq:Ldrt}
    \end{align}
\end{proposition}
\begin{example}
    Let $\mathbf{h}=\begin{ytableau}
            a_0&a_1
\end{ytableau}$, $\mathbf{g}=\begin{ytableau}
            b_1\\
            b_0
\end{ytableau}$, then, $\mathbf{h}|\mathbf{g}=\begin{ytableau}
           \none&\none& b_1\\
            a_0&a_1&b_0
\end{ytableau}$, and $\frac{\mathbf{h}}{\mathbf{g}}=\begin{ytableau}
            \none&b_1\\
            \none&b_0\\
           a_0& a_1
\end{ytableau}$. We have
    \begin{align*}
        L_{\diamond}(\mathbf{h}\ast\mathbf{g})&=L_{\diamond}(\mathbf{h})\ast_{\diamond} L_{\diamond}(\mathbf{g})=(a_0a_1+a_0\diamond a_1)\ast_{\diamond}b_1b_0\\
        &=a_0a_1\ast_{\diamond}b_1b_0+(a_0\diamond a_1)\ast_{\diamond}b_1b_0\\
        &= \Big( a_0a_1b_1b_0 + a_0b_1a_1b_0 + b_1a_0a_1b_0 + (a_0\diamond a_1)b_1b_0 + b_1(a_0\diamond a_1)b_0 \\
        &\qquad + a_0(a_1\diamond b_1)b_0 + (a_0\diamond b_1)a_1b_0 + a_0b_1(a_1\diamond b_0) + b_1a_0(a_1\diamond b_0) \\
        &\qquad + ((a_0\diamond a_1)\diamond b_1)b_0 + (a_0\diamond b_1)(a_1\diamond b_0) + b_1((a_0\diamond a_1)\diamond b_0) \Big) \\
        &\quad + \Big( a_0b_1b_0a_1 + b_1a_0b_0a_1 + b_1b_0a_0a_1 + (a_0\diamond b_1)b_0a_1 + b_1(a_0\diamond b_0)a_1 + b_1b_0(a_0\diamond a_1) \Big)\\
        &= L_{\diamond}(\mathbf{h}|\mathbf{g}) + L_{\diamond}\left(\frac{\mathbf{h}}{\mathbf{g}}\right).
    \end{align*}
\end{example}
\begin{proof}
    By definition, $L_{\diamond}$ maps a Young tableau to a sum of words in $\KA$. The summation is indexed by the semi-standard decompositions of the tableau. The shape of the product $\mathbf{h} \ast \mathbf{g}$ is the disconnected union of $\mathbf{h}$ and $\mathbf{g}$.
    
    For any decomposition $(D_1, \dots, D_r) \in \operatorname{SSD}(\Sh(\mathbf{h} \ast \mathbf{g}))$, let the top-right box of $\mathbf{h}$ belong to the block $D_i$, and the bottom-left box of $\mathbf{g}$ belong to the block $D_j$. Since $\mathbf{h}$ and $\mathbf{g}$ are disconnected, there is no structural constraint on the relative order of the block indices $i$ and $j$. Thus, precisely one of the following two relations holds:
    
    \begin{enumerate}[(i)]
        \item $i \le j$: This means the block containing the top-right box of $\mathbf{h}$ precedes or equals the block containing the bottom-left box of $\mathbf{g}$. This is the SSD condition for two adjacent boxes in the same row. The decompositions satisfying this condition correspond to $\operatorname{SSD}(\Sh(\mathbf{h}|\mathbf{g}))$, where $\mathbf{g}$ is attached to the right of $\mathbf{h}$.
        
        \item $i > j$: This means the block containing the bottom-left box of $\mathbf{g}$ strictly precedes the block containing the top-right box of $\mathbf{h}$. This is the SSD condition for two adjacent boxes in the same column. The decompositions satisfying this condition correspond to $\operatorname{SSD}\left(\Sh(\frac{\mathbf{h}}{\mathbf{g}})\right)$, where $\mathbf{g}$ is attached above $\mathbf{h}$.
    \end{enumerate}
    
    These two mutually exclusive cases partition the set $\operatorname{SSD}(\Sh(\mathbf{h} \ast \mathbf{g}))$. Therefore, the sum of generated words in $\KA$ splits over these two sets, yielding $L_{\diamond}(\mathbf{h}\ast\mathbf{g}) = L_{\diamond}(\mathbf{h}|\mathbf{g}) + L_{\diamond}\left(\frac{\mathbf{h}}{\mathbf{g}}\right)$.
\end{proof}
\begin{proposition}\label{prop:L-Delta}
For any Young tableau $\mathbf{h}$, we have
\begin{equation*}
    (L_{\diamond}\otimes L_{\diamond})\Delta_{cut}(\mathbf{h})
    =
    \Delta_{dec}L_{\diamond}(\mathbf{h}).
\end{equation*}
\end{proposition}

\begin{proof}
Let $\Sh(\mathbf{h})=\lambda/\mu$. We use the convention that the empty decomposition contributes the empty
word $\mathbf{1}$. First, applying $\Delta_{cut}$ and then $L_{\diamond}\otimes L_{\diamond}$
gives
\begin{equation*}
\begin{aligned}
    &(L_{\diamond}\otimes L_{\diamond})\Delta_{cut}(\mathbf{h}) = \sum_{\mu\subseteq\eta\subseteq\lambda} L_{\diamond}(\mathbf{h}_{\eta/\mu})\otimes L_{\diamond}(\mathbf{h}_{\lambda/\eta})  \\
   &=\sum_{\mu\subseteq\eta\subseteq\lambda}\left(\sum_{(E_1,\dots,E_p)\in\operatorname{SSD}(\eta/\mu)}|\mathbf{h}_{E_1}|_{\diamond}\cdots|\mathbf{h}_{E_p}|_{\diamond} \right)\otimes\left(\sum_{(F_1,\dots,F_s)\in\operatorname{SSD}(\lambda/\eta)}|\mathbf{h}_{F_1}|_{\diamond}\cdots|\mathbf{h}_{F_s}|_{\diamond}\right).
\end{aligned}
\end{equation*}
On the other hand, applying $L_{\diamond}$ first and then $\Delta_{dec}$ gives
\begin{equation*}
\begin{aligned}
    \Delta_{dec}L_{\diamond}(\mathbf{h})&= \Delta_{dec}\left(\sum_{(D_1,\dots,D_r)\in\operatorname{SSD}(\Sh(\mathbf{h}))}|\mathbf{h}_{D_1}|_{\diamond}\cdots|\mathbf{h}_{D_r}|_{\diamond} \right) \\
    &= \sum_{(D_1,\dots,D_r)\in\operatorname{SSD}(\Sh(\mathbf{h}))} \sum_{q=0}^{r}\left(|\mathbf{h}_{D_1}|_{\diamond}\cdots|\mathbf{h}_{D_q}|_{\diamond}\right)\otimes\left(|\mathbf{h}_{D_{q+1}}|_{\diamond}\cdots|\mathbf{h}_{D_r}|_{\diamond} \right).
\end{aligned}
\end{equation*}
It remains to identify the terms in these two expansions. Let $(D_1,\dots,D_r)\in\operatorname{SSD}(\Sh({\mathbf{h}}))$ and choose a split position $q$ with $0\leq q\leq r$. By \cref{def:SSD}, this decomposition is represented by a chain
\begin{equation*}
    \mu=\eta^{(0)} \subsetneq\eta^{(1)}  \subsetneq \cdots \subsetneq  \eta^{(r)}  = \lambda
\end{equation*}
such that
\begin{equation*}
    D_t
    =
    D\bigl(\eta^{(t)}/\eta^{(t-1)}\bigr)
    \qquad
    (1\leq t\leq r).
\end{equation*}
Set $\eta:=\eta^{(q)}$. Then $\mu\subseteq\eta\subseteq\lambda$. The first part of the chain gives $(D_1,\dots,D_q)\in\operatorname{SSD}(\eta/\mu)$. The second part gives  $(D_{q+1},\dots,D_r)\in\operatorname{SSD}(\lambda/\eta)$. Therefore, each term in the expansion of $\Delta_{dec}L_{\diamond}(\mathbf{h})$ determines a unique term in the expansion of $(L_{\diamond}\otimes L_{\diamond})\Delta_{cut}(\mathbf{h})$.

Conversely, fix a partition $\eta$ with $\mu\subseteq\eta\subseteq\lambda$. Also fix decompositions $(E_1,\dots,E_p)\in\operatorname{SSD}(\eta/\mu)$ and $(F_1,\dots,F_s)\in\operatorname{SSD}(\lambda/\eta)$. Concatenating the two corresponding chains gives a chain from $\mu$ to $\lambda$. Hence
\begin{equation*}
    (E_1,\dots,E_p,F_1,\dots,F_s)\in\operatorname{SSD}(\lambda/\mu)=\operatorname{SSD}(\Sh({\mathbf{h}})),
\end{equation*}
with split position $p$. This construction is inverse to the construction above. Therefore, the corresponding tensor factors are identical, which proves the proposition.
\end{proof}

\cref{prop:L-alghom} and \cref{prop:L-Delta} imply the following result. 
\begin{theorem}
$(\Yd,\ast,\Delta_{cut})\cong(\KA,\ast_{\diamond},\Delta_{dec})$ as Hopf algebras.
\end{theorem}
\begin{proof}
Write $\pi\colon\Y\to\Yd$ for the projection. By \cref{prop:L-alghom}, $L_\diamond$ induces an algebra isomorphism $\overline L_\diamond\colon(\Yd,\ast)\to(\KA,\ast_\diamond)$ with $L_\diamond=\overline L_\diamond\circ\pi$; cf.\ \eqref{eq:isoYd}.

First, $\Delta_{cut}$ descends to $\Yd$. The map $\overline L_\diamond\otimes\overline L_\diamond$ is an isomorphism, with inverse $\overline L_\diamond^{-1}\otimes\overline L_\diamond^{-1}$. Moreover, $L_\diamond\otimes L_\diamond=(\overline L_\diamond\otimes\overline L_\diamond)\circ(\pi\otimes\pi)$. Let $x\in\ker L_\diamond$. By \cref{prop:L-Delta}, $(L_\diamond\otimes L_\diamond)\Delta_{cut}(x)=\Delta_{dec}L_\diamond(x)=0$. Since $\overline L_\diamond\otimes\overline L_\diamond$ is injective, $(\pi\otimes\pi)\Delta_{cut}(x)=0$. Hence $(\pi\otimes\pi)\circ\Delta_{cut}$ vanishes on $\ker L_\diamond$. It induces a map $\overline\Delta_{cut}\colon\Yd\to\Yd\otimes\Yd$ with $\overline\Delta_{cut}\circ\pi=(\pi\otimes\pi)\circ\Delta_{cut}$.

Second, $\overline\Delta_{cut}$ corresponds to $\Delta_{dec}$. Indeed,
\begin{align*}
    (\overline L_\diamond\otimes\overline L_\diamond)\circ\overline\Delta_{cut}\circ\pi
    =(L_\diamond\otimes L_\diamond)\circ\Delta_{cut}
    =\Delta_{dec}\circ L_\diamond
    =\Delta_{dec}\circ\overline L_\diamond\circ\pi.
\end{align*}
Since $\pi$ is surjective, $(\overline L_\diamond\otimes\overline L_\diamond)\circ\overline\Delta_{cut}=\Delta_{dec}\circ\overline L_\diamond$.

Third, the counits correspond. Let $\epsilon_{\KA}$ be the counit of $(\KA,\Delta_{dec})$; it sends $\mathbf 1$ to $1$ and every non-empty word to $0$. Let $w$ be a non-empty Young tableau. Every word in $L_\diamond(w)$ has length at least $1$, so $\epsilon_{\KA}(L_\diamond(w))=0=\epsilon(w)$. Also $\epsilon_{\KA}(L_\diamond(\mathbf 1))=1=\epsilon(\mathbf 1)$. By linearity, $\epsilon_{\KA}\circ L_\diamond=\epsilon$. In particular, $\epsilon$ vanishes on $\ker L_\diamond$ and induces $\overline\epsilon$ on $\Yd$ with $\epsilon_{\KA}\circ\overline L_\diamond=\overline\epsilon$.

Therefore $\overline L_\diamond$ transports $(\ast,\overline\Delta_{cut},\overline\epsilon)$ to $(\ast_\diamond,\Delta_{dec},\epsilon_{\KA})$. The latter triple is a bialgebra, hence so is the former, and $\overline L_\diamond$ is an isomorphism of bialgebras. Finally, $(\KA,\ast_\diamond,\Delta_{dec})$ is a Hopf algebra. Transporting its antipode along $\overline L_\diamond$ gives the antipode on $\Yd$. Hence $\overline L_\diamond$ is an isomorphism of Hopf algebras.
\end{proof}
It is interesting to determine the kernel of $L_{\diamond}$. We do not describe the full kernel. Instead, we focus on two useful families: the hook formula below and the Jacobi--Trudi formula in the next subsection.

\begin{theorem}[Hook formula]
Let $h\in\A$ and $l\in\Z_{\geq1}$. Let $\{h\}^{\lambda/\mu}\in\Y$ denote a Young tableau of shape $\lambda/\mu$ whose entries are all $h$. Then
    \begin{align*}
        L_\diamond\!\left(\,\{h^{\diamond l}\}^{(1)}\,\right) =\sum_{\substack{a+b=l,\\a\geq1,\,b\geq0}}(-1)^bL_\diamond\!\left(\{h\}^{(a,\{1\}^b)}\right).
    \end{align*}
\end{theorem}
\begin{proof}
    Since all entries are $h$, a semi-standard decomposition of a hook shape into $r$ blocks $D_1,\dots,D_r$ corresponds to a word $h^{\diamond c_1}\cdots h^{\diamond c_r}$ in the quasi-shuffle algebra $\KA$, where $c_t=|D_t|\ge1$ and $\sum_{t=1}^r c_t=l$.

    Consider the hook shape $(l-b,\{1\}^b)$ with leg length $b$. In a semi-standard decomposition, the induced filling is weakly increasing along the row and strictly increasing along the column. The corner box (the intersection of the row and the column) is the minimal box of the column, so it must be assigned the minimum block label $1$; that is, it belongs to $D_1$. The $b$ leg boxes lie in a strictly increasing column, so they carry $b$ distinct labels, all greater than $1$, i.e.\ a $b$-subset $S\subseteq\{2,\dots,r\}$.

    Choose the labels of the $b$ leg boxes as a $b$-subset $S\subseteq\{2,\dots,r\}$. Then the whole filling is forced. In the first row, the boxes are filled from left to right by the blocks in increasing label order. Each block $D_t$ receives its prescribed number of arm boxes ($c_t$ if $t\notin S$, and $c_t-1$ if $t\in S$, the remaining one box of $D_t$ being the leg box). The weakly increasing condition along each row uniquely determines their positions. Hence, for fixed $r$ and composition $(c_1,\dots,c_r)$, the decompositions are in bijection with the $b$-subsets $S\subseteq\{2,\dots,r\}$, giving exactly $\binom{r-1}{b}$ of them.

    Exchanging the order of summation, we rearrange the alternating sum on the right-hand side according to the number of blocks $r$ in the generated words:
    \begin{align*}
        \sum_{b=0}^{l-1} (-1)^b L_{\diamond}\left(\{h\}^{(l-b,\{1\}^b)}\right)
        &= \sum_{r=1}^{l} \left( \sum_{c_1+\dots+c_r=l} h^{\diamond c_1}\cdots h^{\diamond c_r} \right) \sum_{b=0}^{r-1} (-1)^b \binom{r-1}{b} \\
        &= \sum_{r=1}^{l} \left( \sum_{c_1+\dots+c_r=l} h^{\diamond c_1}\cdots h^{\diamond c_r} \right) \delta_{r,1}.
    \end{align*}
    By the binomial theorem, the inner sum is $\sum_{b=0}^{r-1} (-1)^b \binom{r-1}{b} = (1-1)^{r-1}$, which vanishes for all $r>1$ and equals $1$ when $r=1$.

    Thus, the only surviving term corresponds to $r=1$, i.e., all boxes belong to a single block, and the generated word is $h^{\diamond l}=L_{\diamond}\left(\{h^{\diamond l}\}^{(1)}\right)$. This completes the proof.
\end{proof}

\subsection{Jacobi--Trudi formula for the Young tableaux algebra}\label{sec:JTFY}
We now study the Jacobi--Trudi formula in the Young tableaux algebra. This formula is the analogue of the classical Jacobi--Trudi formula for Schur polynomials and Schur multiple zeta values. We also give its image in the quasi-shuffle algebra under the linearization map. First, for a set $X$, define
\[T^{\text{diag}}(X)=\{T=(t_{i,j})\in\YT(X)|t_{i,j}=t_{k,l}\text{ if }i-j=k-l\}.\]
An element $T$ of this set is called a diagonal constant Young tableau. Its entries are constant on each diagonal. In the following, we will use \textbf{$a_n:=t_{i,i+n}\,\,(n\in\Z)$} to denote the diagonal entries. 
\begin{example}
    {
    \begin{align*}
       T=
\begin{ytableau}
            \none&t_{1,2}&t_{1,3}&t_{1,4}\\
            \none&t_{2,2}&t_{2,3}\\
            t_{3,1}&t_{3,2}
\end{ytableau}\,=\begin{ytableau}
            \none&a_1&a_2&a_3\\
            \none&a_0&a_1\\
            a_{-2}&a_{-1}
\end{ytableau}\,
    \end{align*}}
\end{example}

\begin{definition}\label{def:JTF}
    Let $\mathbf{h}=(h_{i,j})\in\Ydiag$ be a diagonal constant Young tableau, and write $h_{i,i+n}=a_n$ for $n\in\Z$. Assume that $\Sh(\mathbf{h})=\lambda/\mu$, where $\lambda=(\lambda_1,\dots,\lambda_r)$ and $\mu=(\mu_1,\dots,\mu_r)$. Let $s:=\lambda_1$, and let $\lambda'=(\lambda'_1,\dots,\lambda_s')$ and $\mu'=(\mu'_1,\dots,\mu'_s)$ be the conjugates of $\lambda$ and $\mu$. We allow zero parts in $\lambda'$ and $\mu'$. We define \emph{{the Jacobi--Trudi function $J$}} as the $\K$-linear map $J:\Ydiag\to\Ydiag$,
    \begin{align}
        J(\mathbf{h})=\det\left[\scalebox{0.8}{$
\begin{array}{|c|}
\hline
a_{-\mu'_j + j - 1} \\
\hline
a_{-\mu'_j + j - 2} \\
\hline
\vdots \\
\hline
a_{-\mu'_j+j-(\lambda'_i-\mu'_j-i+j)} \\
\hline
\end{array}$}
\right]_{1\leq i,j\leq s}. 
    \label{eq:JTfun}
    \end{align}
    Here, $a_{-\mu'_j+j-1}\cdots a_{-\mu'_j+j-(\lambda'_i-\mu'_j-i+j)}=\mathbf 1$ if $\lambda'_i-\mu'_j-i+j=0$ and $0$ if $\lambda'_i-\mu'_j-i+j<0$. For the empty tableau, we use the convention that the empty determinant equals $\mathbf 1$, so $J(\mathbf 1)=\mathbf 1$.
\end{definition}
\begin{remark}
    The string $a_{-\mu'_j+j-1}\cdots a_{-\mu'_j+j-(\lambda'_i-\mu'_j-i+j)}$ denotes the single-column Young tableau with these entries from top to bottom. A position $(i,j)$ may call for entries $a_n$ on diagonals that do not meet $\lambda/\mu$; we set such matrix entries to $0$. Such positions never contribute to the determinant: every permutation product containing one of them also contains an entry of negative length, hence vanishes.
\end{remark}

\begin{lemma}\label{rm:JTf}
The Jacobi--Trudi map $J:\Ydiag\to\Ydiag$ is a $\K$-algebra homomorphism.
\end{lemma}

\begin{proof}
It is enough to check the defining relation $\mathbf h=\prod_{\mathbf c\in\operatorname{SC}(\mathbf h)}\mathbf c$ for a diagonal constant tableau $\mathbf h$. Let $\mathbf h$ have shifted connected components $\mathbf h^{(1)},\dots,\mathbf h^{(r)}$. Each component occupies consecutive columns. This splits $\{1,\dots,s\}$ into consecutive intervals. Put each empty column into the interval on its left. Keep rows and columns in their natural order. Then the Jacobi--Trudi matrix of $\mathbf h$ is block upper triangular. Indeed, let $i$ lie in a later interval than $j$. Let $c+1$ be the first column of the interval of $i$. Then $j\leq c<i$. If column $c$ is empty, then $\lambda'_{c+1}\leq\lambda'_c=\mu'_c$. If not, then columns $c$ and $c+1$ meet different components. So they share no row, and $\lambda'_{c+1}\leq\mu'_c$. In both cases $\lambda'_i\leq\lambda'_{c+1}\leq\mu'_c\leq\mu'_j$. Since $i>j$, this gives $\lambda'_i-\mu'_j-i+j<0$. So the $(i,j)$ entry is $0$. The diagonal blocks are the Jacobi--Trudi matrices of the components. Shifting a component changes neither the lengths $\lambda'_i-\mu'_j-i+j$ nor the diagonal labels of the entries. The determinant of a block upper triangular matrix is the product of the determinants of the diagonal blocks. Hence $J(\mathbf h)=\prod_{k=1}^r J(\mathbf h^{(k)})$. Therefore $J$ respects the defining relations of $\Ydiag$ and extends to a $\K$-algebra homomorphism.
\end{proof}

\begin{proposition}
    Let $\mathbf{h}$ and $\mathbf{g}$ be diagonal constant Young tableaux. With the notations in \cref{prop:Ldrt}, we have
    \begin{align}
        J(\mathbf{h}\ast\mathbf{g})=J(\mathbf{h}|\mathbf{g})+J\left(\frac{\mathbf{h}}{\mathbf{g}}\right).
        \label{eq:Jrt}
    \end{align}
\end{proposition}
\begin{proof}
    Let $M$ be the $m\times m$ matrix corresponding to $J(\mathbf{h})$, and let $N$ be the $n\times n$ matrix corresponding to $J(\mathbf{g})$. Then the matrix $E$ corresponding to $J(\mathbf{h}|\mathbf{g})$ takes the block form
    \begin{align}
        E=
        \left(
        \begin{array}{c|c}
          M & \sim \\ \hdashline
          0' & N
        \end{array}
        \right).
    \end{align}
    Here, the block $0'$ is an $n \times m$ matrix whose entries are all zero except for its top-right entry, which is $1$. The coordinates of this element $1$ in the full matrix $E$ are $(m+1, m)$. 
    Furthermore, $J\left(\frac{\mathbf{h}}{\mathbf{g}}\right)$ is the minor $E_{m+1,m}$ of $E$ corresponding to this element $1$. Therefore, by Laplace expansion, we have
    \begin{align}
        \det M \times \det N + (-1)^{2m+1}\det E_{m+1,m} = \det E.
        \label{eq:LpE}
    \end{align}
    Since $J:\Ydiag\to\Ydiag$ is an algebra homomorphism, we have $J(\mathbf{h}\ast\mathbf{g}) = \det M \times \det N$. Thus, equation \eqref{eq:LpE} implies \eqref{eq:Jrt}.
\end{proof}

\begin{theorem}\label{Thm:JTf-homo&LdJ}
The Jacobi--Trudi function $J:\Ydiag\to\Ydiag$ is a Hopf algebra homomorphism. Moreover,  $J(\mathbf{h})-\mathbf{h}\in \ker L_{\diamond}$ for every diagonal constant Young tableau $\mathbf{h}\in\Ydiag$, i.e.,
\begin{align}
    L_{\diamond}(\mathbf{h})=\det\left[a_{-\mu'_j+j-1}\cdots a_{-\mu'_j+j-(\lambda'_i-\mu'_j-i+j)}\right]_{1\leq i,j\leq s}.
    \label{eq:JTf}
\end{align}
\end{theorem}
\begin{proof}
Let $\Sh(\mathbf{h})=\lambda/\mu$ and $s:=\lambda_1$. Let $\lambda'=(\lambda'_1,\dots,\lambda'_s)$ and $\mu'=(\mu'_1,\dots,\mu'_s)$ be the conjugates of $\lambda$ and $\mu$ (zero parts are allowed). For $1\le i,j\le s$ define
\begin{align*}
    \alpha_j := -\mu'_j+j-1,
    \qquad
    \beta_i  := -\lambda'_i+i,
    \qquad
    m_{i,j}  := \alpha_j-\beta_i+1 = \lambda'_i-\mu'_j-i+j.
\end{align*}
The $(i,j)$-entry of the Jacobi--Trudi matrix is the column tableau $T_{i,j}(\mathbf{h})$. Read from top to bottom, its entries are $a_{\alpha_j}, a_{\alpha_j-1}, \dots, a_{\beta_i}$. This tableau has $m_{i,j}$ boxes. We use the conventions $T_{i,j}(\mathbf{h})=\mathbf{1}$ if $m_{i,j}=0$ and $T_{i,j}(\mathbf{h})=0$ if $m_{i,j}<0$. After applying $L_{\diamond}$, the column $T_{i,j}(\mathbf{h})$ becomes the word
\begin{align*}
    w_{i,j} := L_{\diamond}(T_{i,j}(\mathbf{h}))
             = a_{\alpha_j}a_{\alpha_j-1}\cdots a_{\beta_i} \in \KA.
\end{align*}
Since $\mu'$ and $\lambda'$ are weakly decreasing, both sequences
$(\alpha_j)_{j=1}^s$ and $(\beta_i)_{i=1}^s$ are strictly increasing:
\begin{align*}
    \alpha_{j+1}-\alpha_j = 1+\mu'_j-\mu'_{j+1} \ge 1,
    \qquad
    \beta_{i+1}-\beta_i = 1+\lambda'_i-\lambda'_{i+1} \ge 1.
\end{align*}
The proof proceeds in four steps. In Steps 1--3, we show that $J$ is a coalgebra homomorphism. In Step 4, we deduce the identity $L_{\diamond}(J(\mathbf{h}))=L_{\diamond}(\mathbf{h})$.

\medskip
\noindent\textbf{Step 1.}
We expand $\Delta_{cut}(J(\mathbf{h}))$ by cutting each column factor of the Jacobi--Trudi determinant independently.

For integers $p,q$, let $\mathsf C[p,q]$ denote the column Young tableau whose entries, read from top to bottom, are $a_p,a_{p-1},\dots,a_q$. As in \cref{def:JTF}, we use the conventions $\mathsf C[p,q]=\mathbf{1}$ if $q=p+1$ and $\mathsf C[p,q]=0$ if $q>p+1$. In this notation,
\begin{align*}
    T_{i,j}(\mathbf{h})=\mathsf C[\alpha_j,\beta_i].
\end{align*}
For a column tableau, the coproduct is simple because its sub-Young diagrams are the top segments of the column. Cutting $\mathsf C[p,q]$ after its $t$-th box gives the upper part $\mathsf C[p,p-t+1]$ and the lower part $\mathsf C[p-t,q]$. We call $\gamma$ the \emph{cut content}: the upper part is $\mathsf C[p,\gamma+1]$ and the lower part is $\mathsf C[\gamma,q]$, and the extreme values $\gamma=p$ and $\gamma=q-1$ give the cuts at the top and bottom edges. Then, for $p\ge q-1$,
\begin{align}
    \Delta_{cut}\mathsf C[p,q]
    =\sum_{\gamma=q-1}^{p}
     \mathsf C[p,\gamma+1]\otimes\mathsf C[\gamma,q].
    \label{eq:column-cut}
\end{align}
By definition,
\begin{align}
    \Delta_{cut}(\mathbf{h})
    = \sum_{\mu\subseteq\eta\subseteq\lambda}
      \mathbf{h}_{\eta/\mu}\otimes\mathbf{h}_{\lambda/\eta},
    \label{eq:Delta_K}
\end{align}
and
\begin{align}
    J(\mathbf{h})
    = \det\bigl(T_{i,j}(\mathbf{h})\bigr)
    = \sum_{\sigma\in S_s}\operatorname{sgn}(\sigma)
      \prod_{j=1}^{s} \mathsf C[\alpha_j,\beta_{\sigma(j)}].
    \label{eq:J_K}
\end{align}
Since $\Delta_{cut}$ is an algebra homomorphism, applying this to each term of~\eqref{eq:J_K} and using~\eqref{eq:column-cut} for each factor, we have
\begin{align}
    \Delta_{cut}(J(\mathbf{h}))
    =\sum_{\sigma\in S_s}\operatorname{sgn}(\sigma)
     \sum_{\gamma\in\Gamma(\sigma)}
     \left(\prod_{j=1}^{s}\mathsf C[\alpha_j,\gamma_j+1]\right)
     \otimes
     \left(\prod_{j=1}^{s}\mathsf C[\gamma_j,\beta_{\sigma(j)}]\right),
    \label{eq:DeltaJ-gamma}
\end{align}
where
\begin{align*}
    \Gamma(\sigma)
    :=\bigl\{\gamma=(\gamma_1,\dots,\gamma_s)\in\Z^s
      \;\big|\;
      \beta_{\sigma(j)}-1\le\gamma_j\le\alpha_j
      \text{ for all }j\bigr\}.
\end{align*}
The integer $\gamma_j$ denotes the \emph{cut content} of the $j$-th factor. Any pair $(\sigma,\gamma)$ satisfying $\gamma\in\Gamma(\sigma)$ is called \emph{admissible}. Furthermore, if $T_{\sigma(j),j}(\mathbf{h})=0$ for some $j$, we get $\Gamma(\sigma)=\varnothing$, which matches $\Delta_{cut}(0)=0$.

\medskip
\noindent\textbf{Step 2.}
It turns out that the terms in~\eqref{eq:DeltaJ-gamma} with duplicate cut contents cancel out. Assume $(\sigma,\gamma)$ is admissible, with $\gamma_p=\gamma_q$ for some $p<q$. Choosing the lexicographically smallest such pair $(p,q)$, define
\begin{align*}
    d:=\gamma_p=\gamma_q,
    \qquad
    \sigma':=\sigma\circ(p\ q),
    \qquad
    \gamma':=\gamma.
\end{align*}
To see that the new pair $(\sigma',\gamma')$ is also admissible, note that the conditions for positions $j\neq p,q$ remain unchanged. At positions $p$ and $q$, the original admissibility of $(\sigma,\gamma)$ implies
\begin{align*}
    \beta_{\sigma(p)}-1\le d\le\alpha_p,
    \qquad
    \beta_{\sigma(q)}-1\le d\le\alpha_q.
\end{align*}
Swapping $\sigma(p)$ and $\sigma(q)$ just swaps the two lower bounds and keeps the upper bounds the same,
\begin{align*}
    \beta_{\sigma'(p)}-1=\beta_{\sigma(q)}-1\le d\le\alpha_p,
    \qquad
    \beta_{\sigma'(q)}-1=\beta_{\sigma(p)}-1\le d\le\alpha_q.
\end{align*}
Hence $(\sigma',\gamma')$ is admissible.

The two terms have the same tensor factors. The upper factor $\prod_j\mathsf C[\alpha_j,\gamma_j+1]$ depends only on $\gamma$, which does not change. In the lower factor, only the factors at the positions $p$ and $q$ change: they change from $\mathsf C[d,\beta_{\sigma(p)}]$ and $\mathsf C[d,\beta_{\sigma(q)}]$ to $\mathsf C[d,\beta_{\sigma(q)}]$ and $\mathsf C[d,\beta_{\sigma(p)}]$. Since $\Y$ is commutative, the product is still the same. Moreover, $\operatorname{sgn}(\sigma')=-\operatorname{sgn}(\sigma)$.

The map $(\sigma,\gamma)\mapsto(\sigma',\gamma')$ is an involution on the set of admissible pairs with repeated cut contents. The pair $(p,q)$ depends only on $\gamma$. Since $\gamma$ remains the same, applying the map twice gives $\sigma\circ(p\ q)\circ(p\ q)=\sigma$. This map has no fixed points because $\sigma'\neq\sigma$. These paired terms have the same tensor factors but opposite signs, so they cancel out. Therefore,
\begin{align}
    \Delta_{cut}(J(\mathbf{h}))
    = \sum_{\sigma\in S_s}\operatorname{sgn}(\sigma)
      \sum_{\substack{\gamma\in\Gamma(\sigma)\\ \gamma_i \neq \gamma_k \, (i\neq k)}}
      \Biggl( \prod_{j=1}^{s}\mathsf C[\alpha_j,\gamma_j+1] \Biggr)
      \otimes
      \Biggl( \prod_{j=1}^{s}\mathsf C[\gamma_j,\beta_{\sigma(j)}] \Biggr).
    \label{eq:DeltaJ-distinct}
\end{align}

\medskip
\noindent\textbf{Step 3.}
The terms in~\eqref{eq:DeltaJ-distinct} can now be grouped according to their cut contents.

Let $D=\{d_1<d_2<\cdots<d_s\}$ be a set of $s$ integers. Define two $s\times s$ matrices with entries in $\Ydiag$ by
\begin{align*}
    A_{r,j}:=\mathsf C[\alpha_j,\,d_r+1],
    \qquad
    B_{i,r}:=\mathsf C[d_r,\,\beta_i]
    \qquad
    (1\le i,j,r\le s).
\end{align*}
If the entries of $\gamma\in\Gamma(\sigma)$ are all distinct and form the set $D$, then $\gamma_j=d_{\tau(j)}$ for a unique $\tau\in S_s$. The corresponding term in~\eqref{eq:DeltaJ-distinct} then becomes
\begin{align*}
    \operatorname{sgn}(\sigma)
    \left(\prod_{j=1}^{s}A_{\tau(j),j}\right)
    \otimes
    \left(\prod_{j=1}^{s}B_{\sigma(j),\tau(j)}\right).
\end{align*}
Conversely, let $(\sigma,\tau)\in S_s\times S_s$ be arbitrary and put $\gamma_j:=d_{\tau(j)}$. If $\gamma\notin\Gamma(\sigma)$, then either $d_{\tau(j)}>\alpha_j$ for some $j$, and then $A_{\tau(j),j}=0$, or $d_{\tau(j)}<\beta_{\sigma(j)}-1$ for some $j$, and then $B_{\sigma(j),\tau(j)}=0$. Hence extending the sum to all pairs $(\sigma,\tau)$ only adds zero terms. Now substitute $\rho:=\sigma\circ\tau^{-1}$. Then $\operatorname{sgn}(\sigma)=\operatorname{sgn}(\rho)\operatorname{sgn}(\tau)$, and reindexing the product by $r=\tau(j)$, which is allowed since $\Y$ is commutative, gives $\prod_{j}B_{\sigma(j),\tau(j)}=\prod_{r}B_{\rho(r),r}$. The total contribution of the value set $D$ therefore equals
\begin{align*}
    \sum_{\tau,\rho\in S_s}
    \operatorname{sgn}(\tau)\operatorname{sgn}(\rho)
    \left(\prod_{j=1}^{s}A_{\tau(j),j}\right)
    \otimes
    \left(\prod_{r=1}^{s}B_{\rho(r),r}\right)
    =\det(A_D)\otimes\det(B_D),
\end{align*}
and
\begin{align}
    \Delta_{cut}(J(\mathbf{h}))
    =\sum_{d_1<\cdots<d_s}\det(A_D)\otimes\det(B_D).
    \label{eq:DeltaJ-D}
\end{align}

Next, we determine which sets $D$ contribute. Define
\begin{align*}
    \eta'_r:=r-1-d_r
    \qquad
    (1\le r\le s).
\end{align*}
Since $d_{r+1}-d_r\ge1$, we have $\eta'_r-\eta'_{r+1}=d_{r+1}-d_r-1\ge0$, so $\eta'$ is weakly decreasing. Conversely, every weakly decreasing $\eta'$ arises from a unique strictly increasing $D$.

Every nonzero term of~\eqref{eq:DeltaJ-D} satisfies
    \begin{align}\label{eq:d-support}
        \beta_r-1\le d_r\le\alpha_r \qquad (1\le r\le s),
    \end{align}
meaning $\mu'_r\le\eta'_r\le\lambda'_r$ for all $r$. If $d_k>\alpha_k$ for some $k$, the increasing orderings $d_1<\cdots<d_s$ and $\alpha_1<\cdots<\alpha_s$ ensure that $d_r\ge d_k>\alpha_k\ge\alpha_j$ for all $r\ge k$ and $j\le k$. Therefore,
\begin{align*}
    A_{r,j}=\mathsf C[\alpha_j,d_r+1]=0
    \qquad
    (r\ge k,\ j\le k).
\end{align*}
In the expansion of $\det(A_D)$, the columns $1,\dots,k$ are matched with $k$ distinct rows, while only $k-1$ rows lie outside $\{k,\dots,s\}$. Hence every permutation term contains one of the zero entries above, and $\det(A_D)=0$. Suppose next that $d_k<\beta_k-1$ for some $k$. Since $d_1<\cdots<d_s$ and $\beta_1<\cdots<\beta_s$, for all $i\ge k$ and $r\le k$ we get $d_r\le d_k<\beta_k-1\le\beta_i-1$, hence
\begin{align*}
    B_{i,r}=\mathsf C[d_r,\beta_i]=0
    \qquad
    (i\ge k,\ r\le k),
\end{align*}
and the same argument gives $\det(B_D)=0$. This proves~\eqref{eq:d-support}. In particular, only finitely many sets $D$ contribute to~\eqref{eq:DeltaJ-D}, and every contributing $D$ determines a unique partition $\eta$ with
\begin{align*}
    \mu\subseteq\eta\subseteq\lambda,
    \qquad
    d_r=-\eta'_r+r-1.
\end{align*}
Conversely, every partition $\eta$ with $\mu\subseteq\eta\subseteq\lambda$ gives a strictly increasing sequence $d_r=-\eta'_r+r-1$.

It remains to identify the two determinants. Fix $\eta$ with $\mu\subseteq\eta\subseteq\lambda$ and let $D$ be the associated set. The restrictions $\mathbf{h}_{\eta/\mu}$ and $\mathbf{h}_{\lambda/\eta}$ are again diagonal constant Young tableaux, so $J$ applies to them. Since $d_r+1=-\eta'_r+r$, the entry
\begin{align*}
    A_{r,j}=\mathsf C[-\mu'_j+j-1,\,-\eta'_r+r]
\end{align*}
is the $(r,j)$-entry of the Jacobi--Trudi matrix of $\mathbf{h}_{\eta/\mu}$, and
\begin{align*}
    B_{i,r}=\mathsf C[-\eta'_r+r-1,\,-\lambda'_i+i]
\end{align*}
is the $(i,r)$-entry of the Jacobi--Trudi matrix of $\mathbf{h}_{\lambda/\eta}$. Since the outer partition of $\lambda/\eta$ is $\lambda$ and $\lambda_1=s$, it follows that
\begin{align*}
    \det(B_D)=J(\mathbf{h}_{\lambda/\eta}).
\end{align*}
For $A_D$, the Jacobi--Trudi determinant of $\mathbf{h}_{\eta/\mu}$ in \cref{def:JTF} has size $\eta_1$, which may be smaller than $s$. For $r>\eta_1$ we have $\eta'_r=0$ and hence $d_r+1=r$. If moreover $j\le\eta_1<r$, then $r>j\ge j-\mu'_j=\alpha_j+1$, so $A_{r,j}=0$. If $\eta_1<j\le s$, then $\mu'_j=0$ because $\mu_1\le\eta_1<j$, so $A_{r,j}=\mathsf C[j-1,r]$, which is $0$ for $r>j$ and $\mathbf{1}$ for $r=j$. Thus the last $s-\eta_1$ rows of $A_D$ vanish in the first $\eta_1$ columns and form an upper unitriangular block in the last $s-\eta_1$ columns. The matrix $A_D$ is block upper triangular, and the determinant of the unitriangular block is $\mathbf{1}$, so $\det(A_D)$ equals the determinant of its upper-left $\eta_1\times\eta_1$ block, that is,
\begin{align*}
    \det(A_D)=J(\mathbf{h}_{\eta/\mu}).
\end{align*}

Combining this identification with~\eqref{eq:DeltaJ-D} and~\eqref{eq:Delta_K}, we obtain
\begin{align*}
    \Delta_{cut}(J(\mathbf{h}))
    =\sum_{\mu\subseteq\eta\subseteq\lambda}
     J(\mathbf{h}_{\eta/\mu})\otimes J(\mathbf{h}_{\lambda/\eta})
    =(J\otimes J)(\Delta_{cut}(\mathbf{h})).
\end{align*}
Because $J$ preserves the grading, it commutes with the counit, making it a coalgebra homomorphism. We already know from \cref{rm:JTf} that $J$ is an algebra homomorphism, so it is a bialgebra homomorphism. Finally, because bialgebra homomorphisms between Hopf algebras always commute with antipodes, $J$ is a Hopf algebra homomorphism.
 
\medskip
\noindent\textbf{Step 4.}
Finally, we prove the compatibility with the linearization map, i.e.\ $L_{\diamond}(J(\mathbf{h}))=L_{\diamond}(\mathbf{h})$. We proceed by induction on $n:=|D(\lambda/\mu)|$.
 
\medskip
\textit{1) Base cases.} If $n=0$ then $\mathbf{h}=\mathbf{1}$, $J(\mathbf{h})=\mathbf{1}$, and both sides equal $\mathbf{1}$. If $n=1$ then $\mathbf{h}$ is a single box and $J(\mathbf{h})=\mathbf{h}$, so equality is immediate.
 
\medskip
\textit{2) Inductive step.} Assume $n\ge2$ and that $L_{\diamond}(J(\mathbf{h}'))=L_{\diamond}(\mathbf{h}')$ holds for every $\mathbf{h}'\in\Ydiag$ with $|D(\Sh(\mathbf{h}'))|<n$. The tableau $\mathbf{h}\in\Ydiag$ satisfies $h_{i,j}=a_{j-i}$. Every sub-tableau arising from $\Delta_{cut}(\mathbf{h})$ belongs to $\Ydiag$. Therefore the induction hypothesis applies to each such sub-tableau.
 
Using that $J$ is a coalgebra homomorphism (Step~3) together with \cref{prop:L-Delta}, we obtain
\begin{align}
    \Delta_{dec}(L_{\diamond}(J(\mathbf{h})))
    &= (L_{\diamond}\otimes L_{\diamond})(\Delta_{cut}(J(\mathbf{h}))) \notag\\
    &= \sum_{\mu\subseteq\eta\subseteq\lambda}
       L_{\diamond}(J(\mathbf{h}_{\eta/\mu}))\otimes L_{\diamond}(J(\mathbf{h}_{\lambda/\eta})).
    \label{eq:ddec-LJ}
\end{align}
The boundary terms $\eta=\mu$ and $\eta=\lambda$ contribute $\mathbf{1}\otimes L_{\diamond}(J(\mathbf{h}))$ and $L_{\diamond}(J(\mathbf{h}))\otimes\mathbf{1}$, respectively. For each interior term $\mu\subsetneq\eta\subsetneq\lambda$, the induction hypothesis gives $L_{\diamond}(J(\mathbf{h}_{\eta/\mu}))=L_{\diamond}(\mathbf{h}_{\eta/\mu})$ (and likewise for $\mathbf{h}_{\lambda/\eta}$), so
\begin{align}
    \Delta_{dec}(L_{\diamond}(J(\mathbf{h})))
    &= L_{\diamond}(J(\mathbf{h}))\otimes\mathbf{1}
       +\mathbf{1}\otimes L_{\diamond}(J(\mathbf{h}))
       +\sum_{\mu\subsetneq\eta\subsetneq\lambda}
        L_{\diamond}(\mathbf{h}_{\eta/\mu})\otimes L_{\diamond}(\mathbf{h}_{\lambda/\eta}).
    \label{eq:ddec-LJ-full}
\end{align}
Applying \cref{prop:L-Delta} directly to $\mathbf{h}$ yields the identical formula with $L_{\diamond}(J(\mathbf{h}))$ replaced by $L_{\diamond}(\mathbf{h})$:
\begin{align}
    \Delta_{dec}(L_{\diamond}(\mathbf{h}))
    &= L_{\diamond}(\mathbf{h})\otimes\mathbf{1}
       +\mathbf{1}\otimes L_{\diamond}(\mathbf{h})
       +\sum_{\mu\subsetneq\eta\subsetneq\lambda}
        L_{\diamond}(\mathbf{h}_{\eta/\mu})\otimes L_{\diamond}(\mathbf{h}_{\lambda/\eta}).
    \label{eq:ddec-L-full}
\end{align}
Setting $X:=L_{\diamond}(J(\mathbf{h}))-L_{\diamond}(\mathbf{h})$ and subtracting~\eqref{eq:ddec-L-full} from~\eqref{eq:ddec-LJ-full},
\begin{align}
    \Delta_{dec}(X) = X\otimes\mathbf{1}+\mathbf{1}\otimes X,
    \label{eq:X-primitive}
\end{align}
so $X$ is primitive in the deconcatenation coalgebra $(\KA,\Delta_{dec})$.
 
Write $X=\sum_{r\ge0}X_r$ with $X_r\in\KA_r$. Since $\Delta_{dec}$ is graded, each $X_r$ is primitive separately. Applying the counit $\epsilon$ to~\eqref{eq:X-primitive} gives $\epsilon(X)=0$, hence $X_0=0$. For $r\ge2$, write $X_r=\sum_{a_1,\dots,a_r}c_{a_1,\dots,a_r}\,a_1\cdots a_r$; the component of $\Delta_{dec}(X_r)$ in $\KA_1\otimes\KA_{r-1}$ is $\sum c_{a_1,\dots,a_r}\,a_1\otimes a_2\cdots a_r$, which must vanish by primitivity. Since the tensors $a_1\otimes a_2\cdots a_r$ form a basis of $\KA_1\otimes\KA_{r-1}$, every coefficient is zero, hence $X_r=0$. Therefore, $X=X_1$, and it remains to show that the word-length $1$ parts of $L_{\diamond}(J(\mathbf{h}))$ and $L_{\diamond}(\mathbf{h})$ agree.
 
Set $c_j:=\lambda'_j-\mu'_j$ (the height of column $j$) and $H:=\max_j c_j$. Since $n\ge2$, at least one column is non-empty, so $H\ge1$.

\medskip
\textit{2.1) Word-length $1$ part of $L_{\diamond}(\mathbf{h})$.} A word of length $1$ in $L_{\diamond}(\mathbf{h})$ comes from a one-block semi-standard decomposition $D_1=D(\lambda/\mu)$. Such a decomposition is valid if and only if every column of $\lambda/\mu$ has at most one box: a single block assigns the same entry to all boxes, which violates the strict column condition whenever two boxes share the same column. Hence:
\begin{itemize}
    \item If $H\ge2$: $L_{\diamond}(\mathbf{h})$ has no word-length $1$ part.
    \item If $H=1$: $\lambda/\mu$ is a horizontal strip (each column has at most one box; the shape need not be connected), and the unique one-block decomposition contributes
          \begin{align*}
              |\mathbf{h}|_{\diamond}
              := a_{\alpha_{j_1}}\diamond\cdots\diamond a_{\alpha_{j_t}},
              \qquad
              \{j_1<\cdots<j_t\} := \{j\mid c_j=1\}.
          \end{align*}
\end{itemize}

\medskip
\textit{2.2) Word-length $1$ part of $L_{\diamond}(J(\mathbf{h}))$.} Since $L_{\diamond}$ is an algebra homomorphism (\cref{prop:L-alghom}),
\begin{align}
    L_{\diamond}(J(\mathbf{h}))
    = \sum_{\sigma\in S_s}\operatorname{sgn}(\sigma)\,
      w_{\sigma(1),1}\ast_{\diamond}
      w_{\sigma(2),2}\ast_{\diamond}\cdots\ast_{\diamond}
      w_{\sigma(s),s}.
    \label{eq:Ld-J-expansion}
\end{align}
The quasi-shuffle product can merge letters from \emph{different} factors but never two letters from the \emph{same} factor, so every output word has length at least $\max_j\operatorname{len}(w_{\sigma(j),j})$. A term in~\eqref{eq:Ld-J-expansion} can therefore contribute to word length $1$ only if every factor has length $m_{\sigma(j),j}\in\{0,1\}$, i.e.\
\begin{align}
    \alpha_j \le \beta_{\sigma(j)} \le \alpha_j+1
    \qquad\text{for all }j.
    \label{eq:length-one-condition}
\end{align}
We show that~\eqref{eq:length-one-condition} forces $\sigma=\operatorname{id}$. Suppose $\sigma\neq\operatorname{id}$ and let $p$ be the smallest index with $\sigma(p)\neq p$; then $\sigma(p)>p$. Let $q>p$ be the unique index with $\sigma(q)=p$. From~\eqref{eq:length-one-condition} at $j=p$ and $\sigma(p)>p$:
\begin{align}
    \beta_{p+1}\le\beta_{\sigma(p)}\le\alpha_p+1.
    \label{eq:ineq1}
\end{align}
From~\eqref{eq:length-one-condition} at $j=q$ (so $\beta_{\sigma(q)}=\beta_p$) and $q>p$:
\begin{align}
    \alpha_p+1\le\alpha_{p+1}\le\alpha_q\le\beta_p.
    \label{eq:ineq2}
\end{align}
Combining~\eqref{eq:ineq1} and~\eqref{eq:ineq2}:
\begin{align*}
    \alpha_p+1 \le \alpha_q \le \beta_p < \beta_{p+1} \le \alpha_p+1,
\end{align*}
a contradiction. Therefore $\sigma=\operatorname{id}$.
 
For $\sigma=\operatorname{id}$ each factor has length $m_{j,j}=c_j$:
\begin{itemize}
    \item If $H\ge2$: some $c_j\ge2$, so even the identity term produces no word of length $1$, and $L_{\diamond}(J(\mathbf{h}))$ has no word-length $1$ part.
    \item If $H=1$: each $c_j\in\{0,1\}$ and the non-trivial factors are the single letters $w_{j,j}=a_{\alpha_j}$ for $j\in\{j_1,\dots,j_t\}$. By \cref{lem:full-merge}, the word-length $1$ part of $a_{\alpha_{j_1}}\ast_{\diamond}\cdots\ast_{\diamond} a_{\alpha_{j_t}}$ equals $a_{\alpha_{j_1}}\diamond\cdots\diamond a_{\alpha_{j_t}}  =|\mathbf{h}|_{\diamond}$ with coefficient $1$. Since $\operatorname{sgn}(\operatorname{id})=1$, the word-length $1$ part of $L_{\diamond}(J(\mathbf{h}))$ is $|\mathbf{h}|_{\diamond}$.
\end{itemize}
In both cases, the word-length $1$ parts of $L_{\diamond}(J(\mathbf{h}))$ and $L_{\diamond}(\mathbf{h})$ coincide. Hence $X_1=0$. Together with $X_0=0$ and $X_r=0$ for all $r\ge2$, we obtain $X=0$, i.e.\ $L_{\diamond}(J(\mathbf{h}))=L_{\diamond}(\mathbf{h})$, which gives $J(\mathbf{h})-\mathbf{h}\in\ker L_{\diamond}$. This completes the proof of \cref{Thm:JTf-homo&LdJ}.
\end{proof}

\begin{example}
    Let $\mathbf{h}\in\Ydiag$ have shape $\lambda/\mu=(6,5,5,4,3,3,1)/(3,2,1,1)$. Then $s=\lambda_1=6$, $\lambda'=(7,6,6,4,3,1)$, $\mu'=(4,2,1,0,0,0)$, and
\begin{align*}
    (\alpha_1,\dots,\alpha_6)=(-4,-1,1,3,4,5),
    \qquad
    (\beta_1,\dots,\beta_6)=(-6,-4,-3,0,2,5).
\end{align*}
\begin{center}
        \begin{tikzpicture}[scale=0.55]
            \newcommand{\uboxx}[3]{\draw (#2-1,-#1) rectangle (#2,-#1+1);
              \node at (#2-0.5,-#1+0.5) {$\scriptstyle a_{#3}$};}
            \newcommand{\lboxx}[3]{\draw (#2-1,-#1) rectangle (#2,-#1+1);
              \node at (#2-0.5,-#1+0.5) {$\scriptstyle a_{#3}$};}
           
            \uboxx{1}{4}{3}\uboxx{1}{5}{4}
            \uboxx{2}{3}{1}\uboxx{2}{4}{2}\uboxx{2}{5}{3}
            \uboxx{3}{2}{-1}\uboxx{3}{3}{0}
            \uboxx{4}{2}{-2}\uboxx{4}{3}{-1}
            \uboxx{5}{1}{-4}\uboxx{5}{2}{-3}\uboxx{5}{3}{-2}
            \uboxx{6}{1}{-5}
          
            \lboxx{1}{6}{5}
            \lboxx{3}{4}{1}\lboxx{3}{5}{2}
            \lboxx{4}{4}{0}
            \lboxx{6}{2}{-4}\lboxx{6}{3}{-3}
            \lboxx{7}{1}{-6}

            \node[anchor=west, inner sep=0pt] at (7.2,-3.2) {%
            \scalebox{1}{%
            \begingroup
            \setlength{\arraycolsep}{1.4pt}
            \renewcommand{\arraystretch}{0.78}
            \newcommand{\JTc}[2]{{\scriptstyle\mathsf C[#1,#2]}}
            \newcommand{\JTone}{{\scriptstyle\mathbf 1}}
            \newcommand{\JTzero}{{\scriptstyle 0}}
            $
            J(\mathbf h)=\det
            \left(
            \begin{array}{@{}cccccc@{}}
            \JTc{-4}{-6} & \JTc{-1}{-6} & \JTc{1}{-6} & \JTc{3}{-6} & \JTc{4}{-6} & \JTc{5}{-6} \\
            \JTc{-4}{-4} & \JTc{-1}{-4} & \JTc{1}{-4} & \JTc{3}{-4} & \JTc{4}{-4} & \JTc{5}{-4} \\
            \JTone       & \JTc{-1}{-3} & \JTc{1}{-3} & \JTc{3}{-3} & \JTc{4}{-3} & \JTc{5}{-3} \\
            \JTzero      & \JTone       & \JTc{1}{0}  & \JTc{3}{0}  & \JTc{4}{0}  & \JTc{5}{0}  \\
            \JTzero      & \JTzero      & \JTone      & \JTc{3}{2}  & \JTc{4}{2}  & \JTc{5}{2}  \\
            \JTzero      & \JTzero      & \JTzero     & \JTzero     & \JTone      & \JTc{5}{5}
            \end{array}
            \right)
            $
            \endgroup
            }%
            };
        \end{tikzpicture}
    \end{center}
Take $\sigma=(1,3,5,2,4,6)$, so $\operatorname{sgn}(\sigma)=-1$. The column factors $T_{\sigma(j),j}(\mathbf{h})=\mathsf C[\alpha_j,\beta_{\sigma(j)}]$ are
\begin{align*}
    \mathsf C[-4,-6],\quad
    \mathsf C[-1,-3],\quad
    \mathbf{1},\quad
    \mathsf C[3,-4],\quad
    \mathsf C[4,0],\quad
    \mathsf C[5,5].
\end{align*}
Cutting the $j$-th factor after its $\kappa_j$-th box gives the cut content $\gamma_j=\alpha_j-\kappa_j$.

    \medskip
\noindent (i) \emph{A canceling pair (Step 2).}
Take $\kappa=(2,3,0,7,2,0)$. Then $\gamma=(-6,-4,1,-4,2,5),$ and $\gamma_2=\gamma_4=-4$ is the unique repeated cut content, so $(p,q)=(2,4)$ and $d=-4$. The upper and lower parts of the factors are
\begin{center}
\begin{tabular}{c|c|c|c|c|c}
$j$ & $\sigma(j)$ & $T_{\sigma(j),j}$ & $\kappa_j$ &
upper $\mathsf C[\alpha_j,\gamma_j{+}1]$ & lower $\mathsf C[\gamma_j,\beta_{\sigma(j)}]$\\
\hline
$1$ & $1$ & $\mathsf C[-4,-6]$ & $2$ & $\mathsf C[-4,-5]$ & $\mathsf C[-6,-6]$\\
$2$ & $3$ & $\mathsf C[-1,-3]$ & $3$ & $\mathsf C[-1,-3]$ & $\mathbf{1}$\\
$3$ & $5$ & $\mathbf{1}$       & $0$ & $\mathbf{1}$       & $\mathbf{1}$\\
$4$ & $2$ & $\mathsf C[3,-4]$  & $7$ & $\mathsf C[3,-3]$  & $\mathsf C[-4,-4]$\\
$5$ & $4$ & $\mathsf C[4,0]$   & $2$ & $\mathsf C[4,3]$   & $\mathsf C[2,0]$\\
$6$ & $6$ & $\mathsf C[5,5]$   & $0$ & $\mathbf{1}$       & $\mathsf C[5,5]$
\end{tabular}
\end{center}
The involution of Step~2 is given by $\sigma'=\sigma\circ(2\ 4)=(1,2,5,3,4,6)$. 

Applying $\sigma'$ changes the factors at positions $2$ and $4$ to $\mathsf C[-1,-4]$ and $\mathsf C[3,-3]$. These cuts are still admissible since their cut contents agree. At position $2$, the cut splits $\mathsf C[-1,-4]$ into $\mathsf C[-1,-3]$ and $\mathsf C[-4,-4]$. At position $4$, it splits $\mathsf C[3,-3]$ into $\mathsf C[3,-3]$ and $\mathbf{1}$. All upper parts stay the same. The lower parts $\mathsf C[-4,-4]$ and $\mathbf{1}$ just switch positions, so the lower product is unaffected in the commutative algebra $\Y$. Because $\operatorname{sgn}(\sigma')=-\operatorname{sgn}(\sigma)$, the two terms cancel out. 

\medskip
\noindent (ii) \emph{A surviving term (Step 3).} Take $\kappa=(2,2,0,7,2,0)$. Then $\gamma=(-6,-3,1,-4,2,5)$ is pairwise distinct, with value set $D=\{d_1<\cdots<d_6\}=\{-6,-4,-3,1,2,5\}$. From $\gamma_j=d_{\tau(j)}$ we have
\begin{align*}
    \tau=(1,3,4,2,5,6),
    \qquad
    \rho=\sigma\circ\tau^{-1}=(1,2,3,5,4,6),
\end{align*}
with $\operatorname{sgn}(\tau)=+1$, $\operatorname{sgn}(\rho)=-1$, and indeed $\operatorname{sgn}(\sigma) =\operatorname{sgn}(\rho)\operatorname{sgn}(\tau)=-1$. The associated partition is obtained from $\eta'_r=r-1-d_r$:
\begin{align*}
    \eta'=(6,5,5,2,2,0),
    \qquad
    \eta=(5,5,3,3,3,1),
\end{align*}
and $\mu\subseteq\eta\subseteq\lambda$ holds. The upper parts
\begin{align*}
    \mathsf C[-4,-5],\quad
    \mathsf C[-1,-2],\quad
    \mathbf{1},\quad
    \mathsf C[3,-3],\quad
    \mathsf C[4,3],\quad
    \mathbf{1}
\end{align*}
form the $\tau$-term of $\det(A_D)=J(\mathbf{h}_{\eta/\mu})$, and the lower parts
\begin{align*}
    \mathsf C[-6,-6],\quad
    \mathsf C[-3,-3],\quad
    \mathbf{1},\quad
    \mathsf C[-4,-4],\quad
    \mathsf C[2,0],\quad
    \mathsf C[5,5]
\end{align*}
form the $\rho$-term of $\det(B_D)=J(\mathbf{h}_{\lambda/\eta})$. Hence this summand contributes to the term $J(\mathbf{h}_{\eta/\mu})\otimes J(\mathbf{h}_{\lambda/\eta})$ with $\eta=(5,5,3,3,3,1)$. The following figure shows the shape $\lambda/\mu$ cut along $\eta$.
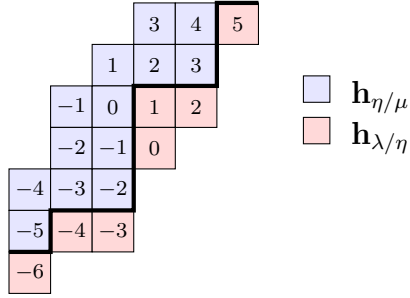
\begin{center}
\begin{tikzpicture}[scale=0.55]
\newcommand{\uboxx}[3]{\draw[fill=blue!10] (#2-1,-#1) rectangle (#2,-#1+1);
  \node at (#2-0.5,-#1+0.5) {$\scriptstyle #3$};}
\newcommand{\lboxx}[3]{\draw[fill=red!15] (#2-1,-#1) rectangle (#2,-#1+1);
  \node at (#2-0.5,-#1+0.5) {$\scriptstyle #3$};}
\uboxx{1}{4}{3}\uboxx{1}{5}{4}
\uboxx{2}{3}{1}\uboxx{2}{4}{2}\uboxx{2}{5}{3}
\uboxx{3}{2}{-1}\uboxx{3}{3}{0}
\uboxx{4}{2}{-2}\uboxx{4}{3}{-1}
\uboxx{5}{1}{-4}\uboxx{5}{2}{-3}\uboxx{5}{3}{-2}
\uboxx{6}{1}{-5}
\lboxx{1}{6}{5}
\lboxx{3}{4}{1}\lboxx{3}{5}{2}
\lboxx{4}{4}{0}
\lboxx{6}{2}{-4}\lboxx{6}{3}{-3}
\lboxx{7}{1}{-6}
\draw[line width=1.6pt]
  (0,-6) -- (1,-6) -- (1,-5) -- (3,-5) -- (3,-2) -- (5,-2) -- (5,0) -- (6,0);
\node[anchor=west] at (6.8,-2.2)
  {\tikz{\draw[fill=blue!10] (0,0) rectangle (0.35,0.35);}
   \ $\mathbf{h}_{\eta/\mu}$};
\node[anchor=west] at (6.8,-3.2)
  {\tikz{\draw[fill=red!15] (0,0) rectangle (0.35,0.35);}
   \ $\mathbf{h}_{\lambda/\eta}$};
\end{tikzpicture}
\end{center}
\end{example}

\begin{remark}
    $L_{\diamond}:\Y\to\KA$ is an algebra homomorphism. Hence, $L_{\diamond}\circ J(\mathbf{h})$ is the determinant on the right-hand side of \eqref{eq:JTf}, which lies in $(\KA,\ast_{\diamond})$. Moreover, the map $J$ introduced in \cref{def:JTF} corresponds to the algebra homomorphism associated with the classical E-type Jacobi--Trudi formula, which expands the skew Young tableau in terms of columns. There are also other constructions of $J$ in the literature. Examples include the H-type Jacobi--Trudi formula and the ribbon version of H. Bachmann and S. Charlton \cite{BC}. These constructions also satisfy the compatibility relation $L_{\diamond}\circ F = L_{\diamond}$, where $F$ denotes the corresponding map. This raises the natural question of whether other such maps exist. 
   The same idea applies to the Giambelli formula for diagonal Schur multiple zeta-functions, established by K.~Matsumoto and M.~Nakasuji \cite{MN}. It induces an algebra homomorphism on $\Ydiag$ satisfying $L_{\diamond}\circ F = L_{\diamond}$. This suggests that a formal proof for this entire class of determinant identities can be conducted by relying solely on the underlying harmonic relations.
    \label{rm:JTLd}
\end{remark}
\begin{example}
    When $\Sh(\mathbf{h})=\lambda/\mu=(4,3,2)/(1,1)$, we have
    \vspace{-0.2cm}
    {\small
    \begin{align*}
        L_{\diamond}\left(\,
\ytableausetup{centertableaux, boxsize=1.5em}
\begin{ytableau}
            \none&a_1&a_2&a_3\\
            \none&a_0&a_1\\
            a_{-2}&a_{-1}
\end{ytableau}\,\right)=\det\begin{pmatrix}
    a_{-2}&a_1a_0a_{-1}a_{-2}&a_2a_1a_0a_{-1}a_{-2}&a_3a_2a_1a_0a_{-1}a_{-2}\\
    1&a_1a_0a_{-1}&a_2a_1a_0a_{-1}&a_3a_2a_1a_0a_{-1}\\
    0&a_1&a_2a_1&a_3a_2a_1\\
    0&0&1&a_3
\end{pmatrix}.
    \end{align*}}
\end{example}

\section{Multiple Schur series}\label{sec:MSS}
In this section, we construct multiple Schur series using the convolution product on the Young tableaux Hopf algebra.  Later, we show that this convolution definition is exactly the usual Schur-type summation over semi-standard Young tableaux.

\subsection{Multiple Schur series via convolution}\label{sec:MSSconvolution}
Let $\calP$ be a commutative $\K$-algebra, and let $\X=(X,\prec)$ be a finite totally ordered set.  We write $a\preceq b$ if either $a\prec b$ or $a=b$.

For a skew shape $\lambda/\mu$, define
\begin{align*}
            \SSYT(\lambda/\mu,\X)=\{ (m_{i,j}) \in      \YT(\lambda/\mu, X) \mid  m_{i,j} \preceq m_{i,j+1},  m_{i,j} \prec m_{i+1,j}\}
        \end{align*}
to be the set of fillings $(m_{i,j})_{(i,j)\in D(\lambda/\mu)}$ with entries in $X$ such that the entries weakly increase along rows and strictly increase along columns.

\begin{definition}\label{def:MSS}
   For each $m\in\X$, let $f_m: \K\A \rightarrow \calP$ be a $\K$-linear map. The associated \emph{multiple Schur series} $F_\X:\Y\to\calP$ is defined on a Young tableau $\mathbf{h}$ by
    \begin{equation}\label{eq:Schursum}
    F_\X(\mathbf{h})=\sum_{(m_{i,j})\in\SSYT(\Sh(\mathbf{h}),\X)}\prod_{(i,j)\in D(\Sh(\mathbf{h}))}f_{m_{i,j}}(h_{i,j}).
\end{equation}
For the empty tableau we set $F_\X(\mathbf{1})=1$, and if $\X=\varnothing$ we set $F_\X=\epsilon$.
\end{definition}

\begin{proposition}\label{prop:MSS-on-Y}
For any family $\{f_m\}_{m\in\X}$, the multiple Schur series $F_\X$ is a $\K$-algebra homomorphism from $(\Y,\ast)$ to $\calP$.
\end{proposition}
\begin{proof}
The algebra $\Y$ has only the defining relations $\mathbf{h}=\prod_{\mathbf{i}\in\con(\mathbf{h})}\mathbf{i}$. So it is enough to check that $F_\X$ respects this relation and that $F_\X(\mathbf{h}\ast\mathbf{g})=F_\X(\mathbf{h})F_\X(\mathbf{g})$.

Let $\mathbf{h}$ be a Young tableau. Two boxes from different connected components lie in different rows and different columns, so they have no row or column relation. Shifting a component does not change its row and column relations, so it does not change its semi-standard condition. Hence, a semi-standard filling of $\mathbf{h}$ is the same as a choice of one semi-standard filling for each connected component. Thus,
\begin{equation*}
    \SSYT(\Sh(\mathbf{h}),\X)=\prod_{\mathbf{i}\in\con(\mathbf{h})}\SSYT(\Sh(\mathbf{i}),\X).
\end{equation*}
Therefore, the product in \eqref{eq:Schursum} splits over the components, and
\begin{equation*}
    F_\X(\mathbf{h})=\prod_{\mathbf{i}\in\con(\mathbf{h})}F_\X(\mathbf{i}).
\end{equation*}
So $F_\X$ respects the defining relation of $\Y$.

For $\mathbf{h},\mathbf{g}\in\Y$, the product $\mathbf{h}\ast\mathbf{g}$ places them as two disconnected components. For the same reason,
\begin{equation*}
    \SSYT(\Sh(\mathbf{h}\ast\mathbf{g}),\X)=\SSYT(\Sh(\mathbf{h}),\X)\times\SSYT(\Sh(\mathbf{g}),\X),
\end{equation*}
so $F_\X(\mathbf{h}\ast\mathbf{g})=F_\X(\mathbf{h})F_\X(\mathbf{g})$. Hence $F_\X$ is a $\K$-algebra homomorphism.
\end{proof}

\cref{prop:MSS-on-Y} shows that $F_\X$ is an algebra homomorphism on $\Y$ for any family $\{f_m\}$. We now look at the special case where each $f_m$ is a $\diamond$-algebra homomorphism. In this case, $F_\X$ has a second description: it factors through the quotient $\Yd$ and can be written as a convolution product. To set this up, we first attach to each $f_m$ a map $\widetilde f_m$ on $\Yd$. The next lemma is the first step.

\begin{lemma}\label{lem:rho-f-character}
Let $f:(\K\A,\diamond)\to\calP$ be a $\K$-algebra homomorphism.  Define a
$\K$-linear map
\begin{align*}
    \rho_f:\KA\longrightarrow\calP
\end{align*}
by
\begin{align*}
    \rho_f(\mathbf{1})=1,
    \qquad
    \rho_f(x)=f(x)\quad (x\in \K\A),
    \qquad
    \rho_f(w)=0\quad \text{otherwise.}
\end{align*}
Then $\rho_f:(\KA,\ast_\diamond)\longrightarrow\calP$ is a $\K$-algebra homomorphism.
\end{lemma}

\begin{proof}
It is enough to check the identity on words.  Let $u,v\in\KA$ be words. If one of $u$ and $v$ is the empty word, the claim is immediate.  Suppose that both are non-empty.  The quasi-shuffle product never merges two letters coming from the same word.  Hence every word appearing in $u\ast_\diamond v$ has length at least $\max\{\ell(u),\ell(v)\}$.
    
Suppose that either $u$ or $v$ has length at least $2$. Then every word in $u\ast_\diamond v$ has length at least $2$. Hence both sides of $\rho_f(u\ast_\diamond v)=\rho_f(u)\rho_f(v)$ are zero.

It remains to consider the case where $u=x$ and $v=y$ are both of word length $1$, with $x,y\in\K\A$. For letters $a,b\in \A$,
\begin{align*}
     a\ast_\diamond b = ab+ba+(a\diamond b).
\end{align*}
The first two terms have word length $2$ and vanish under $\rho_f$. Therefore
\begin{align*}
     \rho_f(a\ast_\diamond b) =\rho_f(a\diamond b)  = f(a\diamond b)  = f(a)f(b)   = \rho_f(a)\rho_f(b),
\end{align*}
because $f$ is an algebra homomorphism for the product $\diamond$. Since the quasi-shuffle product is bilinear, this proves the lemma.
\end{proof}

\begin{definition}\label{def:ftilde-new}
For a $\K$-algebra homomorphism $f:(\K\A,\diamond)\to\calP$, define $\widetilde f:\Yd\longrightarrow\calP$ by
\begin{align*}
    \widetilde f:=\rho_f\circ \overline L_\diamond,
\end{align*}
where $\overline L_\diamond:\Yd\longrightarrow\KA$ is the algebra isomorphism induced by $L_\diamond$.
\end{definition}

\begin{corollary}\label{lem:fm}
The map $\widetilde f:\Yd\to\calP$ is a $\K$-algebra homomorphism.
\end{corollary}

\begin{proof}
The map $\overline L_\diamond$ is a $\K$-algebra homomorphism by \cref{prop:L-alghom}, and $\rho_f$ is a $\K$-algebra homomorphism by \cref{lem:rho-f-character}. Hence, their composition is a $\K$-algebra homomorphism.
\end{proof}
\begin{remark}\label{rem:ftilde-explicit}
For a Young tableau $\mathbf{h}$ viewed in $\Yd$, the map $\widetilde f$ has the following explicit form:
\begin{align*}
    \widetilde f(\mathbf{h}) =
    \begin{cases}
        1,
        & \mathbf{h}=\mathbf{1} \text{ is an empty tableau},\\[2mm]
        \displaystyle
        \prod_{(i,j)\in D(\Sh(\mathbf{h}))} f(h_{i,j}),
        & \Sh(\mathbf{h})\text{ is a horizontal strip},\\[4mm]
        0,
        & \text{otherwise}.
    \end{cases}
\end{align*}
Indeed, $\rho_f$ detects only the length $0$ and length $1$ parts of $L_\diamond(\mathbf{h})$. A word of length $1$ in $L_\diamond(\mathbf{h})$ can only come from a one-block semi-standard decomposition. Such a decomposition exists if and only if $\Sh(\mathbf{h})$ is a horizontal strip: a single block fills every box with the same entry, which is semi-standard exactly when no column contains two boxes. When $\Sh(\mathbf{h})$ is a horizontal strip, its contribution is
\begin{equation*}
     |\mathbf{h}|_\diamond := \mathop{\diamond}_{(i,j)\in D(\Sh(\mathbf{h}))} h_{i,j},
\end{equation*}
and applying $f$ gives the displayed product because $f$ is an algebra homomorphism for $\diamond$. Otherwise no such decomposition exists, so the length $1$ part vanishes and $\widetilde f(\mathbf{h})=0$.
\end{remark}
\begin{theorem}\label{thm:MSS-character}\label{lem:F}
Let $\X=\{m_1\prec m_2\prec\dots\prec m_r\}$. Assume that $f_m:(\K\A,\diamond)\to\calP$ is a $\K$-algebra homomorphism for every $m\in\X$. Then $F_\X$ factors through the quotient $\Yd=\Y/\ker L_\diamond$, and on $\Yd$ it is the ordered convolution
\begin{equation}\label{eq:MSSconv}
    F_\X=\widetilde f_{m_1}\star\widetilde f_{m_2}\star\cdots\star\widetilde f_{m_r},
\end{equation}
where the convolution is taken with respect to the cutting coproduct $\Delta_{cut}$ on $\Yd$:
\begin{equation*}
     \varphi\star\psi:=m_{\calP}\circ(\varphi\otimes\psi)\circ\Delta_{cut}.
\end{equation*}
In particular, $F_\X:\Yd\to\calP$ is a $\K$-algebra homomorphism.
\end{theorem}

\begin{proof}
We show that the convolution in \eqref{eq:MSSconv} equals the Schur sum \eqref{eq:Schursum}. Let $\mathbf{h}$ have shape $\lambda/\mu$. Expanding the convolution along $\Delta_{cut}$ gives
\begin{equation*}
    (\widetilde f_{m_1}\star\cdots\star\widetilde f_{m_r})(\mathbf{h})=\sum_{\mu=\eta^{(0)}\subseteq\eta^{(1)}\subseteq\cdots\subseteq\eta^{(r)}=\lambda}\prod_{\ell=1}^{r}\widetilde f_{m_\ell}\bigl(\mathbf{h}_{\eta^{(\ell)}/\eta^{(\ell-1)}}\bigr).
\end{equation*}
Here empty skew shapes are allowed and contribute $1$. By \cref{rem:ftilde-explicit}, a term is nonzero only when every skew shape $\eta^{(\ell)}/\eta^{(\ell-1)}$ is a horizontal strip. Such chains are in bijection with the semi-standard fillings of $\lambda/\mu$ with entries in $\X$: given a chain, put the entry $m_\ell$ on all boxes of $\eta^{(\ell)}/\eta^{(\ell-1)}$. The rows weakly increase because the $\eta^{(\ell)}$ increase, and the columns strictly increase because each difference is a horizontal strip. Conversely, from a filling $T=(m_{i,j})$ we recover the chain by letting $\eta^{(\ell)}$ be the set of boxes with entries at most $m_\ell$. Under this bijection the term equals $\prod_{(i,j)}f_{m_{i,j}}(h_{i,j})$, so the convolution equals \eqref{eq:Schursum}. In particular $F_\X$ factors through $\Yd$.

It remains to show that the convolution is an algebra homomorphism. Each $\widetilde f_m$ is a $\K$-algebra homomorphism by \cref{lem:fm}. Since $\Delta_{cut}$ is an algebra homomorphism and $\calP$ is commutative, the convolution of two algebra homomorphisms $\Yd\to\calP$ is again one. Indeed, for $x,y\in\Yd$ write $\Delta_{cut}(x)=\sum x_{(1)}\otimes x_{(2)}$ and $\Delta_{cut}(y)=\sum y_{(1)}\otimes y_{(2)}$. Then for algebra homomorphisms $\varphi,\psi:\Yd\to\calP$,
\begin{align*}
    (\varphi\star\psi)(x\ast y)
    &=\sum\varphi(x_{(1)}\ast y_{(1)})\,\psi(x_{(2)}\ast y_{(2)})\\
    &=\sum\varphi(x_{(1)})\varphi(y_{(1)})\,\psi(x_{(2)})\psi(y_{(2)})\\
    &=\Bigl(\sum\varphi(x_{(1)})\psi(x_{(2)})\Bigr)\Bigl(\sum\varphi(y_{(1)})\psi(y_{(2)})\Bigr)=(\varphi\star\psi)(x)\,(\varphi\star\psi)(y).
\end{align*}
By induction on $r$, the ordered convolution $F_\X=\widetilde f_{m_1}\star\cdots\star\widetilde f_{m_r}$ is a $\K$-algebra homomorphism.
\end{proof}

By the definition of multiple Schur series, one can observe a natural decomposition property when the underlying ordered set is divided into consecutive parts. Suppose we cut the ordered set, according to its order, into several consecutive segments. Let $f_1, \ldots, f_s$ be the restrictions of the family $\{f_m\}$ to these segments. Let $F_1, \ldots, F_s$ be the corresponding multiple Schur series associated with $f_1, \ldots, f_s$, respectively. Then we have the following structural property.

\begin{proposition}\label{prop.MSSconv}
Assume each $f_m$ is a $\diamond$-algebra homomorphism. Let $\X=\X_1\sqcup\cdots\sqcup\X_s$ be a decomposition of the totally ordered set $\X$ into consecutive parts; that is, if $a\in\X_i$ and $b\in\X_j$ with $i<j$, then $a\prec b$.  Let $F_i$ be the multiple Schur series associated with the restricted family $(f_m)_{m\in\X_i}$.  Then
\begin{equation*}
     F_\X=F_1\star\cdots\star F_s.
\end{equation*}
\end{proposition}

\begin{proof}
Both sides are the ordered convolution of the same maps $\widetilde f_m$, $m\in\X$, written in the same order.  The result follows from the associativity of convolution.
\end{proof}

\begin{remark}
This convolution property reflects the combinatorial multiplicativity of the multiple Schur structure with respect to the ordered decomposition of indices.
Each $F_i$ can be interpreted as a ``local'' Schur component determined by the restriction of $f$ to a specific segment.
\end{remark}
One point worth noting is that for some algebraic setups, maps $\{f_m\}$ are not $\diamond$-algebra homomorphisms. The Schur sum \eqref{eq:Schursum} is different from the convolution \eqref{eq:MSSconv}. But we can always find a ``nice'' algebraic setup to ensure the multiple Schur series is an algebra homomorphism from $\Yd$ to $\calP$. This means we can always find a way to write multiple Schur series in terms of ``column'' values through $L_\diamond$, which relate to the quasi-shuffle algebra. In \cref{sec:SMLV}, we demonstrate this using Schur multiple $L$-values.

\subsection{Connection with the ring of symmetric functions}\label{sec:RofSymmFunc}
 Given partitions $\lambda$ and $\mu$, the skew Schur function $s_{\lambda/\mu}$ is the function in variables $x_1,x_2, \dots$, defined by
\begin{align*}
    s_{\lambda/\mu}=\sum_{T\in\SSYT(\lambda/\mu,\N)}x^{T},
\end{align*}
where $x^T=x_1^{t_1}x_2^{t_2}\cdots$ and the exponent $t_i$ is the number of boxes in $T$ that have an entry equal to $i$. If the partition $\mu=\varnothing$, then $\lambda/\mu=\lambda$. In this case, $s_{\lambda}$ is the Schur function of shape $\lambda$; these functions give a $\Z$-basis for the ring of symmetric functions. 
Throughout this subsection, let $\Lambda=\Lambda_{\Z}$ denote the ring of symmetric functions over $\Z$. For a commutative ring $R$, we write $\Lambda_R:=\Lambda_{\Z}\otimes_{\Z}R$. In particular, $ \Lambda_{\K}=\Lambda_{\Z}\otimes_{\Z}\K$ is the ring of symmetric functions with coefficients in the commutative ring $\K$. Since the Schur functions $\{s_\lambda\}$ form a $\Z$-basis of $\Lambda_{\Z}$, they form a $\K$-basis of $\Lambda_{\K}$. The definition of the Schur function and \cref{def:MSS} show that the skew Schur function $s_{\lambda/\mu}$ is a multiple Schur series:
\begin{proposition}\label{prop:SymmSchur}
Assume that the algebra $(\K\A,\diamond)$ is freely generated. For $M\geq1$, let $\X=\Z^+_M=\{1,\dots,M\}$ with the standard ordering, and define the family of $\diamond$-algebra homomorphisms $\{f_m\}_{m\in\Z^+_M}$ by
\begin{align*}
    f_m:(\K\A,\diamond)&\longrightarrow\K[x_1,\dots,x_M]\\
    a&\longmapsto x_m
\end{align*}
for every algebra generator $a$ of $(\K\A,\diamond)$. Then, for every skew shape $\lambda/\mu$ and every algebra generator $a$, the constant-entry tableau $\{a\}^{\lambda/\mu}\in\Y$ satisfies
\begin{align*}
    F_{\Z^+_M}\bigl(\{a\}^{\lambda/\mu}\bigr)
    =s_{\lambda/\mu}(x_1,\dots,x_M,0,0,\dots).
\end{align*}
In particular, the values of $F_{\Z^+_M}$ on constant-entry tableaux are symmetric polynomials. Taking the coefficientwise limit in the ring of formal power series $\K[[x_1,x_2,\dots]]$ gives
\begin{align*}
    \lim_{M\to\infty}F_{\Z^+_M}\bigl(\{a\}^{\lambda/\mu}\bigr)=s_{\lambda/\mu}\in\Lambda_{\K}.
\end{align*}
\end{proposition}
\cref{prop:SymmSchur} shows that, for every generator $a$, a constant-entry tableau of shape $\lambda/\mu$ produces the same skew Schur function $s_{\lambda/\mu}$: only the shape matters. This motivates the assignment $s_{\lambda/\mu}\mapsto\{a\}^{\lambda/\mu}$. In the next theorem, we fix an element $a\in\A$ and prove that this assignment extends to an injective $\K$-algebra homomorphism $\iota_a$ sending each skew Schur function $s_{\lambda/\mu}$ to $\{a\}^{\lambda/\mu}$. The proof uses the linearization map $L_\diamond$.
\begin{theorem}\label{Thm:iota}
Let $a\in\A$. The $\K$-linear map defined on the Schur basis
\begin{align*}
    \iota_a:\Lambda_{\K}&\longrightarrow\Yd\\
    s_\lambda&\longmapsto\{a\}^{\lambda}
    \quad
    (\lambda\text{ a partition})
\end{align*}
is a $\K$-algebra homomorphism, and for every skew shape $\lambda/\mu$ one has
\begin{align*}
    \iota_a(s_{\lambda/\mu})=\{a\}^{\lambda/\mu}.
\end{align*}
If the elements $a^{\diamond n}\in\K\A$ $(n\geq1 )$ are linearly independent over $\K$, then $\iota_a$ is injective.
\end{theorem}

\begin{proof}
The map is well-defined because the Schur functions $s_\lambda$ form a $\K$-basis of $\Lambda_{\K}$.

We first prove $\iota_a$ is a $\K$-algebra homomorphism and that it sends skew Schur functions to constant-entry tableaux. Let $\jmath_a:\Lambda_{\K}\to\Yd$ be the $\K$-algebra homomorphism defined by
\begin{align*}
    \jmath_a(e_m)=\{a\}^{(1^m)} \qquad (m\geq1),
\end{align*}
where $e_m$ is the $m$-th elementary symmetric function. This map is well-defined, since $\Lambda_{\K}=\K[e_1,e_2,\dots]$ is the polynomial ring on the $e_m$. By the dual Jacobi--Trudi identity,
\begin{align*}
    s_\lambda = \det\left[ e_{\lambda'_i-i+j} \right]_{1\leq i,j\leq\lambda_1}.
\end{align*}
Applying $\jmath_a$ gives
\begin{align*}
    \jmath_a(s_\lambda) = \det\left[ \{a\}^{(1^{\lambda'_i-i+j})} \right]_{1\leq i,j\leq\lambda_1}.
\end{align*}
By the Jacobi--Trudi formula in $\Yd$ (\cref{Thm:JTf-homo&LdJ}), this determinant equals $\{a\}^{\lambda}$. Hence
\begin{align*}
    \jmath_a(s_\lambda)=\{a\}^{\lambda}=\iota_a(s_\lambda)
\end{align*}
for every partition $\lambda$. Since the $s_\lambda$ form a basis, we get $\jmath_a=\iota_a$. Therefore $\iota_a$ is a $\K$-algebra homomorphism. The same argument applies to the skew Jacobi--Trudi identity
\begin{align*}
    s_{\lambda/\mu} = \det\left[ e_{\lambda'_i-\mu'_j-i+j} \right]_{1\leq i,j\leq\lambda_1}.
\end{align*}
Applying $\iota_a$ and using \cref{Thm:JTf-homo&LdJ} gives
\begin{align*}
    \iota_a(s_{\lambda/\mu}) =\det\left[ \{a\}^{(1^{\lambda'_i-\mu'_j-i+j})} \right]_{1\leq i,j\leq\lambda_1} = \{a\}^{\lambda/\mu}.
\end{align*}

It remains to prove injectivity when $\{a^{\diamond n}\in\K\A\mid n\ge1\}$ are linearly independent over $\K$. It suffices to show that the tableaux $\{a\}^{\lambda}$, where $\lambda$ runs over all partitions, are linearly independent in $\Yd$. Suppose that
\begin{align*}
    \sum_{\lambda} c_\lambda \{a\}^{\lambda}=0
    \qquad
    (c_\lambda\in\K),
\end{align*}
with only finitely many $c_{\lambda}\neq 0$. Applying $L_\diamond$ gives
\begin{align*}
    \sum_{\lambda} c_\lambda\, L_\diamond\bigl(\{a\}^{\lambda}\bigr)=0
    \qquad
    \text{in }\KA.
\end{align*}
For a composition $\alpha=(\alpha_1,\dots,\alpha_l)$ of $n$, write
\begin{align*}
    w_\alpha := a^{\diamond\alpha_1}a^{\diamond\alpha_2}\cdots a^{\diamond\alpha_l}\in\KA.
\end{align*}
Next, we will prove that the elements $w_{\alpha}$, where $\alpha$ runs over all compositions, are linearly independent over $\K$. Since $\KA$ is a free $\K$-module with basis $\A^{\ast}$, we have $\KA=\bigoplus_{l\geq0}\KA_l$ as $\K$-modules, where $\KA_l$ denotes the $\K$-span of the words of length $l$. Expanding each factor $a^{\diamond \alpha_i}$ in the letters shows $w_{\alpha}\in\KA_l$ for every composition $\alpha$ of length $l$. Suppose a $\K$-linear combination of the $w_\alpha$ vanishes. For each $l$, the terms with $\alpha$ of length $l$ lie in $\KA_l$. Since $\bigoplus_{l}\KA_l$ is a direct sum, the terms for each $l$ vanish separately. Therefore, it suffices to prove that the $w_\alpha$ with $\alpha$ of a fixed length $l$ are linearly independent for $l\ge1$.

For $l=1$, the compositions of length $1$ are $\alpha=(t)$ with $t\geq1$, and $w_{(t)}=a^{\diamond t}$. Their linear independence is the assumption in the theorem. Let $l\geq2$ and suppose
\begin{align*}
    \sum_{\alpha=(\alpha_1,\dots,\alpha_l)} d_\alpha\, w_\alpha=0
    \qquad
    (d_\alpha\in\K).
\end{align*}
Write $a^{\diamond t}=\sum_{u\in\A}b_{u,t}\,u$ with $b_{u,t}\in\K$. Every word of length $l$ is uniquely the first letter followed by a word of length $l-1$. Expand the first factor of each $w_\alpha$ in the letters. For each letter $u\in\A$, collect the words starting with $u$ and cancel this first letter. This gives
\begin{align*}
    \sum_{\alpha}d_\alpha\, b_{u,\alpha_1}\,a^{\diamond\alpha_2}\cdots a^{\diamond\alpha_l}=0,
\end{align*}
for every $u\in\A$. By the induction hypothesis, for every letter $u$ and every composition $(\alpha_2,\dots,\alpha_l)$,
\begin{align*}
    \sum_{t\geq1}d_{(t,\alpha_2,\dots,\alpha_l)}\,b_{u,t}=0.
\end{align*}
So the element $\sum_{t\geq1}d_{(t,\alpha_2,\dots,\alpha_l)}\,a^{\diamond t}\in\K\A$ has coefficient $0$ at every letter, hence is $0$. The elements $a^{\diamond t}$ are linearly independent, so $d_{(t,\alpha_2,\dots,\alpha_l)}=0$ for all $t$. This finishes the induction.

For a partition $\lambda$ and a composition $\alpha=(\alpha_1,\dots,\alpha_l)$, let $N_{\lambda,\alpha}$ be the number of semi-standard decompositions $(D_1,\dots,D_l)$ of $\lambda$ with $|D_i|=\alpha_i$ for all $i$. This counts the semi-standard tableaux of shape $\lambda$ where each entry $i$ appears $\alpha_i$ times. So $N_{\lambda,\alpha}=0$ unless $|\lambda|=\alpha_1+\cdots+\alpha_l$. Grouping the semi-standard decompositions by $\alpha$ gives
\begin{align*}
    L_\diamond\bigl(\{a\}^{\lambda}\bigr)=\sum_{\alpha}N_{\lambda,\alpha}\,w_\alpha.
\end{align*}
Substituting this into the relation above writes $0$ as a $\K$-linear combination of the $w_\alpha$. These are linearly independent, so every coefficient is $0$. For every composition $\alpha$, the coefficient of $w_\alpha$ is zero:
\begin{align*}
    \sum_{\lambda} c_\lambda N_{\lambda,\alpha}=0.
\end{align*}
In particular, this holds for every partition $\rho=(\rho_1,\dots,\rho_l)$.
Let $g:=\sum_{\lambda}c_\lambda s_\lambda\in\Lambda_{\K}$. By the definition of the Schur function as the generating series of semi-standard tableaux, the coefficient of $x_1^{\rho_1}\cdots x_l^{\rho_l}$ in $s_\lambda$ is $N_{\lambda,\rho}$. So the coefficient of $x_1^{\rho_1}\cdots x_l^{\rho_l}$ in $g$ is $\sum_{\lambda}c_\lambda N_{\lambda,\rho}=0$ for every partition $\rho$. The function $g$ is symmetric, so the coefficient of any monomial equals the coefficient of the monomial whose exponents are sorted into a partition. So all coefficients of $g$ are $0$, and $g=0$ in $\Lambda_{\K}$. The Schur functions form a $\K$-basis of $\Lambda_{\K}$, so all $c_\lambda$ are $0$. Therefore $\iota_a$ is injective.
This completes the proof.
\end{proof}
Through $\iota_a$, every linear identity among Schur functions becomes an identity among constant-entry tableaux. For example, each skew Schur function can be written as a unique linear combination of non-skew Schur functions. The coefficients are given by the Littlewood--Richardson rule. Applying $\iota_a$ and using \cref{Thm:iota}, we can extend this into $\Yd$. 
\begin{corollary}[Littlewood--Richardson rule] \label{cor:skewtononskew}
 Let $a\in\A$ and $\lambda/\mu$ be a skew shape with $n=|\lambda|-|\mu|$. Then, in $\Yd$, one has
\begin{align*}
    \{a\}^{\lambda/\mu} = \sum_{\nu\vdash n} c^\lambda_{\mu,\nu}\,\{a\}^{\nu},
\end{align*}
where $c^\lambda_{\mu,\nu}$ are the Littlewood--Richardson coefficients.
\end{corollary}
We now give explicit expansions of constant-entry tableaux. The key tool is the Newton identity in the quasi-shuffle algebra $\KA$. It holds over any commutative ring. In characteristic zero, it integrates to the exponential identity of Hoffman--Ihara~\cite{HI}.

\begin{lemma}[Newton identity, cf.\ \cite{HI}]\label{expHIresult} Let $\A$ be an alphabet with a commutative associative product $\diamond$ on $\K\A$, and let $z\in\A$. Write $z^{\,n}$ for the $n$-fold concatenation of $z$ (with $z^{\,0}=\mathbf 1$).
Then in $(\KA,\ast_\diamond)$, for every $n\ge1$,
\begin{align}\label{eq:newton}
    n\,z^{\,n}  =\sum_{i=1}^{n}(-1)^{i-1}\,z^{\diamond i}\ast_\diamond z^{\,n-i}.
\end{align}
Moreover, if $\Q\subseteq\K$, then \eqref{eq:newton} integrates to the exponential form of \textnormal{\cite[Cor.~5.1]{HI}}:
\begin{align}\label{eq:newton-exp}
    \exp_{\ast_\diamond}\!\left(\sum_{i\geq1}\frac{(-1)^{i-1}}{i}\, z^{\diamond i}\,X^i\right)  =\sum_{n=0}^{\infty}z^{\,n}\,X^n .
\end{align}
\end{lemma}
\begin{proof}
Set $E(X):=\sum_{n\ge0}z^{\,n}X^n$ and $P(X):=\sum_{i\ge1}(-1)^{i-1}z^{\diamond i}X^{i}$ in $\KA[[X]]$. We first prove \eqref{eq:newton} over $\Q$, then descend to $\Z$ and base-change. Applying \cite[Thm.~5.1]{HI} with $f=\log(1+t)$ gives $ X\,E'(X)=P(X)\ast_\diamond E(X)$ in $\Q\langle\A\rangle[[X]]$. Comparing the coefficient of $X^n$ yields \eqref{eq:newton} in $\Q\langle\A\rangle$.

For a general commutative ring $\K$, we pass to a universal setting. Let $B=\{u_i\mid i\geq1\}$ be an alphabet and define $u_i\diamond u_{i'}:=u_{i+i'}$ on $\Q B$. Then $u_1^{\diamond i}=u_i$, so both sides of \eqref{eq:newton} for the letter $u_1$ are $\Z$-linear combinations of words in $B$. The words in $B$ form a basis of $\Q\langle B\rangle$, and the identity holds in $\Q\langle B\rangle$ by the first paragraph. Hence it holds in $\Z\langle B\rangle$, and after the base change $\Z\to\K$ it holds in $\K\langle B\rangle$ with the same product $\diamond$. Finally, define a $\K$-linear map $\K B\to\K\A$, $u_i\mapsto z^{\diamond i}$. This map respects the product $\diamond$. By the recursion \eqref{eq:quasishuffle}, the induced map on words $\K\langle B\rangle\to\KA$ is a homomorphism for the quasi-shuffle products. Applying it to the identity in $\K\langle B\rangle$ gives \eqref{eq:newton} in $\KA$.

Finally, if $\Q\subseteq\K$, then each $1/i$ is invertible, and the recursion \eqref{eq:newton} is solved by
\begin{align*}
    E(X)=\exp_{\ast_\diamond}\!\left(\sum_{i\ge1}\frac{(-1)^{i-1}}{i}\,z^{\diamond i}X^i\right),
\end{align*}
which is \eqref{eq:newton-exp}.
\end{proof}

The map $L_\diamond$ is a $\ast$-algebra homomorphism. Moreover, for any $a\in\A$ we have $L_\diamond(\{a\}^{(1^n)})=a^{\,n}$ (the $n$-fold concatenation) and $L_\diamond(\{a^{\diamond i}\}^{(1)})=a^{\diamond i}$. Therefore \eqref{eq:newton}, applied to the letter $z=a$, pulls back to $\Yd$. It expresses a column tableau through shorter columns and depth-one tableaux.

\begin{corollary}\label{cor:expYTresult}
For a countable set $\A$ and any $a\in\A$, the column tableaux $\{a\}^{(1^n)}$ satisfy, for every $n\ge1$,
\begin{align}\label{eq:newtonYT}
    n\,\{a\}^{(1^n)} =\sum_{i=1}^{n}(-1)^{i-1}\,\{a^{\diamond i}\}^{(1)}\ast\{a\}^{(1^{\,n-i})} \qquad\text{in }\Yd.
\end{align}
If $\Q\subseteq\K$, this integrates to
\begin{align}\label{eq:newtonYT-exp}
    \exp_{\ast}\!\left(\sum_{i\geq1}\frac{(-1)^{i-1}}{i}\, \{a^{\diamond i}\}^{(1)}\,X^i\right) =1+\sum_{n=1}^{\infty}\,\{a\}^{(1^n)}\,X^{n}.
\end{align}
\end{corollary}

\begin{proof}
Apply $\overline L_\diamond^{-1}$ to \eqref{eq:newton} (resp.\,\eqref{eq:newton-exp}) with $z=a$, using that $\overline L_\diamond:(\Yd,\ast)\xrightarrow{\sim}(\KA,\ast_\diamond)$ is an algebra isomorphism by \cref{prop:L-alghom}.
\end{proof}
Next we turn to the power sums. For $j\in\N$, let $p_j=\sum_i x_i^{j}$ be the $j$-th power sum, and set $p_{\lambda}=p_{\lambda_1}p_{\lambda_2}\cdots$ for a partition $\lambda=(\lambda_1,\lambda_2,\dots)$. The family $\{p_{\lambda}\}$ is a $\Q$-basis of $\Lambda_{\Q}$.
\begin{corollary}\label{cor:powersum}
Let $a\in\A$. For any partition $\lambda=(\lambda_1,\dots,\lambda_n)$, we have
    \begin{align*}
        \iota_a(p_{\lambda}) =\{a^{\diamond\lambda_1}\}^{(1)}\ast\cdots\ast \{a^{\diamond\lambda_n}\}^{(1)}.
    \end{align*}
   Assume $(\K\A,\diamond)$ is freely generated and let $a$ be a generator. In terms of the $\diamond$-homomorphisms $\{f_m\}$ of \cref{prop:SymmSchur}, this gives
    \begin{align*}
        \lim_{M\to\infty}\prod_{i=1}^n F_{\Z_M^+}\bigl(\{a^{\diamond\lambda_i}\}^{(1)}\bigr)=p_{\lambda}.
    \end{align*}
\end{corollary}
\begin{proof}
Newton's identity $n\,e_n=\sum_{i=1}^{n}(-1)^{i-1}p_i\,e_{n-i}$ holds in $\Lambda_{\K}$. Applying the algebra homomorphism $\iota_a$ and using $\iota_a(e_m)=\iota_a(s_{(1^m)})=\{a\}^{(1^m)}$ gives
\begin{align*}
    n\,\{a\}^{(1^n)}=\sum_{i=1}^{n}(-1)^{i-1}\,\iota_a(p_i)\ast\{a\}^{(1^{n-i})}.
\end{align*}
We compare this with \eqref{eq:newtonYT} by induction on $n$. For $n=1$ both identities give $\iota_a(p_1)=\{a\}^{(1)}=\{a^{\diamond 1}\}^{(1)}$. For $n\geq2$, subtracting the two identities and using the induction hypothesis leaves $(-1)^{n-1}\bigl(\iota_a(p_n)-\{a^{\diamond n}\}^{(1)}\bigr)=0$. Hence $\iota_a(p_n)=\{a^{\diamond n}\}^{(1)}$ for all $n\geq1$, and the first identity follows from the multiplicativity of $\iota_a$. For the second identity, $f_m(a^{\diamond\lambda_i})=f_m(a)^{\lambda_i}=x_m^{\lambda_i}$, so $F_{\Z_M^+}\bigl(\{a^{\diamond\lambda_i}\}^{(1)}\bigr)=\sum_{m=1}^{M}x_m^{\lambda_i}$, and the product converges coefficientwise to $p_\lambda$.
\end{proof}

Identity \eqref{eq:newtonYT} solves each column $\{a\}^{(1^n)}$ recursively as a polynomial in the depth-one tableaux $\{a^{\diamond l}\}^{(1)}$. The only divisions are by $1,\dots,n$. Combining this with the Jacobi--Trudi formula (\cref{Thm:JTf-homo&LdJ}), which writes a diagonal tableau as a determinant of column tableaux, we obtain the following.

\begin{theorem}\label{Thm:YTpoly}
Let $\{a\}^{\lambda/\mu}\in\Yd$ be a constant-entry tableau of shape $\lambda/\mu$ with $n=|\lambda|-|\mu|$. Assume that $n!$ is invertible in $\K$. Then
\begin{align*}
    \{a\}^{\lambda/\mu}\in\Z[\tfrac1{n!}]\bigl[\{a^{\diamond l}\}^{(1)}\mid l\ge1\bigr],
\end{align*}
the subalgebra of $\Yd$ generated over $\Z[\tfrac1{n!}]$ by the depth-one tableaux under $\ast$. In particular, if $\Q\subseteq\K$, then $\{a\}^{\lambda/\mu}\in\Q[\{a^{\diamond l}\}^{(1)}\mid l\ge1]$.
\end{theorem}

\begin{proof}
By the Jacobi--Trudi formula (\cref{Thm:JTf-homo&LdJ}), the element $\{a\}^{\lambda/\mu}$ equals a determinant in $\Yd$. Its entries are column tableaux $\{a\}^{(1^m)}$, where $\{a\}^{(1^0)}=\mathbf 1$ and $\{a\}^{(1^m)}=0$ for $m<0$. Expanding it gives a $\Z$-linear combination of $\ast$-products of such columns. In the product corresponding to a permutation $\sigma$ of $\{1,\dots,\lambda_1\}$, the sizes of the column factors are $\lambda'_{\sigma(j)}-\mu'_j-\sigma(j)+j$ for $1\leq j\leq\lambda_1$, and they sum to $\sum_i\lambda'_i-\sum_j\mu'_j=n$ independently of $\sigma$. Hence in every nonzero product each factor $\{a\}^{(1^m)}$ satisfies $0\leq m\leq n$. By \cref{cor:expYTresult}, induction on $m$ writes each such $\{a\}^{(1^m)}$ as a polynomial in the depth-one tableaux $\{a^{\diamond l}\}^{(1)}$ ($1\le l\le m$) with coefficients in $\Z[\tfrac1{m!}]\subseteq\Z[\tfrac1{n!}]$: the $m$-th step of \eqref{eq:newtonYT} divides only by $m$, and $\{a\}^{(1^{m-i})}$ already lies in $\Z[\tfrac1{(m-i)!}]$ by hypothesis. Substituting these expansions yields the claim. When $\Q\subseteq\K$, all $1/n!$ are invertible, and the subalgebra is the $\Q$-subalgebra generated by the depth-one tableaux.
\end{proof}
Note that \cref{Thm:YTpoly} and \cref{cor:expYTresult} hold for an arbitrary letter $a$; they do not require $(\K\A,\diamond)$ to be freely generated.
\begin{remark}
Let $a$ be a generator of the freely generated algebra $(\K\A,\diamond)$. Then \cref{Thm:YTpoly} is the expansion of the Schur function $s_{\lambda/\mu}$ in the power sums, under the embedding $\iota_a$. This is why we need $\Z[\tfrac1{n!}]$ in weight $n$. The power sums are only a $\Q$-basis of $\Lambda$. For example, in the expansion of $\{a\}^{(n)}$, the monomial $ \bigl(\{a^{\diamond1}\}^{(1)}\bigr)^{\ast n}$ appears with coefficient $1/n!$. The expansion in non-skew tableaux has no such denominators. For any letter $a$, we have $\{a\}^{\lambda/\mu}=\sum_{\nu\vdash n}c^\lambda_{\mu,\nu}\{a\}^{\nu}$ in $\Yd$. The coefficients $c^\lambda_{\mu,\nu}$ are integers.
\end{remark}

\begin{example}\label{example:poly}
    In the following formulas, we use the explicit description given in \cref{Thm:YTpoly} with identical entries $a=2$. We shall present the formulas for Young tableaux, together with the corresponding case for Schur multiple zeta values.
    \begin{enumerate}[(i)]
        \item  In the smallest case $n=1$, $\{2\}^{(1)}={ \ytableausetup{centertableaux, boxsize=1.3em}        \begin{ytableau}
             2
            \end{ytableau} }$, there is just one choice for $h_1l_1+\dots+h_rl_r=1$.
            \begin{align*}
                 { \ytableausetup{centertableaux, boxsize=1.3em} \begin{ytableau}
                    2
                \end{ytableau} }={ \ytableausetup{centertableaux, boxsize=1.3em} \begin{ytableau}
                    2
                 \end{ytableau} }.
            \end{align*}
        \item For $n=2$, we have two cases { \ytableausetup{centertableaux, boxsize=1.3em}  \begin{ytableau}
             2&2
            \end{ytableau} } and { \ytableausetup{centertableaux, boxsize=1.3em}   \begin{ytableau}
             2\\
             2
            \end{ytableau} }. There are only two possible factors ${ \ytableausetup{centertableaux, boxsize=1.3em} \begin{ytableau}
             2
            \end{ytableau} }^2$ and ${ \ytableausetup{centertableaux, boxsize=1.3em}   \begin{ytableau}
             4
            \end{ytableau} }$. Therefore, we have
            \begin{align*}
                { \ytableausetup{centertableaux, boxsize=1.3em} \begin{ytableau}
             2&2
            \end{ytableau} }=\frac{1}{2}{ \ytableausetup{centertableaux, boxsize=1.3em}                 \begin{ytableau}
             2
            \end{ytableau} }^2+\frac{1}{2}{ \ytableausetup{centertableaux, boxsize=1.3em} \begin{ytableau}
             4
            \end{ytableau} }\quad\text{and}\quad{ \ytableausetup{centertableaux, boxsize=1.3em} \begin{ytableau}
             2\\
             2
            \end{ytableau} }=\frac{1}{2}{ \ytableausetup{centertableaux, boxsize=1.3em} \begin{ytableau}
             2
            \end{ytableau} }^2-\frac{1}{2}{ \ytableausetup{centertableaux, boxsize=1.3em}   \begin{ytableau}
             4
            \end{ytableau} }.
            \end{align*}
        \item For $n=4$, the possible factors are the following: ${ \ytableausetup{centertableaux, boxsize=1.3em} \begin{ytableau}
             2
            \end{ytableau} }^4$, ${ \ytableausetup{centertableaux, boxsize=1.3em} \begin{ytableau}
             2
            \end{ytableau} }^2\ast{ \ytableausetup{centertableaux, boxsize=1.3em} \begin{ytableau}
             4
            \end{ytableau} }$, ${ \ytableausetup{centertableaux, boxsize=1.3em} \begin{ytableau}
             2
            \end{ytableau} }\ast{ \ytableausetup{centertableaux, boxsize=1.3em} \begin{ytableau}
             6
            \end{ytableau} }$, ${ \ytableausetup{centertableaux, boxsize=1.3em} \begin{ytableau}
             4
            \end{ytableau} }^2$, and ${\ytableausetup{centertableaux, boxsize=1.3em} \begin{ytableau}
             8
            \end{ytableau} }$.
            We have the following formulas:
            \begin{align*}
            &{\ytableausetup{centertableaux, boxsize=1.3em} \begin{ytableau}
             2&2&2&2
            \end{ytableau} }=\frac{1}{24}{ \ytableausetup{centertableaux, boxsize=1.3em} \begin{ytableau}
             2
            \end{ytableau} }^4+\frac{1}{4}{ \ytableausetup{centertableaux, boxsize=1.3em} \begin{ytableau}
             2
            \end{ytableau} }^2\ast{ \ytableausetup{centertableaux, boxsize=1.3em} \begin{ytableau}
             4
            \end{ytableau} }+\frac{1}{8}{ \ytableausetup{centertableaux, boxsize=1.3em} \begin{ytableau}
             4
            \end{ytableau} }^2+\frac{1}{3}{ \ytableausetup{centertableaux, boxsize=1.3em} \begin{ytableau}
             2
            \end{ytableau} }\ast{ \ytableausetup{centertableaux, boxsize=1.3em} \begin{ytableau}
             6
            \end{ytableau} }+\frac{1}{4}{\ytableausetup{centertableaux, boxsize=1.3em} \begin{ytableau}
             8
            \end{ytableau} },\\
                &{ \ytableausetup{centertableaux, boxsize=1.3em} \begin{ytableau}
             2&2&2\\
             2
            \end{ytableau} }=\frac{1}{8}{ \ytableausetup{centertableaux, boxsize=1.3em} \begin{ytableau}
             2
            \end{ytableau} }^4+\frac{1}{4}{ \ytableausetup{centertableaux, boxsize=1.3em} \begin{ytableau}
             2
            \end{ytableau} }^2\ast{ \ytableausetup{centertableaux, boxsize=1.3em} \begin{ytableau}
             4
            \end{ytableau} }-\frac{1}{8}{ \ytableausetup{centertableaux, boxsize=1.3em} \begin{ytableau}
             4
            \end{ytableau} }^2-\frac{1}{4}{\ytableausetup{centertableaux, boxsize=1.3em} \begin{ytableau}
             8
            \end{ytableau} },\\
            &{\ytableausetup{centertableaux, boxsize=1.3em} \begin{ytableau}
             2&2\\
             2&2
            \end{ytableau} }=\frac{1}{12}{ \ytableausetup{centertableaux, boxsize=1.3em} \begin{ytableau}
             2
            \end{ytableau} }^4+\frac{1}{4}{ \ytableausetup{centertableaux, boxsize=1.3em} \begin{ytableau}
             4
            \end{ytableau} }^2-\frac{1}{3}{ \ytableausetup{centertableaux, boxsize=1.3em} \begin{ytableau}
             2
            \end{ytableau} }\ast{ \ytableausetup{centertableaux, boxsize=1.3em} \begin{ytableau}
             6
            \end{ytableau} },\\
             &{\ytableausetup{centertableaux, boxsize=1.3em} \begin{ytableau}
             \none&2&2\\
             2&2
            \end{ytableau} }=\frac{5}{24}{ \ytableausetup{centertableaux, boxsize=1.3em} \begin{ytableau}
             2
            \end{ytableau} }^4+\frac{1}{4}{ \ytableausetup{centertableaux, boxsize=1.3em} \begin{ytableau}
             2
            \end{ytableau} }^2\ast{ \ytableausetup{centertableaux, boxsize=1.3em} \begin{ytableau}
             4
            \end{ytableau} }+\frac{1}{8}{ \ytableausetup{centertableaux, boxsize=1.3em} \begin{ytableau}
             4
            \end{ytableau} }^2-\frac{1}{3}{ \ytableausetup{centertableaux, boxsize=1.3em} \begin{ytableau}
             2
            \end{ytableau} }\ast{ \ytableausetup{centertableaux, boxsize=1.3em} \begin{ytableau}
             6
            \end{ytableau} }-\frac{1}{4}{\ytableausetup{centertableaux, boxsize=1.3em} \begin{ytableau}
             8
            \end{ytableau} },\\
             &{\ytableausetup{centertableaux, boxsize=1.3em} \begin{ytableau}
             \none&2\\
             2&2\\
             2
            \end{ytableau} }=\frac{5}{24}{ \ytableausetup{centertableaux, boxsize=1.3em} \begin{ytableau}
             2
            \end{ytableau} }^4-\frac{1}{4}{ \ytableausetup{centertableaux, boxsize=1.3em} \begin{ytableau}
             2
            \end{ytableau} }^2\ast{ \ytableausetup{centertableaux, boxsize=1.3em} \begin{ytableau}
             4
            \end{ytableau} }+\frac{1}{8}{ \ytableausetup{centertableaux, boxsize=1.3em} \begin{ytableau}
             4
            \end{ytableau} }^2-\frac{1}{3}{ \ytableausetup{centertableaux, boxsize=1.3em} \begin{ytableau}
             2
            \end{ytableau} }\ast{ \ytableausetup{centertableaux, boxsize=1.3em} \begin{ytableau}
             6
            \end{ytableau} }+\frac{1}{4}{\ytableausetup{centertableaux, boxsize=1.3em} \begin{ytableau}
             8
            \end{ytableau} },\\
           &{\ytableausetup{centertableaux, boxsize=1.3em} \begin{ytableau}
             2&2\\
             2\\
             2
            \end{ytableau} }=\frac{1}{8}{ \ytableausetup{centertableaux, boxsize=1.3em} \begin{ytableau}
             2
            \end{ytableau} }^4-\frac{1}{4}{ \ytableausetup{centertableaux, boxsize=1.3em} \begin{ytableau}
             2
            \end{ytableau} }^2\ast{ \ytableausetup{centertableaux, boxsize=1.3em} \begin{ytableau}
             4
            \end{ytableau} }-\frac{1}{8}{ \ytableausetup{centertableaux, boxsize=1.3em} \begin{ytableau}
             4
            \end{ytableau} }^2+\frac{1}{4}{\ytableausetup{centertableaux, boxsize=1.3em} \begin{ytableau}
             8
            \end{ytableau} },\\
            &{\ytableausetup{centertableaux, boxsize=1.3em} \begin{ytableau}
             2\\
             2\\
             2\\
             2
            \end{ytableau} }=\frac{1}{24}{ \ytableausetup{centertableaux, boxsize=1.3em} \begin{ytableau}
             2
            \end{ytableau} }^4-\frac{1}{4}{ \ytableausetup{centertableaux, boxsize=1.3em} \begin{ytableau}
             2
            \end{ytableau} }^2\ast{ \ytableausetup{centertableaux, boxsize=1.3em} \begin{ytableau}
             4
            \end{ytableau} }+\frac{1}{8}{ \ytableausetup{centertableaux, boxsize=1.3em} \begin{ytableau}
             4
            \end{ytableau} }^2+\frac{1}{3}{ \ytableausetup{centertableaux, boxsize=1.3em} \begin{ytableau}
             2
            \end{ytableau} }\ast{ \ytableausetup{centertableaux, boxsize=1.3em} \begin{ytableau}
             6
            \end{ytableau} }-\frac{1}{4}{\ytableausetup{centertableaux, boxsize=1.3em} \begin{ytableau}
             8
            \end{ytableau} }.
            \end{align*}
    \end{enumerate}
\end{example}

By applying the linearization map $L_{\diamond}$, any Young tableau $\mathbf{h}\in\Y$ becomes a $\Z$-linear combination of column Young tableaux in $\Yd$. This provides a fundamental method to write a skew Young tableau as a $\Z$-linear combination of non-skew Young tableaux in $\Yd$. However, under this procedure, the entries of the resulting tableaux usually become quite complicated.

We therefore ask whether there exist special types of skew Young tableaux for which such a decomposition can be written in a simpler and nicer form. The answer is yes. In the most special case where all entries are equal, any skew Young tableau $\mathbf{h}$ with all entries equal to $a$ can indeed be expressed as a linear combination of non-skew Young tableaux whose entries are all $a$.
\section{Further Examples}\label{sec:Example}
This section illustrates the scope of the Young tableaux algebra and of multiple Schur series. We discuss five classes of Young tableaux algebras: Schur multiple zeta values, Schur sum of monotangent functions, Schur multiple Eisenstein series, the $q$-analogue of Schur multiple zeta values, and the Schur multiple $L$-values.  In each case, we will explain how \cref{def:MSS} can be applied. One important application is the Fourier expansion of the Schur multiple Eisenstein series. \cref{prop.MSSconv} relates this expansion to the Schur multiple zeta values and the Schur sum of monotangent functions. We also explain the application of \cref{Thm:YTpoly} in each case.

In the first examples below, we use the basic setup $\K=\Q$, $\A=\N$, and $\diamond=+$; later subsections will modify this setup when needed. We write $\mathcal{H}^1=\Q\boxed{\N}$ for the algebra generated by shifted connected Young tableaux with positive entries. $\mathcal{H}^0$ is the subalgebra of $\mathcal{H}^1$ spanned by the Young tableaux $\mathbf{h}=(h_{ij})$ with
\begin{align*}
            \begin{cases}
                \text{$h_{ij}\ge 1$ for all $(i,j)\in D(\Sh(\mathbf{h})) \setminus C(\Sh(\mathbf{h}))$ }, \\[3pt]
                 \text{$h_{ij}\ge2$ for all $(i,j)\in C(\Sh(\mathbf{h}))$}.
            \end{cases}
        \end{align*}
       Here $C(\lambda/\mu):=\{(i,j)\in D(\lambda/\mu)\mid (i+1,j)\notin D(\lambda/\mu),\ (i,j+1)\notin D(\lambda/\mu)\}$ is the set of all corners of the shape $\lambda/\mu$. This span is closed under $\ast$, since the corners of a disjoint union of shapes are the corners of its parts. Equivalently, $\mathcal{H}^0$ is generated by the shifted connected Young tableaux whose corner entries are at least $2$.
   For the sub-alphabet $\N_{\geq2}$, the subalgebra $\mathcal{H}^2:=\Q\boxed{\N_{\geq2}}\subset\mathcal{H}^1$ is generated by shifted connected Young tableaux with entries $k\geq2$. Let $\Hd^i$ denote the linearized quotient corresponding to $\mathcal{H}^i$, given by \eqref{eq:defKAdia}, for $i=0, 1, 2$.
\subsection{Schur multiple zeta values}\label{sec:examSMZV}
We first consider Schur multiple zeta values. These are values of the Schur multiple zeta function at integer points, introduced in \cite{NPY}. They can be elegantly redefined using our multiple Schur series framework.

\begin{definition}[Truncated Schur multiple zeta values]
    Let $\X=\Z^+_M=\{1,\dots,M\}$ with the standard ordering. Set $\A=\N$, and $\calP=\R$. Consider the family of $\K$-algebra homomorphisms $\{f_m\}_{m\in\Z^+_M}$ given by
    \begin{align*}
        f_m:\K\N&\to\R\\
        h&\mapsto\frac{1}{m^h}.
    \end{align*}
    The multiple Schur series $\zeta_M(\mathbf{h}):=F_{\Z^+_M}(\mathbf{h})$ associated with this family is called the \textit{truncated Schur multiple zeta value}.
\end{definition}

Let $\mathbf{h}\in\Hd^0$ be a Young tableau. Nakasuji, Phuksuwan, and Yamasaki showed in \cite{NPY} that the limit as $M\to\infty$ exists and the sum converges absolutely. Therefore, we can naturally define the Schur multiple zeta value as this limit:
\begin{align*}
    \zeta(\mathbf{h}) := \lim_{M\to\infty} \zeta_M(\mathbf{h})=\sum_{(m_{i,j})\in\SSYT(\Sh(\mathbf{h}),\N)}\prod_{(i,j)\in D(\Sh(\mathbf{h}))}\frac{1}{m_{i,j}^{h_{i,j}}}.
\end{align*}
By \cref{lem:F} and the absolute convergence of the series on $\Hd^0$, the evaluation map $\zeta: \Hd^0 \to \R$ is a well-defined $\Q$-algebra homomorphism.

Let $\Hd^{2,\operatorname{diag}}\subset\Hd^2\subset\Hd^0$ be the subalgebra generated by diagonal constant Young tableaux, i.e., those whose entries are constant along each diagonal. Applying the homomorphism $\zeta$ to the Jacobi--Trudi formula for the Young tableaux algebra (\cref{Thm:JTf-homo&LdJ}) recovers the Jacobi--Trudi formula for Schur multiple zeta values. This was originally proved as Theorem 1.1 (2) in \cite{NPY}.

\begin{corollary}[{\cite[Theorem 1.1 (2)]{NPY}}]
    For any diagonal constant Young tableau $\mathbf{h} \in \Hd^{2,\operatorname{diag}}$ with shape $\lambda/\mu$ and diagonal entries $a_n$ ($n \in \Z$), we have
    \begin{align*}
        \zeta(\mathbf{h}) = \det\left[ \zeta\left(a_{-\mu'_j+j-1},\cdots, a_{-\mu'_j+j-(\lambda'_i-\mu'_j-i+j)}\right) \right]_{1\leq i,j\leq s},
    \end{align*}
    where $s=\lambda_1$ and the elements in the determinant are evaluated as multiple zeta values through the linearization map $L_{\diamond}$.
\end{corollary}

By \cref{Thm:YTpoly}, any skew Young tableau with identical entries can be expressed as a polynomial in depth-one tableaux. Applying the evaluation homomorphism $\zeta$ to this result yields the expansion formulas for Schur multiple zeta values, as in \cref{example:poly}. This generalizes the explicit relations presented in \cite[Example 2.9]{NPY}. Specifically, take the constant entry $a=h \ge 2$. Then the Schur multiple zeta value of shape $\lambda/\mu$ is a polynomial in the Riemann zeta values $\zeta(h), \zeta(2h), \dots, \zeta(nh)$. \cref{Thm:YTpoly} provides an algebraic proof of this reduction.

\begin{example}[{\cite[Example 2.9]{NPY}}]
    For Schur multiple zeta values, we have
    \begin{align*}
&\zeta\left({\ytableausetup{centertableaux, boxsize=1.3em} \begin{ytableau}
 2&2&2&2
\end{ytableau} }\right)
= \frac{1}{24}\,\zeta\left({ \ytableausetup{centertableaux, boxsize=1.3em} \begin{ytableau}
 2
\end{ytableau} }\right)^4
+ \frac{1}{4}\,\zeta\left({ \ytableausetup{centertableaux, boxsize=1.3em} \begin{ytableau}
 2
\end{ytableau} }\right)^2  \zeta\left({ \ytableausetup{centertableaux, boxsize=1.3em} \begin{ytableau}
 4
\end{ytableau} }\right)
+ \frac{1}{8}\,\zeta\left({ \ytableausetup{centertableaux, boxsize=1.3em} \begin{ytableau}
 4
\end{ytableau} }\right)^2 + \frac{1}{3}\,\zeta\left({ \ytableausetup{centertableaux, boxsize=1.3em} \begin{ytableau}
 2
\end{ytableau} }\right)  \zeta\left({ \ytableausetup{centertableaux, boxsize=1.3em} \begin{ytableau}
 6
\end{ytableau} }\right)
+ \frac{1}{4}\,\zeta\left({\ytableausetup{centertableaux, boxsize=1.3em} \begin{ytableau}
 8
\end{ytableau} }\right),\\[0.5em]
&\zeta\left({ \ytableausetup{centertableaux, boxsize=1.3em} \begin{ytableau}
 2&2&2\\
 2
\end{ytableau} }\right)
= \frac{1}{8}\,\zeta\left({ \ytableausetup{centertableaux, boxsize=1.3em} \begin{ytableau}
 2
\end{ytableau} }\right)^4
+ \frac{1}{4}\,\zeta\left({ \ytableausetup{centertableaux, boxsize=1.3em} \begin{ytableau}
 2
\end{ytableau} }\right)^2  \zeta\left({ \ytableausetup{centertableaux, boxsize=1.3em} \begin{ytableau}
 4
\end{ytableau} }\right)
- \frac{1}{8}\,\zeta\left({ \ytableausetup{centertableaux, boxsize=1.3em} \begin{ytableau}
 4
\end{ytableau} }\right)^2
- \frac{1}{4}\,\zeta\left({\ytableausetup{centertableaux, boxsize=1.3em} \begin{ytableau}
 8
\end{ytableau} }\right),\\[0.5em]
&\zeta\left({\ytableausetup{centertableaux, boxsize=1.3em} \begin{ytableau}
 2&2\\
 2&2
\end{ytableau} }\right)
= \frac{1}{12}\,\zeta\left({ \ytableausetup{centertableaux, boxsize=1.3em} \begin{ytableau}
 2
\end{ytableau} }\right)^4
+ \frac{1}{4}\,\zeta\left({ \ytableausetup{centertableaux, boxsize=1.3em} \begin{ytableau}
 4
\end{ytableau} }\right)^2
- \frac{1}{3}\,\zeta\left({ \ytableausetup{centertableaux, boxsize=1.3em} \begin{ytableau}
 2
\end{ytableau} }\right)  \zeta\left({ \ytableausetup{centertableaux, boxsize=1.3em} \begin{ytableau}
 6
\end{ytableau} }\right),\\[0.5em]
&\zeta\left({\ytableausetup{centertableaux, boxsize=1.3em} \begin{ytableau}
 \none&2&2\\
 2&2
\end{ytableau} }\right)
= \frac{5}{24}\,\zeta\left({ \ytableausetup{centertableaux, boxsize=1.3em} \begin{ytableau}
 2
\end{ytableau} }\right)^4
+ \frac{1}{4}\,\zeta\left({ \ytableausetup{centertableaux, boxsize=1.3em} \begin{ytableau}
 2
\end{ytableau} }\right)^2  \zeta\left({ \ytableausetup{centertableaux, boxsize=1.3em} \begin{ytableau}
 4
\end{ytableau} }\right)
+ \frac{1}{8}\,\zeta\left({ \ytableausetup{centertableaux, boxsize=1.3em} \begin{ytableau}
 4
\end{ytableau} }\right)^2
- \frac{1}{3}\,\zeta\left({ \ytableausetup{centertableaux, boxsize=1.3em} \begin{ytableau}
 2
\end{ytableau} }\right)  \zeta\left({ \ytableausetup{centertableaux, boxsize=1.3em} \begin{ytableau}
 6
\end{ytableau} }\right)
- \frac{1}{4}\,\zeta\left({\ytableausetup{centertableaux, boxsize=1.3em} \begin{ytableau}
 8
\end{ytableau} }\right),\\[0.5em]
&\zeta\left({\ytableausetup{centertableaux, boxsize=1.3em} \begin{ytableau}
 \none&2\\
 2&2\\
 2
\end{ytableau} }\right)
= \frac{5}{24}\,\zeta\left({ \ytableausetup{centertableaux, boxsize=1.3em} \begin{ytableau}
 2
\end{ytableau} }\right)^4
- \frac{1}{4}\,\zeta\left({ \ytableausetup{centertableaux, boxsize=1.3em} \begin{ytableau}
 2
\end{ytableau} }\right)^2  \zeta\left({ \ytableausetup{centertableaux, boxsize=1.3em} \begin{ytableau}
 4
\end{ytableau} }\right)
+ \frac{1}{8}\,\zeta\left({ \ytableausetup{centertableaux, boxsize=1.3em} \begin{ytableau}
 4
\end{ytableau} }\right)^2 
- \frac{1}{3}\,\zeta\left({ \ytableausetup{centertableaux, boxsize=1.3em} \begin{ytableau}
 2
\end{ytableau} }\right)  \zeta\left({ \ytableausetup{centertableaux, boxsize=1.3em} \begin{ytableau}
 6
\end{ytableau} }\right)
+ \frac{1}{4}\,\zeta\left({\ytableausetup{centertableaux, boxsize=1.3em} \begin{ytableau}
 8
\end{ytableau} }\right),\\[0.5em]
&\zeta\left({\ytableausetup{centertableaux, boxsize=1.3em} \begin{ytableau}
 2&2\\
 2\\
 2
\end{ytableau} }\right)
= \frac{1}{8}\,\zeta\left({ \ytableausetup{centertableaux, boxsize=1.3em} \begin{ytableau}
 2
\end{ytableau} }\right)^4
- \frac{1}{4}\,\zeta\left({ \ytableausetup{centertableaux, boxsize=1.3em} \begin{ytableau}
 2
\end{ytableau} }\right)^2  \zeta\left({ \ytableausetup{centertableaux, boxsize=1.3em} \begin{ytableau}
 4
\end{ytableau} }\right)
- \frac{1}{8}\,\zeta\left({ \ytableausetup{centertableaux, boxsize=1.3em} \begin{ytableau}
 4
\end{ytableau} }\right)^2
+ \frac{1}{4}\,\zeta\left({\ytableausetup{centertableaux, boxsize=1.3em} \begin{ytableau}
 8
\end{ytableau} }\right),\\[0.5em]
&\zeta\left({\ytableausetup{centertableaux, boxsize=1.3em} \begin{ytableau}
 2\\
 2\\
 2\\
 2
\end{ytableau} }\right)
= \frac{1}{24}\,\zeta\left({ \ytableausetup{centertableaux, boxsize=1.3em} \begin{ytableau}
 2
\end{ytableau} }\right)^4
- \frac{1}{4}\,\zeta\left({ \ytableausetup{centertableaux, boxsize=1.3em} \begin{ytableau}
 2
\end{ytableau} }\right)^2  \zeta\left({ \ytableausetup{centertableaux, boxsize=1.3em} \begin{ytableau}
 4
\end{ytableau} }\right)
+ \frac{1}{8}\,\zeta\left({ \ytableausetup{centertableaux, boxsize=1.3em} \begin{ytableau}
 4
\end{ytableau} }\right)^2
+ \frac{1}{3}\,\zeta\left({ \ytableausetup{centertableaux, boxsize=1.3em} \begin{ytableau}
 2
\end{ytableau} }\right)  \zeta\left({ \ytableausetup{centertableaux, boxsize=1.3em} \begin{ytableau}
 6
\end{ytableau} }\right)
-\frac{1}{4}\,\zeta\left({\ytableausetup{centertableaux, boxsize=1.3em} \begin{ytableau}
 8
\end{ytableau} }\right).
\end{align*}
\end{example}

\subsection{Schur sum of monotangent functions} \label{sec:examSMonoF}
We start with the Fourier expansion of Eisenstein series. For even $h\ge2$, the series is defined by
\begin{align*}
    \G(h;\tau)=\frac{1}{2}\sum_{\substack{(m,n)\in\Z^2\\(m,n)\neq(0,0)}}\frac{1}{(m\tau+n)^h}=\frac{1}{2}\sum_{n\neq0}\frac{1}{n^h}+\sum_{m=1}^{\infty}\Psi(h;m\tau),
\end{align*}
where for $h=2$ we use the Eisenstein summation. Here $\Psi(h;x)$ is the monotangent function given by
\begin{align*}
    \Psi(h;x)=\sum_{d\in\Z}\frac{1}{(x+d)^h}.
\end{align*}
In \cite{Bou}, O.~Bouillot introduced and studied a multiple version of the monotangent function $\Psi(h_1,\dots,h_r;x)$, called the multitangent function, appearing in the calculation of the Fourier expansion of classical Eisenstein series, which is given by
\begin{align*}
    \Psi(h_1,\dots,h_r;x)=\sum_{\substack{n_1<n_2<\cdots<n_r\\n_i\in\Z}}\frac{1}{(x+n_1)^{h_1}\cdots(x+n_r)^{h_r}},
\end{align*}
for $h_1,\dots,h_r\geq2$.  One of the main results of \cite{Bou} is the reduction of multitangent functions to monotangent functions. Theorem~3 of \cite{Bou} gives an explicit formula for writing multitangent functions as a multiple-zeta-linear combination of monotangent functions. It is crucial for the calculation of the Fourier expansion of multiple Eisenstein series in \cite{B}. Similarly, we can naturally consider the Schur generalization $\Psi(\mathbf{h};x)$:
\begin{definition}
    Let $\mathbf{h}\in\Hd^1$. For $N>0$ and $n\in\Z_N=\{-N,\dots,N\}$, the truncated Schur multitangent function $\Psi_{N}(\mathbf{h};x)$ is defined by the family
     \begin{align*}
        f_n: \K\N &\rightarrow \mathcal{O}(\C\setminus\Z)\\
         h &\longmapsto \frac{1}{(n+x)^h},
    \end{align*}
    i.e.,
    \begin{align*}
        \Psi_{N}(\mathbf{h};x)=\sum_{n_{i,j}\in \SSYT(\Sh(\mathbf{h}),\Z_N) } \prod_{(i,j) \in D(\Sh(\mathbf{h}))} \frac{1}{( n_{i,j}+x)^{h_{i,j}}}.
    \end{align*}
    We set $\Psi_N(\varnothing;x)=1$. Let
\begin{align*}
    \widetilde C(\lambda/\mu):=\{(i,j)\in D(\lambda/\mu)\mid (i-1,j)\notin D(\lambda/\mu),\ (i,j-1)\notin D(\lambda/\mu)\}.
\end{align*}
 Then the following limit is absolutely convergent if $h_{i,j}\geq2$ for all $(i,j)\in C(\Sh(\mathbf{h}))\cup\widetilde C(\Sh(\mathbf{h}))$, which gives the Schur multitangent function $\Psi(\mathbf{h};x)$:
    \begin{align*}
        \Psi(\mathbf{h};x)=\lim_{N\to\infty}\Psi_{N}(\mathbf{h};x)=\sum_{n_{i,j}\in \SSYT(\Sh(\mathbf{h}),\Z) } \prod_{(i,j) \in D(\Sh(\mathbf{h}))} \frac{1}{( n_{i,j}+x)^{h_{i,j}}},
    \end{align*}
    and again we set $\Psi(\varnothing;x)=1$.
\end{definition}
Indeed, the entries in $C(\Sh(\mathbf{h}))$ control the variables tending to $+\infty$, and the entries in $\widetilde C(\Sh(\mathbf{h}))$ control the variables tending to $-\infty$. For a single column, this recovers the condition $h_1,h_r\geq2$ of \cite{Bou}. In particular, the condition holds for every $\mathbf{h}\in\Hd^2$. 

This function is discussed in \cite{Yu}. There, it is related to the Fourier expansion of the Schur multiple Eisenstein series, which will be introduced in the next subsection. We will also consider another Schur generalization of the monotangent function, the Schur sum of the monotangent function, denoted by $\check{g}$. 

To define these objects, set $\Z_M = \{-M,\dots,-1,0,1,\dots,M\}$. For $M,N\geq 1$ and $\tau \in \mathbb{H}$, consider the set $X^\tau_{M,N} = \Z_M\tau + \Z_N $. In this set, we consider the order $m_1 \tau + n_1 \prec m_2 \tau + n_2 :\Leftrightarrow m_1 < m_2 \vee (m_1=m_2 \wedge n_1<n_2)$. We obtain a finite ordered set $\mathcal{X}^\tau_{M,N} = (X^\tau_{M,N}, \prec)$. We also define the subset of ``positive'' lattice points by \[X^{\tau,>0}_{M,N} = \{ \lambda \in X^\tau_{M,N} \mid 0 \prec \lambda \}. \]  
\begin{definition} Set $X^{\tau,\mathbb{H}}_{M,N}:=\{\lambda=m\tau+n \in X^{\tau,>0}_{M,N}|m\geq1\}=X^{\tau,>0}_{M,N}\cap\mathbb{H}$. Let $\mathcal{X}^{\tau,\mathbb{H}}_{M,N}$ be the corresponding ordered set. Then for $\mathbf{h}\in\Hd^1$, 
    \begin{align*}
        \check{g}_{M,N}(\mathbf{h};\tau):=\sum_{(m_{i,j} \tau + n_{i,j}) \in \SSYT(\Sh(\mathbf{h}),\mathcal{X}^{\tau,\mathbb{H}}_{M,N}) } \prod_{(i,j) \in D(\Sh(\mathbf{h}))} \frac{1}{(m_{i,j} \tau + n_{i,j})^{h_{i,j}}}\,.
    \end{align*}
    For $\mathbf{h}\in\Hd^2$, the following limit exists:
    \begin{align*}
        \check{g}(\mathbf{h};\tau):=\lim_{M\to\infty}\lim_{N\to\infty}\check{g}_{M,N}(\mathbf{h};\tau).
    \end{align*}
\end{definition}
As an application of \cref{prop.MSSconv}, we have
\begin{corollary} \label{Lem:checkgconv}
    For $M,N>0$, we have
    \begin{align*}
        \check{g}_{M,N}(-;\tau)=\Psi_N(-;\tau)\star\Psi_N(-;2\tau)\star\cdots\star\Psi_N(-;M\tau).
    \end{align*}
\end{corollary}
After passing to the limits in the order specified above, and using the reduction of Schur multitangent functions studied in \cite{Yu}, this convolution decomposition yields a $q$-expansion of $\check g$ as a $\Q$-linear combination of multiple zeta values times the series $(2\pi i)^{|\mathbf{h}'|}g(\mathbf h')$ defined in \eqref{eq:gBbbk}.

\subsection{Schur multiple Eisenstein series}\label{sec:examSMES}
For the finite ordered set $\mathcal{X}^{\tau,>0}_{M,N} = (X^{\tau,>0}_{M,N}, \prec)$ we define the family $(f_{\lambda})_{\lambda\in\mathcal{X}^{\tau,>0}_{M,N}}$ by
\begin{align*}
    f_{m\tau+n}: \K\N &\rightarrow \C\\
    h &\longmapsto \frac{1}{(m\tau +n)^h},
\end{align*}
and define the {\emph{(truncated) Schur multiple Eisenstein series}} for $\mathbf{h}=\left({h_{i,j}}\right) \in \Hd^1$ by $\mathbb{G}_{M,N}(\mathbf{h};\tau)=F_{\mathcal{X}^{\tau,>0}_{M,N}}(\mathbf{h})$. In other words, these are given for $\mathbf{h} \in \Hd^1$ by
\begin{align}
    \mathbb{G}_{M,N}(\mathbf{h}; \tau) = \sum_{(m_{i,j} \tau + n_{i,j}) \in \SSYT(\Sh(\mathbf{h}),\mathcal{X}^{\tau,>0}_{M,N}) } \prod_{(i,j) \in D(\Sh(\mathbf{h}))} \frac{1}{(m_{i,j} \tau + n_{i,j})^{h_{i,j}}}\,.
    \label{eq:defGn}
\end{align}
One can then show that the following limit exists:
\begin{propdef}
    Let $\mathbf{h}\in\Hd^2$. The \emph{{Schur multiple Eisenstein series $\mathbb{G}(\mathbf{h}; \tau)$}} is defined by the following limit:
    \begin{align}
       \mathbb{G}(\mathbf{h}; \tau) := \lim_{M\rightarrow \infty}\lim_{N\rightarrow \infty}   \mathbb{G}_{M,N}(\mathbf{h}; \tau)\,.
       \label{eq:bfG}
    \end{align}
\end{propdef}
Without the limiting procedure, absolute convergence requires stronger conditions. Namely, one needs $h_{i,j}\geq3$ for $(i,j)\in C(\Sh(\mathbf{h}))$ and $h_{i,j}\geq2$ for $(i,j)\in D(\Sh(\mathbf{h}))\setminus C(\Sh(\mathbf{h}))$. The order of the limits in \eqref{eq:bfG} ensures convergence even in the case where  $h_{i,j}=2$ for $(i,j)\in C(\Sh(\mathbf{h}))$. The argument is similar to that in \cite{BT}. This approach aligns with the usual method for defining the quasimodular form $\G(2;\tau)$.

Let $\mathbb{H}:= \{ \tau \in \mathbb{C} \mid \Im(\tau) > 0 \}$ be the upper half-plane. For any $\mathbf{h}\in\Hd^2$, $\G(\mathbf{h};\tau)$ is a holomorphic function on the upper half-plane, i.e., $\G(\mathbf{h};\tau)\in\OH$.
\begin{proposition}
    $\G(-;\tau):\Hd^2\to\OH$ is a $\K$-algebra homomorphism; explicitly, for any $\mathbf{h}_1,\mathbf{h}_2\in\Hd^2$,
    \begin{align*}
        \G(\mathbf{h}_1;\tau)\G(\mathbf{h}_2;\tau)=\G(\mathbf{h}_1\ast\mathbf{h}_2;\tau).
    \end{align*}
\end{proposition}
\begin{proof}
From the definition, $\G_{M,N}(-;\tau):\Hd^1\to\OH$ is a homomorphism. For any $\mathbf{h}\in\Hd^2\subset\Hd^1$, the limit exists and converges uniformly. Therefore, $\G(-;\tau)$ is a homomorphism.
\end{proof}
\begin{proposition}
    For any $\mathbf{h}\in\Hd^2$, we have
    \begin{align*}
        \G(\mathbf{h};\tau+1)=\G(\mathbf{h};\tau).
    \end{align*}
\end{proposition}
\begin{proof}
    Through the linearization of Young tableaux, the Schur multiple Eisenstein series can be expressed as a linear combination of multiple Eisenstein series. The result follows from the corresponding identity for multiple Eisenstein series (see \cite{B})
\(
\G(h_1,\dots,h_r;\tau+1)=\G(h_1,\dots,h_r;\tau).
\)
\end{proof}
This implies that Schur multiple Eisenstein series admit Fourier expansions. 
In fact, we can write the Fourier expansion of Schur multiple Eisenstein series in the following way:
\begin{equation}\label{eq:Gzetag}
    \G(-;\tau)=m\circ\left(\zeta(-)\otimes\check{g}(-;\tau)\right)\circ\Delta_{cut}.  
\end{equation}
A detailed study of these Fourier expansions will be presented in a coming paper \cite{Yu}. 

As an application of the Jacobi--Trudi formula (\cref{Thm:JTf-homo&LdJ}), since $\G(-;\tau)$ is an algebra homomorphism, we obtain a determinant expression for the Schur multiple Eisenstein series in terms of multiple Eisenstein series.

\begin{corollary}[Jacobi--Trudi formula]
     Assume that $\mathbf{h}=(h_{i,j})_{(i,j)\in D(\Sh(\mathbf{h}))}\in\Hd^{2,\operatorname{diag}}$ is a diagonal constant Young tableau. Let $\Sh(\mathbf{h})=\lambda/\mu$, where $\lambda=(\lambda_1,\dots,\lambda_r)$ and $\mu=(\mu_1,\dots,\mu_r)$. Let $\lambda'=(\lambda'_1,\dots,\lambda_s')$ and $\mu'=(\mu_1',\dots,\mu_s')$ be the conjugates of $\lambda$ and $\mu$. We allow zero parts in $\lambda'$ and $\mu'$. Then we have 
    \begin{align}
        \mathbb{G}(\mathbf{h};\tau)=\det\left[\mathbb{G}(a_{-\mu'_j+j-1},\dots, a_{-\mu'_j+j-(\lambda'_i-\mu'_j-i+j)};\tau)\right]_{1\leq i,j\leq s}.
        \label{eq:GJTf}
    \end{align}
    Here, $\mathbb{G}(a_{-\mu'_j+j-1},\dots, a_{-\mu'_j+j-(\lambda'_i-\mu'_j-i+j)};\tau):=1$ if $\lambda'_i-\mu'_j-i+j=0$ and $0$ if $\lambda'_i-\mu'_j-i+j<0$. 
    \label{cor:Jacobi-Trudi}
\end{corollary}
We now consider the special case where all entries in the Young tableau are identical. Suppose every entry equals $h \ge 2$. \cref{cor:Jacobi-Trudi} expresses this Schur multiple Eisenstein series as a polynomial in standard multiple Eisenstein series of the form $\G(h, \dots, h; \tau)$. To evaluate these specific terms, we use the following identity.
\begin{corollary}For $h\geq2$, we have
    \begin{align*}
        \exp\left(\sum_{i\geq1}\frac{(-1)^{i-1}}{i}\G(ih;\tau)X^i\right)=\sum_{n=0}^{\infty}\G(\underbrace{h,\dots,h}_n;\tau)X^n.
    \end{align*}
    \label{expresult}
\end{corollary}
\begin{proof}
Apply the algebra homomorphism $\G(-;\tau)$ to \eqref{eq:newtonYT-exp} with $a=h$. Since the target algebra is $\C$, the $\ast$-exponential is sent to the usual exponential.
\end{proof}
\cref{expresult} is an application of \cref{cor:expYTresult}. It rewrites the series $\G(h, \dots, h; \tau)$ as a polynomial in depth-$1$ Eisenstein series $\G(ih; \tau)$. 

Together, these two corollaries provide a complete reduction. First, \cref{cor:Jacobi-Trudi} reduces a constant-entry Schur multiple Eisenstein series to a polynomial in $\G(h, \dots, h; \tau)$. Then, \cref{expresult} reduces those terms further to polynomials in $\G(ih; \tau)$. Since standard depth-$1$ Eisenstein series have well-known Fourier expansions, this reduction explicitly determines the Fourier expansion of the original Schur multiple Eisenstein series. To see how this reduction works, we consider the following example:
\begin{align*}
    &\G(2;\tau)=\G\left(\ \ytableausetup{centertableaux, boxsize=1em}{ \begin{ytableau}
2 
\end{ytableau}}\ ;\tau\right),\,\G(4;\tau)=\G\left(\ \ytableausetup{centertableaux, boxsize=1em}{\begin{ytableau}
2 & 2
\end{ytableau}}\ ;\tau\right)-\G\left(\ \ytableausetup{centertableaux, boxsize=1em}{ \begin{ytableau}
2  \\
2 
\end{ytableau}}\ ;\tau\right),\\ 
&\G(6;\tau)=\G\left(\ \ytableausetup{centertableaux, boxsize=1em}{ \begin{ytableau}
2&2 & 2
\end{ytableau}}\ ;\tau\right)-\G\left(\ \ytableausetup{centertableaux, boxsize=1em}{\begin{ytableau}
2 & 2 \\
2 & \none
\end{ytableau}};\tau\right)+\G\left(\ \ytableausetup{centertableaux, boxsize=1em}{\begin{ytableau}
2  \\
2 \\
2
\end{ytableau}}\ ;\tau\right).
\end{align*}
These formulas show that the relation is invertible. Thus $\G(2;\tau)$, $\G(4;\tau)$, and $\G(6;\tau)$ can be written as $\Z$-linear combinations of Schur multiple Eisenstein series with all entries equal to $2$. Since they generate the ring of quasimodular forms, this means that all quasimodular forms can be expressed using these Schur multiple Eisenstein series. This idea extends to the following modularity results:
\begin{corollary}[Modularity]
    Let $\mathbf{h}=\{h\}^{\lambda/\mu}\in\Hd^2$. Then
        \begin{enumerate}[(i)]
            \item $\G(\mathbf{h};\tau)\in\Q[\G(hl;\tau)|l\geq1]$.
            \item If $h=2$, $\G(\mathbf{h};\tau)$ is a quasimodular form, and \emph{every} quasimodular form can be written as a linear combination of these $\G(\mathbf{h};\tau)$.
            \item For even $h\geq4$, $\G(\mathbf{h};\tau)$ is a modular form.
        \end{enumerate}
\end{corollary}
\begin{proof}
Part (i) follows by applying the algebra homomorphism $\G(-;\tau)$ to \cref{Thm:YTpoly} with $a=h$. Assume $h=2$. Then part (i) shows that $\G(\mathbf h;\tau)$ is a polynomial in the Eisenstein series $\G(2l;\tau)$ $(l\geq1)$. Hence $\G(\mathbf h;\tau)$ is quasimodular. Conversely, the formulas above show that $\G(2;\tau)$, $\G(4;\tau)$, and $\G(6;\tau)$ are in the $\Q$-span of Schur multiple Eisenstein series with all entries equal to $2$. This span is closed under products, because $\G(-;\tau)$ is an algebra homomorphism and the constant-entry tableaux are closed under products by \cref{Thm:iota}. Therefore this span contains $\Q[\G(2;\tau),\G(4;\tau),\G(6;\tau)]$. This is the ring of quasimodular forms for $\mathrm{SL}_2(\Z)$. This proves (ii).

Assume that $h\geq4$ is even. By part (i), $\G(\mathbf h;\tau)$ is a polynomial in the Eisenstein series $\G(hl;\tau)$ $(l\geq1)$. Each $\G(hl;\tau)$ is a modular form. Moreover, the polynomial is homogeneous of weight $h|\lambda/\mu|$. Hence $\G(\mathbf h;\tau)$ is a modular form. This proves (iii).
\end{proof}

\subsection{Schur multiple divisor-sums and \texorpdfstring{$q$}{q}-analogue of Schur multiple zeta values}\label{sec:examSMDS&qanalog}
Divisor sums $\sigma_r(n)=\sum_{d\mid n}d^r$ are important objects in number theory. They also appear in the Fourier expansion of Eisenstein series. We consider a Schur-type sum of the divisors: for partitions $\mu\subsetneq\lambda$, and $\mathbf{h}\in\YT(\lambda/\mu;\Z_{\geq1})$, we define
\begin{align*}
    \sigma_{\mathbf{h}}(n)=\sum_{\substack{\sum_{(i,j)\in D(\lambda/\mu)}u_{i,j}v_{i,j}=n,\\(u_{i,j})\in\SSYT(\lambda/\mu;\N),\\ (v_{i,j})\in\N^{D(\lambda/\mu)}}}\prod_{(i,j)\in D(\lambda/\mu)}v_{i,j}^{h_{i,j}-1}.
\end{align*}
The motivation to study them comes from the Fourier expansion of the Schur multiple Eisenstein series. In particular, we consider their generating functions, which we denote by
\begin{equation}\label{eq:gBbbk}
    g(\mathbf{h}):=\prod_{(i,j)\in D(\Sh(\mathbf{h}))}\frac{1}{(h_{i,j}-1)!}\sum_{n>0}\sigma_{\mathbf{h}}(n)q^n\in\Q[[q]]. 
\end{equation}
For $\lambda=\mu$ we set $g(\mathbf 1):=1$, in accordance with \cref{def:MSS}. These $q$-series also appear in the Fourier expansion of the Schur multiple Eisenstein series in a different way than \eqref{eq:Gzetag} in \cite{Yu}. Throughout this subsection, $\diamond$ does not denote $+$, but the product given in \eqref{eq:BKdia} below. The argument in \cite{BK} gives the following proposition.
\begin{proposition}
    For any partitions $\mu\subsetneq\lambda$ and any index $\mathbf{h}\in\mathcal{H}^0$, we have
    \begin{align*}
g(\mathbf{h})=\sum_{(u_{i,j})\in\SSYT(\lambda/\mu;\N)}\prod_{(i,j)\in D(\lambda/\mu)}\frac{P_{h_{i,j}}(q^{u_{i,j}})}{(h_{i,j}-1)!\,\left(1-q^{u_{i,j}}\right)^{h_{i,j}}},
    \end{align*}
    where $P_s(X)$ is the $s$-th Eulerian polynomial, given by
     \begin{align*}
        \frac{P_s(X)}{(1-X)^s}=\sum_{n\geq1}n^{s-1}X^n.
    \end{align*}
\end{proposition}
Notice that this is a well-defined $q$-series for any index since $P_h(X)\in X\Q[X]$. These series can be seen as modified $q$-analogues of Schur multiple zeta values:
\begin{proposition}
    For any index $\mathbf{h}\in\mathcal{H}^0$, we have
    \begin{align*}
        \lim_{q\to1^-}(1-q)^{\sum_{(i,j)\in D(\Sh(\mathbf{h}))}h_{i,j}}g(\mathbf{h})=\zeta(\mathbf{h}).
    \end{align*}
\end{proposition}
The proof is the same as in \cite{BK}. The summation and the limit may be interchanged. This follows from the uniform convergence of the inner sum for $|q|<1$. Then for $h>1$, using the Eulerian polynomial identity, we have
\begin{align*}
    \lim_{q\to1^-}(1-q)^h\frac{P_h(q^m)}{(h-1)!(1-q^m)^h}=\frac{P_h(1)}{(h-1)!\,m^h}=\frac{1}{m^h}.
\end{align*}
For the $q$-series $g(\mathbf{h})$, we use the following algebra setup.
Define the product $\diamond$ on $\Q\N$ by
\begin{align}\label{eq:BKdia}
    {h_1}\diamond {h_2}=({h_1+h_2})+\sum_{j=1}^{h_1+h_2-1}\left(\lambda^j_{h_1,h_2}+\lambda^j_{h_2,h_1}\right)\,j,\quad(j\in\Q\N)
\end{align}
where the rational coefficients $\lambda^j_{h_1,h_2}$ are given by
\begin{align*}
    \lambda^j_{h_1,h_2}=(-1)^{h_2-1}\binom{h_1+h_2-1-j}{h_1-j}\frac{B_{h_1+h_2-j}}{(h_1+h_2-j)!}.
\end{align*}
Here, $B_h$ is the Bernoulli number, with the convention $B_1=-\tfrac12$. By \cite{BK}, each of the following maps $f_m$ $(m\geq1)$ is a $\diamond$-algebra homomorphism for the product \eqref{eq:BKdia}, and $g(\mathbf{h})$ is the multiple Schur series associated with the family $\{f_m\}_{m\geq1}$: 
\begin{align*}
    f_m:\K\N&\longrightarrow \Q[[q]]\\
    h&\longmapsto \frac{1}{(h-1)!}\frac{P_h(q^m)}{(1-q^m)^h}.
\end{align*}
Hence, by \cref{thm:MSS-character}, each truncation $F_{\Z^+_M}$ vanishes on $\ker L_\diamond$; since every coefficient of $g(\mathbf{h})$ agrees with that of $F_{\Z^+_M}(\mathbf{h})$ for large $M$, the map $g$ induces a $\Q$-algebra homomorphism $g:\Hd^1\to\Q[[q]]$.
One special case of this $q$-series $g(\mathbf{h})$ is MacMahon's series. Take $\lambda=(\{1\}^r)$, $\mu=\varnothing$, and $h_{i,j}=2$ for any $(i,j)\in D(\lambda)$. Then we have
\begin{align*}
    g\left({
\ytableausetup{boxsize=1em} 
\begin{ytableau}
2 \\
2 \\
\none[\scalebox{0.5}{$\vdots$}]\\
2
\end{ytableau}
}\right)=\sum_{0<m_1<\cdots<m_r}\frac{q^{m_1+\cdots+m_r}}{(1-q^{m_1})^2\cdots(1-q^{m_r})^2}, \quad(r\geq1).
\end{align*}
In the forthcoming work \cite{BY}, we study the case where all entries are equal to $2$. In this case, these $q$-series are interesting objects in their own right. They are quasimodular forms, and they span the ring of all quasimodular forms with $\Q$ coefficients. 
Denoting $g_{\lambda}=g(\{2\}^{\lambda})$, we can write $\Delta(q)$ as follows:
\begin{align*}
\Delta(q) &= q \prod_{n=1}^{\infty} (1-q^n)^{24}\\
&= g_{(1)} - 27 g_{(2)} + 383 g_{(1,1)} - 72744 g_{(1,1,1)} - 2250 g_{(4)} + 22410 g_{(3,1)} - 40536 g_{(2,2)} \\
& - 476298 g_{(2,1,1)} + 1399626 g_{(1^4)} + 98280 g_{(3,2)} - 181440 g_{(3,1,1)} - 2996784 g_{(2,2,1)} \\
& + 6106968 g_{(2,1,1,1)} - 8846712 g_{(1,1,1,1,1)} + 22680 g_{(6)} - 294840 g_{(3,3)} + 272160 g_{(3,1,1,1)} \\
& - 17889984 g_{(2,2,2)} + 17912664 g_{(2,2,1,1)} - 18184824 g_{(2,1,1,1,1)} + 17889984 g_{(1,1,1,1,1,1)}.
\end{align*}

\subsection{Schur multiple \texorpdfstring{$L$}{L}-values}\label{sec:SMLV}
Let $\chi$ be a Dirichlet character modulo $N$ and $\mathbf{h}\in\YT(\lambda/\mu;\Z_{\geq1})$. Assume that $h_{i,j}>1$ at every corner. We define the \emph{Schur multiple Dirichlet $L$-value} by
\begin{align*}
    L_\chi(\mathbf{h})
    =\sum_{(m_{i,j})\in\SSYT(\lambda/\mu,\N)}\prod_{(i,j)\in D(\lambda/\mu)}\frac{\chi(m_{i,j})}{m_{i,j}^{h_{i,j}}}.
\end{align*}
Since $|\chi(m)|\leq1$, this sum converges absolutely under the same condition as the Schur multiple zeta values. A single box gives the Dirichlet $L$-value
\ytableausetup{centertableaux,boxsize=1.1em}
\begin{align*}
    L_\chi\!\left(\,\begin{ytableau} h \end{ytableau}\,\right)=L(h,\chi)=\sum_{m\geq1}\frac{\chi(m)}{m^h},
\end{align*}
so $L_\chi$ is the Schur version of the Dirichlet $L$-values, and it is related to the multiple Dirichlet $L$-values studied in \cite{Yam}.

This is a multiple Schur series. We will compare the two representations of Schur multiple Dirichlet $L$-values as multiple Schur series by selecting two different sets $\A$ and corresponding maps $\{f_m\}_{m\in\X}$. 

Take $\A_1=\N$ and $\calP=\C$, and for $m\in\Z^+_M:=\{1,\dots,M\}$ define
\begin{align*}
    f_m:\K\N&\longrightarrow\C\\
    h&\longmapsto\frac{\chi(m)}{m^h}.
\end{align*}
Then $L_\chi(\mathbf{h})=\lim_{M\to\infty}F_{\Z^+_M}(\mathbf{h})$, the limit of the multiple Schur series of the family $\{f_m\}$.

It is natural to ask when this series comes from the linearization $L_\diamond$. By \cref{thm:MSS-character}, this happens when $f_m$ is a $\diamond$-algebra homomorphism. With $\diamond=+$,  
\begin{align*}
    f_m(h_1+h_2)=f_m(h_1)f_m(h_2)\Leftrightarrow\frac{\chi(m)}{m^{h_1+h_2}}=\frac{\chi(m)^2}{m^{h_1+h_2}},
\end{align*}
that is, $\chi(m)^2=\chi(m)$ for all $m$. This holds when $\chi$ is the principal character. Thus, for a principal character $\chi$ and $\diamond=+$, each $f_m$ is a homomorphism. The map $L_\chi$ factors through $\Yd$ and comes from $L_\diamond$, as in the case of the Schur multiple zeta values.

For a non-principal character $\chi$, the situation is different. Since $\chi(m)^2\neq\chi(m)$ for some $m$, the family $\{f_m\}_{m\in\Z^+_M}$ is not a family of algebra homomorphisms for $\diamond=+$.

In fact, no choice of $\diamond$ makes all maps $\{f_m\}_{m\in\Z^+_M}$ algebra homomorphisms in this algebraic setup. For arbitrary $\diamond$, write $h_1\diamond h_2=\sum_{h}c_h\,h$ with finitely many nonzero $c_h\in\Q$. Then $f_m(h_1\diamond h_2)=\chi(m)\sum_h c_h m^{-h}$, while $f_m(h_1)f_m(h_2)=\chi(m)^2 m^{-(h_1+h_2)}$. Fix a residue class $b$ modulo $N$ with $\chi(b)\neq0$. If the equality $f_m(h_1\diamond h_2)=f_m(h_1)f_m(h_2)$ held for all $m\equiv b\pmod N$, then the two polynomials in $m^{-1}$ would agree at infinitely many values of $m$, hence coefficientwise, forcing $c_{h_1+h_2}=\chi(b)$ and $c_h=0$ otherwise. Since the $c_h$ do not depend on $b$, the character $\chi$ would be constant on all residue classes with $\chi(b)\neq0$, so $\chi$ would be principal. The product $\diamond$ acts only on the entries. The obstruction comes from the factor $\chi(m)$, which depends on the summation variable $m$. This dependence cannot be absorbed into a product on $\K\A$. Hence $L_\chi$ does not factor through $\Yd$, and it is not the image of any quasi-shuffle element under $L_\diamond$. It can only be defined by the semi-standard sum. This is the kind of multiple Schur series that is really different from the values built from columns through $L_\diamond$.

Therefore, an interesting point arises. The previous result implies that, for a non-principal character, the Schur multiple Dirichlet $L$-values cannot be obtained from the column values through this quasi-shuffle linearization. This suggests that their $\Q$-span may be larger than the $\Q$-span generated by multiple Dirichlet $L$-values of the same character. We illustrate this phenomenon in the following.

Let $\chi=\chi_3$ be the quadratic Dirichlet character modulo $3$. By a column Schur multiple Dirichlet $L$-value we mean an admissible fixed-character multiple Dirichlet $L$-value in \cite{AK,Yam}, namely a value of the form
\begin{equation*}
    L_{\chi}\!\left(
    \begin{ytableau}
        h_1\\
        \vdots\\
        h_r
    \end{ytableau}
    \right)
    =
    \sum_{0<m_1<\cdots<m_r}
    \frac{\chi(m_1)\cdots\chi(m_r)}
    {m_1^{h_1}\cdots m_r^{h_r}} .
\end{equation*}
Let $\mathcal C_{\chi,w}$ be the $\Q$-span of such column Schur multiple Dirichlet $L$-values of weight $w$, where the weight is $w:=h_1+\cdots+h_r$. Let $\mathcal S_{\chi,w}$ be the $\Q$-span of all admissible Schur multiple Dirichlet $L$-values of weight $w$.  Also let $\mathcal C_{\chi,w}^{\times}$ be the $\Q$-span of products of one or more column Schur multiple Dirichlet $L$-values of total weight $w$.

Numerical evidence suggests that the space $\mathcal C_\chi:=\sum_w \mathcal C_{\chi,w}$ is not closed under products. Already the product of two depth-one values of weight \(2\) produces, in weight \(4\), a principal-character diagonal contribution
\begin{equation*}
    L(2,\chi_3)^2
    =
    2L_{\chi_3}\!\left(
    \begin{ytableau}
        2\\
        2
    \end{ytableau}
    \right)
    +
    \sum_{\substack{m\geq 1\\ 3\nmid m}}\frac{1}{m^4}
    =
    2L_{\chi_3}\!\left(
    \begin{ytableau}
        2\\
        2
    \end{ytableau}
    \right)
    +
    \frac{80}{81}\zeta(4)\in\mathcal{C}_{\chi_3,4}^{\times}.
\end{equation*}
Equivalently,
\begin{equation*}
    L(2,\chi_3)^2
    =
    L_{\chi_3}\!\left(
    \begin{ytableau}
        2&2
    \end{ytableau}
    \right)
    +
    L_{\chi_3}\!\left(
    \begin{ytableau}
        2\\
        2
    \end{ytableau}
    \right)\in \mathcal{S}_{\chi_3,4}.
\end{equation*}
Thus the product of two column Schur multiple Dirichlet $L$-values naturally lies in $\mathcal{S}_{\chi,w}$, but it contains a new contribution coming from $\chi_3(m)^2=\chi_0(m)$, where $\chi_0$ is the principal character modulo $3$. Numerically, this contribution is not absorbed by $\mathcal{C}_{\chi,w}$.

On the other hand, $\mathcal{S}_{\chi}:=\sum_w\mathcal{S}_{\chi,w}$ is an algebra. For admissible tableaux $\mathbf h,\mathbf g\in\Y$ we have
\begin{equation*}
    L_{\chi}(\mathbf h)L_{\chi}(\mathbf g)
    =
    L_{\chi}(\mathbf h\ast\mathbf g).
\end{equation*}
Hence products of Schur multiple Dirichlet $L$-values are again Schur multiple Dirichlet $L$-values, in the sense of the Young tableaux algebra.

Since every admissible column value has weight at least $2$, a product of two or more factors has weight at least $4$; hence $\mathcal{C}_{\chi_3,3}^{\times}=\mathcal{C}_{\chi_3,3}$. In weight $3$, numerical computations suggest that 
\[\mathcal{C}_{\chi_3,3}=\mathcal{C}_{\chi_3,3}^{\times}\subsetneq\mathcal{S}_{\chi_3,3}.\]
The admissible Schur multiple Dirichlet $L$-values in weight $3$ are
\begin{equation*}
    L_{\chi_3}\!\left(
    \begin{ytableau}
        3
    \end{ytableau}
    \right),
    \qquad
    L_{\chi_3}\!\left(
    \begin{ytableau}
        1\\
        2
    \end{ytableau}
    \right),
    \qquad
    L_{\chi_3}\!\left(
    \begin{ytableau}
        1&2
    \end{ytableau}
    \right).
\end{equation*}
The row and the column differ by the principal character term
\begin{equation*}
    L_{\chi_3}\!\left(
    \begin{ytableau}
        1&2
    \end{ytableau}
    \right)
    -
    L_{\chi_3}\!\left(
    \begin{ytableau}
        1\\
        2
    \end{ytableau}
    \right)
    =
    \sum_{\substack{m\geq 1\\ 3\nmid m}}\frac{1}{m^3}
    =
    \frac{26}{27}\zeta(3).
\end{equation*}
Therefore $ L_{\chi_3}\!\left(
    \begin{ytableau}
        1&2
    \end{ytableau}
    \right)$ belongs to $\mathcal C_{\chi_3,3}$ if and only if $\zeta(3)$ belongs to $\operatorname{span}_{\Q}\{L(3,\chi_3),L_{\chi_3}\!\left( \begin{ytableau}
        1\\
        2
    \end{ytableau}\right)\}$. A \textsc{pslq} search finds \emph{no} such relation. Thus, at this level of evidence, $\mathcal{C}_{\chi_3,3}^{\times}\subsetneq \mathcal{S}_{\chi_3,3}$. 
    
    Thus the computations suggest that, in this algebra setup, the $\Q$-space spanned by Schur multiple Dirichlet $L$-values is larger than the $\Q$-space spanned by multiple Dirichlet $L$-values. This is different from the cases of Schur multiple zeta values and of Schur multiple Eisenstein series.

    As mentioned before, for any multiple Schur series, we can always find a ``nice'' algebraic setup to ensure it can be linearized to the quasi-shuffle algebra by $L_{\diamond}$. Here is the second algebraic setup for Schur multiple Dirichlet $L$-values.

    Let $\A_2=\N\times \N$ and $\calP=\C$. For $m\in\Z_M^+$, we set 
    \begin{equation*}
    \begin{aligned}
         f_m': \K(\N\times \N)& \longrightarrow \C\\
        (h,a)&\longmapsto\frac{\chi^a(m)}{m^h}.
    \end{aligned}
    \end{equation*}
    The $\diamond$ in $\K(\N\times \N)$ is defined by
    \begin{equation*}
        (a,b)\diamond(c,d)=(a+c,b+d),\qquad a,b,c,d\in\N.
    \end{equation*}
    Then the Schur multiple Dirichlet $L$-value associated with $\{f_m'\}_{m\in\X}$ is given by $L'_{\chi}(\mathbf{h})=\lim_{M\to\infty}F_{\Z_M^+}(\mathbf{h})$.
    One can easily check that for any $\mathbf{h}=\{h_{i,j}\}\in \K\boxed{\A_1}$, letting $\mathbf{h}'=\{(h_{i,j},1)\}\in\K\boxed{\A_2}$, we have
    \begin{equation*}
        L_{\chi}(\mathbf{h})=L'_{\chi}(\mathbf{h}').
    \end{equation*}
    In this algebraic setup, $\{f_m'\}_{m\in\X}$ is a family of algebra homomorphisms. Hence, by \cref{thm:MSS-character} applied to each truncation and passing to the limit, as for the Schur multiple zeta values, the Schur multiple Dirichlet $L$-value $L'_{\chi}$ is an algebra homomorphism on the subalgebra of $\K\boxed{\N\times \N}_{\diamond}$ generated by shifted connected tableaux whose corner entries $(h,a)$ satisfy $h\geq2$. This condition guarantees absolute convergence, since $|\chi^a(m)|\leq1$. One needs to note that, in this case, the ``column'' values are not only the multiple Dirichlet $L$-values. For example, 
    \ytableausetup{boxsize=1.35em}
    \begin{align*}
       L_{\chi}\!\left(
    \begin{ytableau}
        1&2
    \end{ytableau}
    \right)=\!\sum_{0<m_1\le m_2}\!\frac{\chi(m_1)\chi(m_2)}{m_1 m_2^2}
    =\sum_{0<m}\frac{\chi^2(m)}{m^3}+\!\sum_{0<m_1< m_2}\!\frac{\chi(m_1)\chi(m_2)}{m_1 m_2^2}
    =L'_{\chi}\left(\begin{ytableau}
        \scriptstyle (3,2)
    \end{ytableau}\right)+L'_{\chi}\left(\begin{ytableau}
       \scriptstyle  (1,1)\\
       \scriptstyle (2,1)
    \end{ytableau}\right).
    \end{align*}

\end{document}